\documentclass[twoside]{irmaems}
\usepackage{amssymb}
\usepackage{amsmath} %use \usepackage{amsmath,cmmib57}
\usepackage{latexsym}

\usepackage{psfrag}
\usepackage[usenames,dvips]{color}
\usepackage{multicol}

\usepackage{graphicx}
\usepackage[all]{xy}

\makeatletter
\newenvironment{figurehere}
  {\def\@captype{figure}}
  {}
\makeatother

\renewcommand\theequation{\arabic{section}.\arabic{equation}}

\theoremstyle{definition} %%% for statements in roman typeface

 \newtheorem{definition}{Definition}[section]
 \newtheorem{remark}[definition]{Remark}
 \newtheorem{example}[definition]{Example}

\newtheorem*{notation}{Notation}  %%% for statements without numbering

\theoremstyle{plain}      %%% for statements in italic typeface

 \newtheorem{proposition}[definition]{Proposition}
 \newtheorem{theorem}[definition]{Theorem}
 
 \newtheorem{lemma}[definition]{Lemma}

\newcommand{\C}{\mathbb{C}}
\newcommand{\Z}{\mathbb{Z}}
\newcommand{\Q}{\mathbb{Q}}
\newcommand{\R}{\mathbb{R}}
\newcommand{\N}{\mathbb{N}}

\def\M{\mathcal{M}}
\def\Mcomb{\M^{comb}}
\def\Mbar{\overline{\M}}
\def\Mbarcomb{\Mbar^{comb}}
\def\W{W}
\def\Wbar{\ol{\W}}
\def\GG{\mathbb{G}}
\def\Cc{\mathcal{C}}

\def\Teich{\mathcal{T}}

\def\pa{\partial}

\def\Aut{\mathrm{Aut}}

\def\k{\kappa}
\def\rar{\rightarrow}
\def\arr#1#2{\stackrel{#1}{#2}}
\def\hra{\hookrightarrow}
\def\a{\alpha}
\def\Af{\mathfrak{A}}

\def\b{\beta}
\def\Ao{{\Af^{\circ}}}
\def\A8{{\Af^{\infty}}}

\def\l{\lambda}

\def\i{\iota}
\def\g{\gamma}
\def\G{\Gamma}

\def\D{\Delta}
\def\d{\delta}
\def\lra{\longrightarrow}

\def\s{\sigma}
\def\Si{\Sigma}
\def\e{\varepsilon}
\def\la{\langle}
\def\ra{\rangle}

\def\ul#1{\underline{#1}}
\def\ol#1{\overline{#1}}
\def\wh#1{\widehat{#1}}
\def\dis{\displaystyle}

\def\ora#1{\overrightarrow{#1}}
\def\ola#1{\overleftarrow{#1}}
\def\bold#1{\mbox{\boldmath$#1$}}
\def\ua{\ul{\bold{\a}}}
\def\ub{\ul{\bold{\b}}}
\def\ug{\ul{\bold{\g}}}

\def\up{\ul{p}}

\def\gr8{{\mathrm{gr}}_{\infty}}

\begin{document}

\title{Riemann surfaces, ribbon graphs \\ and combinatorial classes}

\author{Gabriele Mondello}
%\thanks{
%Work partially supported by SNF Grant No.~yy-63821.xx} and  Second author\thanks{Work partially
%supported by SNF Grant No.~xx-65213.yy}}
%
\address{
Department of Mathematics, Massachusetts Institute of Technology\\
77 Massachusetts Avenue, Cambridge MA 02139 USA\\
e-mail:\,\tt{gabriele@math.mit.edu}
}

\maketitle

\begin{abstract}
We begin by describing the duality between arc systems
and ribbon graphs embedded in a punctured surface
and explaining how to cellularize
the moduli space of curves in two different ways:
using Jenkins-Strebel differentials and using
hyperbolic geometry. We also briefly discuss
how these two methods are related.
Next, we recall the definition of Witten cycles
and we illustrate their connection with tautological
classes and Weil-Petersson geometry. Finally,
we exhibit a simple argument to prove that Witten classes
are stable.
\end{abstract}



\setcounter{tocdepth}{2}
\tableofcontents   

%%%%%%%%%%%%%%%%%%%%%%%%%%%%%%%%%%%%%%%%%%%%%%%%%%%%%%%%%%%%%%%%%%%%
\section{Introduction}\label{s-intro}
%%%%%%%%%%%%%%%%%%%%%%%%%%%%%%%%%%%%%%%%%%%%%%%%%%%%%%%%%%%%%%%%%%%%
%%%%%%%%%%%%%%%%%%%%%%%%%%%%%%%%%%%%%%%%%%%%%%%%%%%%%%%%%%%%%%%%%%
\subsection{Overview}
%%%%%%%%%%%%%%%%%%%%%%%%%%%%%%%%%%%%%%%%%%%%%%%%%%%%%%%%%%%%%%%%%%
\subsubsection{Moduli space and Teichm\"uller space.}
%%%%%%%%%%%%%%%%%%%%%%%%%%%%%%%%%%%%%%%%%%%%%%%%%%%%%%%%%%%%%%%%%%
Consider a compact oriented surface $S$ of genus $g$ together
with a finite subset $X=\{x_1,\dots,x_n\}$, such that $2g-2+n>0$.

The moduli space $\M_{g,X}$ is the set of all $X$-pointed
Riemann surfaces of genus $g$ up to isomorphism. Its universal
cover can be identified with
the Teichm\"uller space $\Teich(S,X)$, which
parametrizes complex structures on $S$ up to isotopy (relative
to $X$); equivalently, $\Teich(S,X)$ parametrizes
isomorphism classes of $(S,X)$-marked Riemann surfaces.
Thus, $\M_{g,X}$ is the quotient of $\Teich(S,X)$
under the action of the mapping class group
$\G(S,X)=\mathrm{Diff}_+(S,X)/\mathrm{Diff}_0(S,X)$.

As $\Teich(S,X)$ is contractible (Teichm\"uller \cite{teichmueller:collected}),
we also have that $\M_{g,X}\simeq B\G(S,X)$.
However, $\G(S,X)$ acts on $\Teich(S,X)$ discontinuously
but with finite stabilizers. Thus, $\M_{g,X}$ is naturally
an orbifold and $\M_{g,X}\simeq B\G(S,X)$ must be
intended in the orbifold category.

%%%%%%%%%%%%%%%%%%%%%%%%%%%%%%%%%%%%%%%%%%%%%%%%%%%%%%%%%%%%%%%%%%%
\subsubsection{Algebro-geometric point of view.}
%%%%%%%%%%%%%%%%%%%%%%%%%%%%%%%%%%%%%%%%%%%%%%%%%%%%%%%%%%%%%%%%%%%
As compact Riemann surfaces are complex algebraic curves,
$\M_{g,X}$ has an algebraic structure and is in fact
a Deligne-Mumford stack, which is the algebraic analogue
of an orbifold. The underlying space $M_{g,X}$ (forgetting the
isotropy groups) is a quasi-projective variety.

The interest for enumerative geometry of algebraic curves
naturally led to seeking for a suitable compactification of
$\M_{g,X}$. Deligne and Mumford \cite{deligne-mumford:irreducibility}
understood that it was sufficient to consider algebraic curves
with mild singularities to compactify $\M_{g,X}$. In fact,
their compactification $\Mbar_{g,X}$ is the moduli space of
$X$-pointed stable (algebraic) curves of genus $g$,
where a complex
projective curve $C$ is ``stable'' if its only singularities are
nodes (that is, in local analytic coordinates $C$ looks
like $\{(x,y)\in \C^2\,|\,xy=0\}$) and every irreducible component
of the smooth locus of $C\setminus X$ has negative Euler characteristic.

The main tool to prove the
completeness of $\Mbar_{g,X}$
is the stable reduction theorem, which essentially
says that a smooth holomorphic family $\mathcal{C}^*\rar \Delta^*$
of $X$-pointed Riemann surfaces of genus $g$ over the pointed
disc can be completed to a family over $\Delta$ (after
a suitable change of base $z\mapsto z^k$) using a stable curve.

The beauty of $\Mbar_{g,X}$ is that it is smooth (as an orbifold)
and that its coarse space $\ol{M}_{g,X}$ is a projective variety
(Mumford \cite{mumford:stability},
Gieseker \cite{gieseker:lectures}, Knudsen \cite{knudsen:projectivity2}
\cite{knudsen:projectivity3}, Koll\'ar \cite{kollar:projectivity}
and Cornalba \cite{cornalba:projectivity}).

%%%%%%%%%%%%%%%%%%%%%%%%%%%%%%%%%%%%%%%%%%%%%%%%%%%%%%%%%%%%%%%%%%%
\subsubsection{Tautological maps.}
%%%%%%%%%%%%%%%%%%%%%%%%%%%%%%%%%%%%%%%%%%%%%%%%%%%%%%%%%%%%%%%%%%%
The map $\Mbar_{g,X\cup\{y\}}\rar \Mbar_{g,X}$ that forgets
the $y$-point can be identified to the universal
curve over $\Mbar_{g,X}$ and
is the first example of tautological map.

Moreover, $\Mbar_{g,X}$ has a natural algebraic stratification,
in which each stratum corresponds to a topological type of
curve: for instance,
smooth curves correspond to the open stratum $\M_{g,X}$.
As another example:
irreducible curves with one node correspond to an irreducible
locally closed subvariety of (complex) codimension $1$,
which is the image of the (generically $2:1$)
tautological boundary
map $\M_{g-1,X\cup\{y_1,y_2\}}\rar \Mbar_{g,X}$ that glues
$y_1$ to $y_2$. Thus, every stratum is the image of
a (finite-to-one) tautological boundary map, and thus
is isomorphic to a finite quotient of a product of smaller
moduli spaces.

%%%%%%%%%%%%%%%%%%%%%%%%%%%%%%%%%%%%%%%%%%%%%%%%%%%%%%%%%%%%%%%%%%%
\subsubsection{Augmented Teichm\"uller space.}
%%%%%%%%%%%%%%%%%%%%%%%%%%%%%%%%%%%%%%%%%%%%%%%%%%%%%%%%%%%%%%%%%%%
Teichm\"uller theorists are more interested in compactifying
$\Teich(S,X)$ rather than $\M_{g,X}$. One of the most
popular way to do it is due to Thurston (see \cite{FLP:travaux}):
the boundary of $\Teich(S,X)$ is thus made of
projective measured laminations and it is homeomorphic to a sphere.

Clearly, there cannot be any clear link between a compactification
of $\Teich(S,X)$ and of $\M_{g,X}$, as the infinite
discrete group $\G(S,X)$ would not act
discontinuously on a compact boundary $\pa\Teich(S,X)$.

Thus, the $\G(S,X)$-equivariant bordification of $\Teich(S,X)$
whose quotient is $\M_{g,X}$ cannot be compact.
A way to understand it is to endow $\M_{g,X}$ (and $\Teich(S,X)$)
with the Weil-Petersson metric \cite{weil:onthemoduli} and
to show that its completion is exactly $\Mbar_{g,X}$
\cite{masur:extension}.
Hence, the Weil-Petersson completion $\ol{\Teich}(S,X)$
can be identified to the set of $(S,X)$-marked stable Riemann
surfaces.

Similarly to $\Mbar_{g,X}$, also $\ol{\Teich}(S,X)$
has a stratification by topological type and each stratum
is a (finite quotient of a) product of smaller Teichm\"uller spaces.

%%%%%%%%%%%%%%%%%%%%%%%%%%%%%%%%%%%%%%%%%%%%%%%%%%%%%%%%%%%%%%%%%%%%
\subsubsection{Tautological classes.}
%%%%%%%%%%%%%%%%%%%%%%%%%%%%%%%%%%%%%%%%%%%%%%%%%%%%%%%%%%%%%%%%%%%%
The moduli space $\Mbar_{g,X}$ comes equipped with natural
vector bundles: for instance, $\mathcal{L}_i$ is the holomorphic
line bundle whose fiber at $[C]$ is $T^{\vee}_{C,x_i}$.
Chern classes of these line bundles and their push-forward
through tautological maps generate the so-called tautological
classes (which can be seen in the Chow ring or in
cohomology). The $\k$ classes were
first defined by Mumford \cite{mumford:towards}
and Morita \cite{morita:surface} and then modified
(to make them behave better under tautological maps)
by Arbarello and Cornalba \cite{arbarello-cornalba:combinatorial}.
The $\psi$ classes were defined by E.Miller \cite{miller:homology}
and their importance was successively
rediscovered by Witten \cite{witten:intersection}.

The importance of the tautological classes is due to the following facts
(among others):
\begin{itemize}
\item
their geometric meaning appears quite clear
\item
they behave very naturally under the tautological maps
(see, for instance, \cite{arbarello-cornalba:combinatorial})
\item
they often occur in computations of enumerative geometry;
that is, Poincar\'e duals of interesting algebraic loci are
often tautological (see \cite{mumford:towards})
but not always (see \cite{graber-pandharipande:nontautological})!
\item
they are defined on $\Mbar_{g,X}$ for every $g$ and $X$
(provided $2g-2+|X|>0$),
and they generate
the stable cohomology ring over $\Q$ due to
Madsen-Weiss's solution \cite{madsen-weiss:mumford}
of Mumford's conjecture (see Section~\ref{ss:stability})
\item
there is a set of generators ($\psi$'s and $\k$'s) which have
non-negativity properties (see \cite{arakelov:families}
and \cite{mumford:towards})
\item
they are strictly related to the Weil-Petersson geometry
of $\Mbar_{g,X}$ (see \cite{wolpert:homology},
\cite{wolpert:positive}, \cite{wolpert:chern}
and \cite{mirzakhani:witten}).
\end{itemize}

%%%%%%%%%%%%%%%%%%%%%%%%%%%%%%%%%%%%%%%%%%%%%%%%%%%%%%%%%%%%%%%%%%%%
\subsubsection{Simplicial complexes associated to a surface.}
%%%%%%%%%%%%%%%%%%%%%%%%%%%%%%%%%%%%%%%%%%%%%%%%%%%%%%%%%%%%%%%%%%%%
One way to analyze the (co)homology of $\M_{g,X}$, and so
of $\G(S,X)$, is to construct a highly connected
simplicial complex on which $\G(S,X)$ acts.
This is usually achieved by considering complexes of
disjoint, pairwise non-homotopic simple closed curves
on $S\setminus X$ with suitable properties
(for instance, Harvey's complex of curves
\cite{harvey:geometric}).

If $X$ is nonempty (or if $S$ has boundary), then
one can construct a complex using systems of
homotopically nontrivial, disjoint 
arcs joining two (not necessarily distinct) points in $X$
(or in $\pa S$), thus obtaining
the arc complex $\Af(S,X)$ (see \cite{harer:virtual}).
It has an ``interior'' $\Ao(S,X)$ made of systems
of arcs that cut $S\setminus X$ in discs (or pointed discs)
and a complementary ``boundary'' $\A8(S,X)$.

An important result, which has many fathers
(Harer-Mumford-Thurston \cite{harer:virtual},
Penner \cite{penner:decorated}, Bowditch-Epstein
\cite{bowditch-epstein:natural}), says that
$|\Ao(S,X)|$ is $\G(S,X)$-equivariantly homeomorphic
to $\Teich(S,X)\times \Delta_X$
(where $\Delta_X$ is the standard simplex
in $\R^X$). Thus, we can transfer the cell
structure of $|\Ao(S,X)|$ to an (orbi)cell structure
on $\M_{g,X}\times\Delta_X$.

The homeomorphism is realized by coherently associating
a weighted system of arcs to every $X$-marked
Riemann surface, equipped
with a decoration $\up\in\Delta_X$.
There are two traditional ways
to do this: using the flat structure arising
from a Jenkins-Strebel quadratic differential
(Harer-Mumford-Thurston) with prescribed
residues at $X$
or using the hyperbolic metric coming from the
uniformization theorem (Penner and Bowditch-Epstein).
Quite recently, several other ways have been
introduced (see \cite{luo:decomposition},
\cite{luo:rigidity},
\cite{mondello:wp} and \cite{mondello:triang}).

%%%%%%%%%%%%%%%%%%%%%%%%%%%%%%%%%%%%%%%%%%%%%%%%%%%%%%%%%%%%%%%%%%%%
\subsubsection{Ribbon graphs.}
%%%%%%%%%%%%%%%%%%%%%%%%%%%%%%%%%%%%%%%%%%%%%%%%%%%%%%%%%%%%%%%%%%%%
To better understand the homeomorphism between
$|\Ao(S,X)|$ and $\Teich(S,X)\times\Delta_X$, it is often
convenient to adopt a dual point of view,
that is to think of weighted systems of arcs as 
of metrized graphs $\GG$,
embedded in $S\setminus X$ through a
homotopy equivalence.

This can be done by picking a vertex in each
disc cut by the system of arcs and joining
these vertices by adding an edge transverse to
each arc. What we obtain is an $(S,X)$-marked
metrized ribbon graph.
Thus, points in $|\Ao(S,X)|/\G(S,X)\cong\M_{g,x}\times\Delta_X$
correspond to metrized $X$-marked ribbon graphs of genus $g$.

This point of view is particularly useful to understand
singular surfaces (see also \cite{bowditch-epstein:natural},
\cite{kontsevich:intersection},
\cite{looijenga:cellular}, \cite{penner:boundary},
\cite{zvonkine:strebel}, \cite{acgh:II}
and \cite{mondello:triang}).
The object dual to a weighted
system of arcs in $\A8(S,X)$ is a collection of
data that we called an $(S,X)$-marked
``enriched'' ribbon graph. Notice that an
$X$-marked ``enriched'' metrized ribbon graph
does not carry
all the information needed to construct
a stable Riemann surface. Hence, the map
$\Mbar_{g,X}\times\Delta_X \rar |\Af(S,X)|/\G(S,X)$
is not a injective on the locus of singular curves, but still
it is a homeomorphism on a dense open subset.

%%%%%%%%%%%%%%%%%%%%%%%%%%%%%%%%%%%%%%%%%%%%%%%%%%%%%%%%%%%%%%%%%%%%
\subsubsection{Topological results.}
%%%%%%%%%%%%%%%%%%%%%%%%%%%%%%%%%%%%%%%%%%%%%%%%%%%%%%%%%%%%%%%%%%%%
The utility of the $\G(S,X)$-equivariant homotopy equivalence
$\Teich(S,X)\simeq |\Ao(S,X)|$ is the possibility of making
topological computations on $|\Ao(S,X)|$.
For instance, Harer \cite{harer:virtual} determined the virtual
cohomological dimension of $\G(S,X)$ (and so of $\M_{g,X}$)
using the high connectivity of $|\A8(S,X)|$ and he has established
that $\G(S,X)$ is a virtual duality group, by showing that
$|\A8(S,X)|$ is spherical.
An analysis of the singularities of $|\Af(S,X)|/\G(S,X)$
is in \cite{penner:arc}.

Successively, Harer-Zagier \cite{harer-zagier:euler} and Penner
\cite{penner:euler} have computed the orbifold Euler characteristic
of $\M_{g,X}$, where by ``orbifold'' we mean that a cell
with stabilizer $G$ has Euler characteristic $1/|G|$.
Because of the cellularization, the problem translates into
enumerating $X$-marked ribbon graphs of genus $g$ and
counting them with the correct sign.

Techniques for enumerating graphs and ribbon graphs
(see, for instance, \cite{biz:quantum}) have been
known to physicists for long time: they use asymptotic
expansions of Gaussian integrals over spaces of matrices.
The combinatorics of iterated integrations by parts
is responsible for the appearance of (ribbon) graphs (Wick's lemma).
Thus, the problem of computing $\chi^{orb}(\M_{g,X})$ can
be reduced to evaluating a matrix integral
(a quick solution is also given by Kontsevich
in Appendix~D of \cite{kontsevich:intersection}).

%%%%%%%%%%%%%%%%%%%%%%%%%%%%%%%%%%%%%%%%%%%%%%%%%%%%%%%%%%%%%%%%%%%%
\subsubsection{Intersection-theoretical results.}
%%%%%%%%%%%%%%%%%%%%%%%%%%%%%%%%%%%%%%%%%%%%%%%%%%%%%%%%%%%%%%%%%%%%
As $\M_{g,X}\times\Delta_X$ is not just homotopy equivalent to
$|\Ao(S,X)|/\G(S,X)$ but actually homeomorphic (through a
piecewise-linear real-analytic diffeomorphism), it is clear
that one can try to rephrase integrals over $\M_{g,X}$
as integrals over $|\Ao(S,X)|/\G(S,X)$, that is as
sums over maximal systems of arcs of integrals over a single
simplex.
This approach looked promising in order to compute Weil-Petersson
volumes (see Penner \cite{penner:WP-volumes}). Kontsevich
\cite{kontsevich:intersection} used it to compute volumes
coming from a ``symplectic form'' $\Omega=p_1^2\psi_1+\dots+p_n^2\psi_n$,
thus solving Witten's conjecture \cite{witten:intersection}
on the intersection numbers of the $\psi$ classes.

However, in Witten's paper \cite{witten:intersection} matrix
integrals entered in a different way. The idea
was that, in order to integrate
over the space of all conformal structures on $S$, one can
pick a random decomposition of $S$ into polygons, give
each polygon a natural Euclidean structure and extend
it to a conformal structure on $S$, thus obtaining a ``random''
point of $\M_{g,X}$. Refining the polygonalization of $S$
leads to a measure on $\M_{g,X}$. Matrix integrals are
used to enumerate these polygonalizations.

Witten also noticed that this refinement procedure
may lead to different limits,
depending on which polygons we allow.
For instance, we can consider decompositions into
$A$ squares,
or into $A$ squares and $B$ hexagons, and so on.
Dualizing this last polygonalization, we
obtain ribbon graphs embedded in $S$ with $A$ vertices
of valence $4$ and $B$ vertices of valence $6$.
The corresponding locus in $|\Ao(S,X)|$ is called a
Witten subcomplex.

%%%%%%%%%%%%%%%%%%%%%%%%%%%%%%%%%%%%%%%%%%%%%%%%%%%%%%%%%%%%%%%%%%%%
\subsubsection{Witten classes.}
%%%%%%%%%%%%%%%%%%%%%%%%%%%%%%%%%%%%%%%%%%%%%%%%%%%%%%%%%%%%%%%%%%%%
Kontsevich \cite{kontsevich:intersection} and Penner
\cite{penner:poincare} proved that Witten subcomplexes obtained
by requiring that the ribbon graphs have $m_i$ vertices of
valence $(2m_i+3)$ can be oriented
(see also \cite{conant-vogtmann:onatheorem})
and they give cycles in
$\Mbarcomb_{g,X}:=|\Af(S,X)|/\G(S,X)\times\R_+$, which are
denoted by $\ol{W}_{m_*,X}$.
The $\Omega$-volumes of these $\ol{W}_{m_*,X}$ are also computable using 
matrix integrals \cite{kontsevich:intersection}
(see also \cite{dfiz:polynomial}).

In \cite{kontsevich:topology}, Kontsevich
constructed similar cycles using structure constants of
finite-dimensional cyclic $A_{\infty}$-algebras with
positive-definite scalar product and he also claimed that
the classes $\W_{m_*,X}$ (restriction of $\ol{W}_{m_*,X}$ to
$\M_{g,X}$) are Poincar\'e dual to tautological classes.

This last statement (usually called
Witten-Kontsevich's conjecture) was settled independently by
Igusa \cite{igusa:mmm} \cite{igusa:kontsevich} and
Mondello \cite{mondello:combinatorial}, while very little
is known about the nature of the (non-homogeneous) $A_\infty$-classes.

%%%%%%%%%%%%%%%%%%%%%%%%%%%%%%%%%%%%%%%%%%%%%%%%%%%%%%%%%%%%%%%%%%%%
\subsubsection{Surfaces with boundary.}
%%%%%%%%%%%%%%%%%%%%%%%%%%%%%%%%%%%%%%%%%%%%%%%%%%%%%%%%%%%%%%%%%%%%
The key point of all constructions of a ribbon graph
out of a surface is that $X$ must be nonempty, so that
$S\setminus X$ can be retracted by deformation
onto a graph.
In fact, it is not difficult to see that the spine construction
of Penner and Bowditch-Epstein can be performed (even in a more
natural way) on hyperbolic surfaces $\Si$ with geodesic boundary.
The associated cellularization of the corresponding
moduli space is due to Luo
\cite{luo:decomposition}
(for smooth surfaces)
and by Mondello \cite{mondello:triang}
(also for singular surfaces, using Luo's result).

The interesting fact (see \cite{mondello:wp} and \cite{mondello:triang})
is that gluing semi-infinite cylinders at $\pa\Si$ 
produces (conformally)
punctured surfaces that ``interpolate'' between
hyperbolic surfaces with cusps and flat surfaces arising
from Jenkins-Strebel differentials.

%%%%%%%%%%%%%%%%%%%%%%%%%%%%%%%%%%%%%%%%%%%%%%%%%%%%%%%%%%%%%%%%%%%
\subsection{Structure of the paper.}
%%%%%%%%%%%%%%%%%%%%%%%%%%%%%%%%%%%%%%%%%%%%%%%%%%%%%%%%%%%%%%%%%%%
In Sections~\ref{ss:arcs} and \ref{ss:ribbon}, we carefully
define systems of arcs and ribbon graphs, both in the singular
and in the nonsingular case, and we explain how the duality
between the two works. Moreover, we recall Harer's results
on $\Ao(S,X)$ and $\A8(S,X)$ and we state a simple criterion
for compactness inside $|\Ao(S,X)|/\G(S,X)$.

In Sections~\ref{ss:deligne-mumford} and \ref{ss:system},
we describe the Deligne-Mumford moduli space of curves and
the structure of its boundary, the associated stratification
and boundary maps. In \ref{ss:augmented}, we explain
how the analogous bordification of the Teichm\"uller space
$\ol{\Teich}(S,X)$ can be obtained as completion
with respect to the Weil-Petersson metric.

Tautological classes and rings
are introduced in \ref{ss:tautological}
and Kontsevich's compactification of $\M_{g,X}$
is described in \ref{ss:konts}.

In \ref{ss:hmt}, we explain and sketch a proof
of Harer-Mumford-Thurston cellularization of the
moduli space and we illustrate the analogous
result of Penner-Bowditch-Epstein in \ref{ss:pbe}.
In \ref{ss:boundary}, we quickly discuss
the relations between the two constructions
using hyperbolic surfaces with geodesic boundary.

In \ref{ss:witten}, we define Witten subcomplexes
and Witten cycles and we prove (after Kontsevich)
that $\Omega$ orients them. We sketch the
ideas involved in the proof the Witten cycles
are tautological in Section~\ref{ss:witten-tautological}.

Finally, in \ref{ss:stability}, we recall Harer's
stability theorem and we exhibit a combinatorial
construction that shows that Witten cycles are stable.
The fact (and probably also the construction)
is well-known and it is also a direct consequence
of Witten-Kontsevich's conjecture and Miller's work.

%%%%%%%%%%%%%%%%%%%%%%%%%%%%%%%%%%%%%%%%%%%%%%%%%%%%%%%%%%%%%%%%%%%
\subsection{Acknowledgments.}
%%%%%%%%%%%%%%%%%%%%%%%%%%%%%%%%%%%%%%%%%%%%%%%%%%%%%%%%%%%%%%%%%%%
It is a pleasure to thank Shigeyuki Morita,
Athanase Papadopoulos and Robert C.~Penner for
the stimulating workshop ``Teichm\"uller space
(Classical and Quantum)'' they organized
in Oberwolfach (May 28th-June 3rd, 2006)
and the MFO for the hospitality.

I would like to thank Enrico Arbarello for
all I learnt from him about Riemann surfaces
and for his constant encouragement.

%%%%%%%%%%%%%%%%%%%%%%%%%%%%%%%%%%%%%%%%%%%%%%%%%%%%%%%%%%%%%%%%%%%
%%%%%%%%%%%%%%%%%%%%%%%%%%%%%%%%%%%%%%%%%%%%%%%%%%%%%%%%%%%%%%%%%%%%%
\section{Systems of arcs and ribbon graphs}\label{sec:ribbon}
%%%%%%%%%%%%%%%%%%%%%%%%%%%%%%%%%%%%%%%%%%%%%%%%%%%%%%%%%%%%%%%%%%%%

Let $S$ be a compact oriented differentiable surface of genus $g$
with $n>0$ distinct marked points $X=\{x_1,\dots,x_n\}\subset S$.
We will always assume that the Euler characteristic
of the punctured surface $S\setminus X$
is negative, that is $2-2g-n<0$. This restriction
only rules out the cases in which $S\setminus X$ is the
sphere with less than $3$ punctures.

Let $\mathrm{Diff}_+(S,X)$ be the group of orientation-preserving
diffeomorphisms of $S$ that fix $X$ pointwise.
The {\it mapping class group} $\G(S,X)$ is
the group of connected components of $\mathrm{Diff}_+(S,X)$.

In what follows, we borrow some notation
and some ideas from \cite{looijenga:cellular}.

%%%%%%%%%%%%%%%%%%%%%%%%%%%%%%%%%%%%%%%%%%%%%%%%%%%%%%%%%%%%%%%%%%%%%%%%
\subsection{Systems of arcs}\label{ss:arcs}
%%%%%%%%%%%%%%%%%%%%%%%%%%%%%%%%%%%%%%%%%%%%%%%%%%%%%%%%%%%%%%%%%%%%%%%%
\subsubsection{Arcs and arc complex.}
%%%%%%%%%%%%%%%%%%%%%%%%%%%%%%%%%%%%%%%%%%%%%%%%%%%%%%%%%%%%%%%%%%%%%%%%
An oriented {\it arc} in $S$ is a smooth path
$\ora{\a}:[0,1]\rar S$ such that $\ora{\a}([0,1])\cap X=
\{\ora{\a}(0),\ora{\a}(1)\}$,
up to reparametrization.
Let $\mathcal{A}^{or}(S,X)$ be the space of oriented arcs
in $S$, endowed with its natural topology.
Define $\s_1:\mathcal{A}^{or}(S,X)\rar\mathcal{A}^{or}(S,X)$
to be the orientation-reversing
operator and we will write $\s_1(\ora{\a})=\ola{\a}$.
Call $\a$ the $\s_1$-orbit of $\ora{\a}$ and denote by
$\mathcal{A}(S,X)$ the (quotient)
space of $\s_1$-orbits in $\mathcal{A}^{or}(S,X)$.

A {\it system of $(k+1)$-arcs} in $S$ is a collection
$\ua=\{\a_0,\dots,\a_k\}\subset \mathcal{A}(S,X)$ of $k+1$
unoriented arcs such that:
\begin{itemize}
\item
if $i\neq j$, then the intersection
of $\a_i$ and $\a_j$ is contained in $X$
\item
no arc in $\ua$ is homotopically trivial
\item
no pair of arcs in $\ua$ are homotopic to each other.
\end{itemize}
We will denote by $S\setminus\ua$ the
{\it complementary subsurface} of $S$
obtained by removing $\a_0,\dots,\a_k$. 

Each connected component
of the space of systems of $(k+1)$-arcs $\mathcal{AS}_k(S,X)$
is clearly contractible, with
the topology induced by the inclusion $\mathcal{AS}_k(S,X)
\hra \mathcal{A}(S,X)/\mathfrak{S}_k$.

Let $\Af_k(S,X)$ be the set of homotopy classes
of systems of $k+1$ arcs, that is $\Af_k(S,X):=\pi_0\mathcal{AS}_k(S,X)$.

%\begin{definition}
The {\it arc complex} is the simplicial complex
$\dis \Af(S,X)=\bigcup_{k\geq 0}\Af_k(S,X)$.
%\end{definition}

\begin{notation}
We will implicitly identify arc systems $\ua$ and $\ua'$
that are homotopic to each other. Similarly,
we will identify the isotopic subsurfaces
$S\setminus\ua$ and $S\setminus\ua'$.
\end{notation}

%%%%%%%%%%%%%%%%%%%%%%%%%%%%%%%%%%%%%%%%%%%%%%%%%%%%%%%%%%%%%%%%
\subsubsection{Proper simplices.}\label{sss:proper}
%%%%%%%%%%%%%%%%%%%%%%%%%%%%%%%%%%%%%%%%%%%%%%%%%%%%%%%%%%%%%%%%

An arc system $\ua\in\Af(S,X)$ {\it fills}
(resp. {\it quasi-fills}) $S$ if $S\setminus\ua$
is a disjoint union of subsurfaces
homeomorphic to discs (resp.
discs and pointed discs).
It is easy to check that the star of $\ua$
is finite if and only if $\ua$ quasi-fills $S$.
In this case, we also say that $\ua$ is a
{\it proper} simplex of $\Af(S,X)$

Denote by $\A8(S,X)\subset\Af(S,X)$ the subcomplex of 
non-proper simplices and let $\Ao(S,X)=\Af(S,X)\setminus\A8(S,X)$
be the collection of proper ones.

\begin{notation}
We denote by $|\A8(S,X)|$ and $|\Af(S,X)|$ the topological
realizations of $\A8(S,X)$ and $\Af(S,X)$.
We will use the symbol $|\Ao(S,X)|$ to mean the complement
of $|\A8(S,X)|$ inside $|\Af(S,X)|$.
\end{notation}

%%%%%%%%%%%%%%%%%%%%%%%%%%%%%%%%%%%%%%%%%%%%%%%%%%%%%%%%%%%%%%%%%%
\subsubsection{Topologies on $|\Af(S,X)|$.}
%%%%%%%%%%%%%%%%%%%%%%%%%%%%%%%%%%%%%%%%%%%%%%%%%%%%%%%%%%%%%%%%%%

The realization $|\Af(S,X)|$ of the
arc complex can be endowed with two natural topologies
(as is remarked in \cite{bowditch-epstein:natural},
\cite{looijenga:cellular} and \cite{acgh:II}).

The former (which we call {\it standard}) is the finest
topology that makes the inclusions $|\ua|\hra|\Af(S,X)|$
continuous for all $\ua\in\Af(S,X)$; in other words,
a subset $U\subset |\Af(S,X)|$
is declared to be open if and only if
$U \cap |\ua|$ is open for every $\ua\in\Af(S,X)$.
The latter topology is induced by the path {\it metric} $d$,
which is the largest metric that restricts to the Euclidean
one on each closed simplex.

The two topologies are the same where $|\Af(S,X)|$ is locally
finite, but the latter is coarser elsewhere.
We will always consider all realizations to be endowed
with the metric topology.

%%%%%%%%%%%%%%%%%%%%%%%%%%%%%%%%%%%%%%%%%%%%%%%%%%%%%%%%%%%%%%%
\subsubsection{Visible subsurfaces.}
%%%%%%%%%%%%%%%%%%%%%%%%%%%%%%%%%%%%%%%%%%%%%%%%%%%%%%%%%%%%%%%

For every system of arcs $\ua\in\Af(S,X)$,
define $S(\ua)_+$ to be the largest isotopy class
of open subsurfaces of $S$ such that
\begin{itemize}
\item
every arc in $\ua$ is contained in $S(\ua)_+$
\item
$\ua$ quasi-fills $S(\ua)_+$.
\end{itemize}

The {\it visible subsurface}
$S(\ua)_+$ can be constructed by taking the union of
a thickening a representative
of $\ua$ inside $S$ and all those connected components
of $S\setminus\ua$ which are homeomorphic to discs or
punctured discs (this construction appears
first in \cite{bowditch-epstein:natural}).
We will always consider $S(\ua)_+$ as an open
subsurface (up to isotopy), homotopically equivalent to its
closure $\ol{S(\ua)_+}$, which is
an embedded surface with boundary.

\begin{center}
\begin{figurehere}
\psfrag{x1}{$x_1$}
\psfrag{x2}{$x_2$}
\psfrag{x3}{$x_3$}
\psfrag{S}{$S$}
\includegraphics[width=0.7\textwidth]{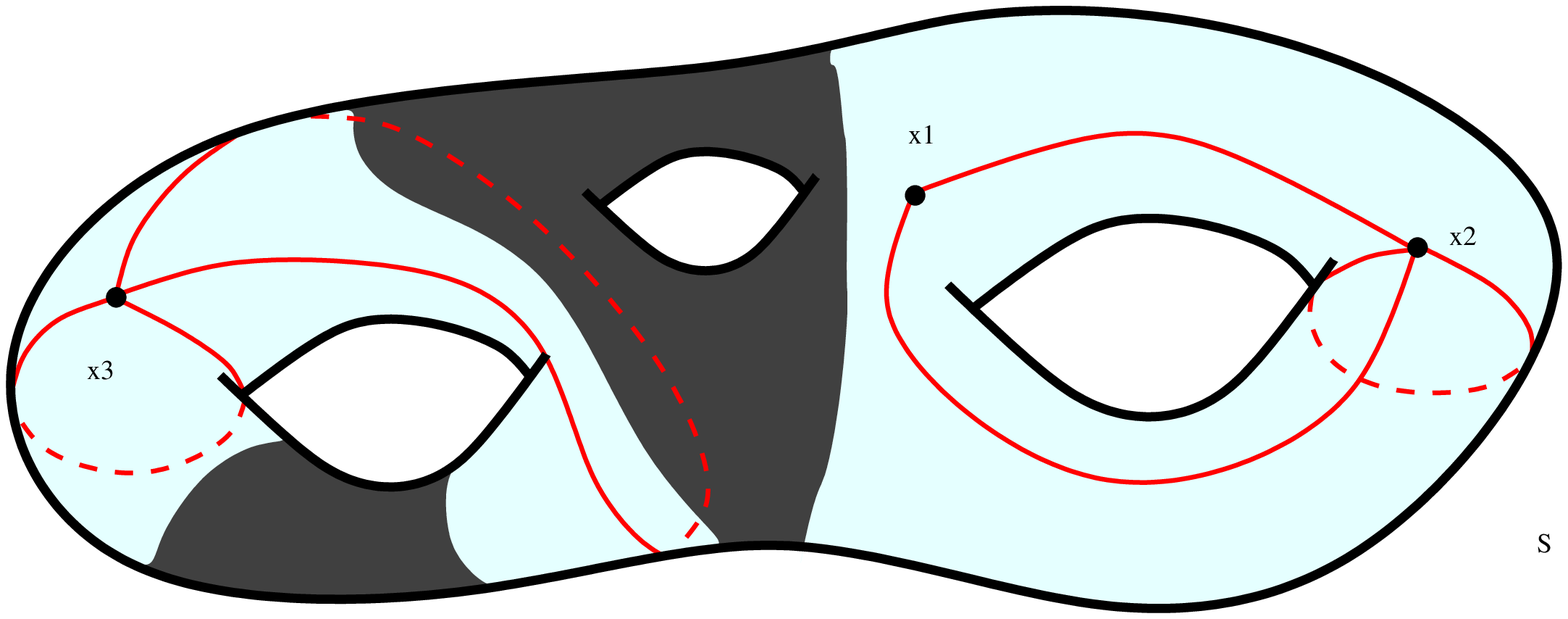}
\caption{The invisible subsurface is the dark non-cylindrical
component.}
\end{figurehere}
\end{center}

One can rephrase \ref{sss:proper} by saying that $\ua$
is proper if and only if all $S$ is $\ua$-visible.

We call {\it invisible subsurface}
$S(\ua)_-$ associated to $\ua$
the union of the connected components
of $S\setminus\ol{S(\ua)_+}$
which are not unmarked cylinders.
We also say that a marked point $x_i$ is (in)visible
for $\ua$ if it belongs to the $\ua$-(in)visible subsurface.

%%%%%%%%%%%%%%%%%%%%%%%%%%%%%%%%%%%%%%%%%%%%%%%%%%%%%%%%%%%%%%%%
\subsubsection{Ideal triangulations.}
%%%%%%%%%%%%%%%%%%%%%%%%%%%%%%%%%%%%%%%%%%%%%%%%%%%%%%%%%%%%%%%%

A maximal system of arcs $\ua\in\Af(S,X)$ is also called an
{\it ideal triangulation} of $S$. In fact, it is easy to check
that, in this case, each component
of $S\setminus \ua$ bounded by three arcs and so is
a ``triangle''. (The term ``ideal'' comes from the fact that
one often thinks of $(S,X)$ as a hyperbolic surface with
cusps at $X$ and of $\ua$ as a collection of hyperbolic geodesics.)
It is also clear that such an $\ua$ is proper.

\begin{center}
\begin{figurehere}
\psfrag{x1}{$x_1$}
\psfrag{x2}{$x_2$}
\psfrag{S}{$S$}
\includegraphics[width=0.5\textwidth]{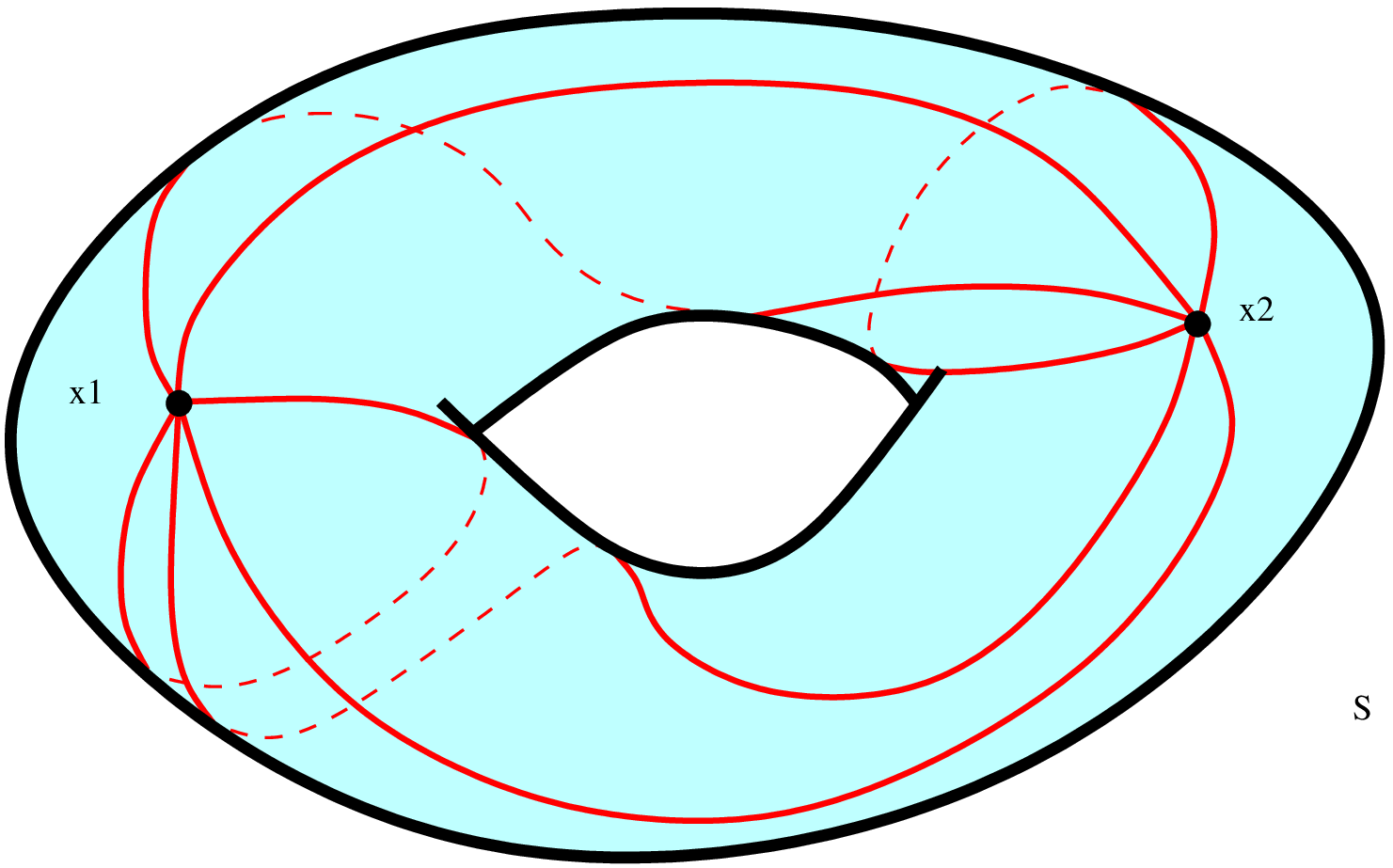}
\caption{An example of an ideal triangulation for $(g,n)=(1,2)$.}
\end{figurehere}
\end{center}

A simple calculation with the Euler characteristic of $S$
shows that an ideal triangulation is made of exactly $6g-6+3n$ arcs.

%%%%%%%%%%%%%%%%%%%%%%%%%%%%%%%%%%%%%%%%%%%%%%%%%%%%%%%%%%%%%%%
\subsubsection{The spine of $|\Ao(S,X)|$.}
%%%%%%%%%%%%%%%%%%%%%%%%%%%%%%%%%%%%%%%%%%%%%%%%%%%%%%%%%%%%%%

Consider the baricentric subdivision $\Af(S,X)'$, whose
$k$-simplices are chains $(\ua_0\subsetneq \ua_1\subsetneq\dots
\subsetneq \ua_k)$. There is an obvious piecewise-affine
homeomorphism $|\Af(S,X)'|\rar |\Af(S,X)|$, that sends
a vertex $(\ua_0)$ to the baricenter of $|\ua_0|\subset|\Af(S,X)|$.

Call $\Ao(S,X)'$ the subcomplex of $\Af(S,X)'$, whose
simplices are chains of simplices that belong to $\Ao(S,X)$.
Clearly, $|\Ao(S,X)'|\subset |\Af(S,X)'|$ is contained in
$|\Ao(S,X)|\subset|\Af(S,X)|$
through the homeomorphism above.

It is a general fact that there is a deformation retraction
of $|\Ao(S,X)|$ onto $|\Ao(S,X)'|$: on each simplex of
$|\Af(S,X)'|\cap |\Ao(S,X)|$ this is given by projecting
onto the face contained in $|\Ao(S,X)'|$.
It is also clear that the retraction is $\G(S,X)$-equivariant.

In the special case of $X=\{x_1\}$, a proper system contains
at least $2g$ arcs; whereas a maximal system contains
exactly $6g-3$ arcs. Thus, the (real) dimension of
$|\Ao(S,X)'|$ is $(6g-3)-2g=4g-3$.

\begin{proposition}[Harer \cite{harer:virtual}]\label{prop:spine}
If $X=\{x_1\}$, the spine $|\Ao(S,X)'|$ has dimension $4g-3$.
\end{proposition}

%%%%%%%%%%%%%%%%%%%%%%%%%%%%%%%%%%%%%%%%%%%%%%%%%%%%%%%%%%%%%%%%
\subsubsection{Action of $\s$-operators.}\label{sss:sigma-arc}
%%%%%%%%%%%%%%%%%%%%%%%%%%%%%%%%%%%%%%%%%%%%%%%%%%%%%%%%%%%%%%%%

For every arc system $\ua=\{\a_0,\dots,\a_k\}$,
denote by $E(\ua)$ the subset $\{\ora{\a_0},\ola{\a_0},\dots,
\ora{\a_k},\ola{\a_k}\}$ of $\pi_0\mathcal{A}^{or}(S,X)$.
The action of $\s_1$ clearly restricts to $E(\ua)$.

For each $i=1,\dots,n$,
the orientation of $S$ induces a cyclic ordering of the
oriented arcs in $E(\ua)$ outgoing from $x_i$.

If $\ora{\a_j}$ starts at $x_i$, then define
$\s_{\infty}(\ora{\a_j})$ to be the oriented arc in $E(\ua)$
outgoing from $x_i$
that comes just {\it before} $\ora{\a_j}$.
Moreover, $\s_0$ is defined by $\s_0=\s_{\infty}^{-1}\s_1$.

If we call $E_t(\ua)$ the orbits of $E(\ua)$ under the
action of $\s_t$, then
\begin{itemize}
\item
$E_1(\ua)$ can be identified with $\ua$
\item
$E_{\infty}(\ua)$ can be identified with
the set of $\ua$-visible marked points
\item
$E_0(\ua)$ can be identified to the set of connected
components of $S(\ua)_+\setminus\ua$.
\end{itemize}
Denote by $[\ora{\a_j}]_t$ the $\s_t$-orbit of $\ora{\a_j}$,
so that $[\ora{\a_j}]_1=\a_j$ and $[\ora{\a_j}]_\infty$ is
the starting point of $\ora{\a_j}$, whereas
$[\ora{\a_j}]_0$ is the component of $S(\ua)_+\setminus\ua$
adjacent to $\a_j$ and which induces the orientation $\ora{\a_j}$
on it.

%%%%%%%%%%%%%%%%%%%%%%%%%%%%%%%%%%%%%%%%%%%%%%%%%%%%%%%%%%%%%%%%%%
\subsubsection{Action of $\G(S,X)$ on $\Af(S,X)$.}
%%%%%%%%%%%%%%%%%%%%%%%%%%%%%%%%%%%%%%%%%%%%%%%%%%%%%%%%%%%%%%%%%%

There is a natural right action of the mapping class group
\[
\xymatrix@R=0in{
\mathcal{A}(S,X)\times\G(S,X) \ar[rr] && \mathcal{A}(S,X) \\
(\a,g) \ar@{|->}[rr] && \a\circ g
}
\]
The induced action on $\Af(S,X)$
preserves $\A8(S,X)$ and so $\Ao(S,X)$.

It is easy to see that the stabilizer (under $\G(S,X)$) of
a simplex $\ua$ fits in the following exact sequence
\[
1 \rar \G_{cpt}(S\setminus\ua,X)\rar
\mathrm{stab}_\G(\ua) \rar \mathfrak{S}(\ua)
\]
where $\mathfrak{S}(\ua)$ is the group of permutations
of $\ua$ and $\G_{cpt}(S\setminus\ua,X)$ is the mapping
class group of orientation-preserving diffeomorphisms
of $S\setminus\ua$ {\it with compact support} that fix $X$.
Define the image of $\mathrm{stab}_\G(\ua)\rar\mathfrak{S}(\ua)$
to be the {\it automorphism group of $\ua$}.

We can immediately conclude that $\ua$ is proper
if and only if $\mathrm{stab}_\G(\ua)$ is finite
(equivalently, if and only if
$\G_{cpt}(S\setminus\ua,X)$ is trivial).

%%%%%%%%%%%%%%%%%%%%%%%%%%%%%%%%%%%%%%%%%%%%%%%%%%%%%%%%%%%%%%%
\subsubsection{Weighted arc systems.}
%%%%%%%%%%%%%%%%%%%%%%%%%%%%%%%%%%%%%%%%%%%%%%%%%%%%%%%%%%%%%%%

A point $\ol{w}\in |\Af(S,X)|$ consists of a map
$\ol{w}:\Af_0(S,X)\rar [0,1]$ such that
\begin{itemize}
\item
the support of $\ol{w}$ is a simplex $\ua=\{\a_0,\dots,\a_k\}\in\Af(S,X)$
\item
$\dis\sum_{i=0}^k \ol{w}(\a_i)=1$.
\end{itemize}
We will call $\ol{w}$ the {\it (projective) weight} of $\ua$.
A {\it weight} for $\ua$ is a
point of $w\in|\Af(S,X)|_{\R}:=|\Af(S,X)|\times\R_+$,
that is a map $w:\Af_0(S,X)\rar\R_+$
with support on $\ua$. Call $\ol{w}$ its associated
projective weight.

%%%%%%%%%%%%%%%%%%%%%%%%%%%%%%%%%%%%%%%%%%%%%%%%%%%%%%%%%%%%%%%%%%%%%
\subsubsection{Compactness in $|\Ao(S,X)|/\G(S,X)$.}\label{sss:compact}
%%%%%%%%%%%%%%%%%%%%%%%%%%%%%%%%%%%%%%%%%%%%%%%%%%%%%%%%%%%%%%%%%%%%%

We are going to prove a simple criterion for a subset of
$|\Ao(S,X)|/\G(S,X)$ to be compact.

Call $\mathcal{C}(S,X)$ the set of free homotopy classes
of simple closed curves on $S\setminus X$, which are neither
contractible nor homotopic to a puncture.

Define the ``intersection product''
\[
\i:\mathcal{C}(S,X)\times |\Af(S,X)|\rar \R_{\geq 0}
\]
as $\i(\g,\ol{w})=\sum_{\a} \i(\g,\a) \ol{w}(\a)$,
where $\i(\g,\a)$ is the {\it geometric} intersection number.
We will also refer to $\i(\g,\ol{w})$ as to the {\it length}
of $\g$ at $\ol{w}$. Consequently, we will say that
the {\it systol} at $\ol{w}$ is
\[
\mathrm{sys}(\ol{w})=
\mathrm{inf}\{\i(\g,\ol{w})\,|\,\g\in\mathcal{C}(S,X)\}.
\]

Clearly, the function $\mathrm{sys}$ descends to
\[
\mathrm{sys}:|\Af(S,X)|/\G(S,X)\rar \R_+\,
\]

\begin{lemma}\label{lemma:compact}
A closed subset $K\subset |\Ao(S,X)|/\G(S,X)$ is {\it compact} if
and only if $\exists \e>0$ such that
$\mathrm{sys}([\ol{w}])\geq \e$ for all $[\ol{w}]\in K$.
\end{lemma}

\begin{proof}
In $\R^N$ we easily have $\dis d_2\leq d_1\leq \sqrt{N}\cdot d_2$,
where $d_r$ is the $L^r$-distance. Similarly,
in $|\Af(S,X)|$ we have
\[
d(\ol{w},|\A8(S,X)|)\leq \mathrm{sys}(\ol{w})
\leq \sqrt{N} \cdot d(\ol{w},|\A8(S,X)|)
\]
where $N=6g-7+3n$. The same holds in $|\Af(S,X)|/\G(S,X)$.

Thus, if $[\ua]\in\Ao(S,X)/\G(S,X)$,
then $|\ua|\cap \mathrm{sys}^{-1}([\e,\infty))
\cap |\Ao(S,X)|/\G(S,X)$ is compact for
every $\e>0$.
As $|\Ao(S,X)|/\G(S,X)$ contains finitely many
cells, we conclude that $\mathrm{sys}^{-1}([\e,\infty))
\cap|\Ao(S,X)|/\G(S,X)$ is compact.

Vice versa, if $\mathrm{sys}:K\rar \R_+$ is not bounded
from below, then we can find a sequence $[\ol{w}_m]
\subset K$ such that $\mathrm{sys}(\ol{w}_m)\rar 0$.
Thus, $[\ol{w}_m]$ approaches $|\Af(S,X)|/\G(S,X)$
and so is divergent in $|\Ao(S,X)|/\G(S,X)$.
\end{proof}

%%%%%%%%%%%%%%%%%%%%%%%%%%%%%%%%%%%%%%%%%%%%%%%%%%%%%%%%%%%%%%%%%%
\subsubsection{Boundary weight map.}
%%%%%%%%%%%%%%%%%%%%%%%%%%%%%%%%%%%%%%%%%%%%%%%%%%%%%%%%%%%%%%%%%%

Let $\Delta_X$ be the standard simplex in $\R^X$.
The {\it boundary weight map}
$\ell_{\pa}:|\Af(S,X)|_{\R}\rar \Delta_X\times\R_+\subset \R^X$
is the piecewise-linear map that sends
$\{\a\}\mapsto [\ora{\a}]_{\infty}+[\ola{\a}]_{\infty}$.
The projective boundary weight map
$\frac{1}{2}\ell_{\pa}:|\Af(S,X)|\rar \Delta_X$ instead sends
$\{\a\}\mapsto \frac{1}{2}[\ora{\a}]_{\infty}+\frac{1}{2}
[\ola{\a}]_{\infty}$.

%%%%%%%%%%%%%%%%%%%%%%%%%%%%%%%%%%%%%%%%%%%%%%%%%%%%%%%%%%%%%%%%%%%
\subsubsection{Results on the arc complex.}
%%%%%%%%%%%%%%%%%%%%%%%%%%%%%%%%%%%%%%%%%%%%%%%%%%%%%%%%%%%%%%%%%%%

A few things are known about the topology of $|\Af(S,X)|$.

\begin{itemize}
\item[(a)]
The space of proper arc systems $|\Ao(S,X)|$ can be naturally
given the structure of piecewise-affine topological manifold
with boundary
(Hubbard-Masur \cite{hubbard-masur:foliations}, credited to
Whitney) of (real) dimension $6g-7+3n$.
\item[(b)]
The space $|\Ao(S,X)|$ is $\G(S,X)$-equivariantly
homeomorphic to $\Teich(S,X)\times\Delta_X$,
where $\Teich(S,X)$ is the Teichm\"uller space of $(S,X)$
(see \ref{sss:teichmueller} for definitions
and Section~\ref{sec:triangulation} for an extensive
discussion on this result), and so is
contractible. This result could also be probably
extracted from \cite{hubbard-masur:foliations}, but
it is first more explicitly stated in
Harer \cite{harer:virtual} (who attributes it to
Mumford and Thurston),
Penner \cite{penner:decorated} and Bowditch-Epstein
\cite{bowditch-epstein:natural}.
As the moduli space of $X$-marked
Riemann surfaces of genus $g$ can be obtained
as $\M_{g,X}\cong\Teich(S,X)/\G(S,X)$ (see \ref{sss:moduli}),
then $\M_{g,X}\simeq B\G(S,X)$ in the orbifold category.
\item[(c)] 
The space $|\A8(S,X)|$ is homotopy equivalent to
an infinite wedge of spheres of dimension $2g-3+n$
(Harer \cite{harer:virtual}).
\end{itemize}

Results (b) and (c) are the key step in the following.

\begin{theorem}[Harer \cite{harer:virtual}]
$\G(S,X)$ is a virtual duality group
(that is, it has a subgroup of finite index
which is a duality group)
of dimension $4g-4+n$ for $n>0$
(and $4g-5$ for $n=0$).
\end{theorem}

Actually, it is sufficient to work with $X=\{x_1\}$,
in which case the upper bound is given by (b)
and Proposition~\ref{prop:spine}, and
the duality by (c).

%%%%%%%%%%%%%%%%%%%%%%%%%%%%%%%%%%%%%%%%%%%%%%%%%%%%%%%%%%%%%%%%%%%%%
%%%%%%%%%%%%%%%%%%%%%%%%%%%%%%%%%%%%%%%%%%%%%%%%%%%%%%%%%%%%%%%%%%%%%
\subsection{Ribbon graphs}\label{ss:ribbon}
%%%%%%%%%%%%%%%%%%%%%%%%%%%%%%%%%%%%%%%%%%%%%%%%%%%%%%%%%%%%%%%%%%%%%
\subsubsection{Graphs.}
%%%%%%%%%%%%%%%%%%%%%%%%%%%%%%%%%%%%%%%%%%%%%%%%%%%%%%%%%%%%%%%%%%%%%

A {\it graph} $G$ is a triple $(E,\sim,\s_1)$, where $E$ is a finite
set, $\s_1:E\rar E$ is a fixed-point-free involution and $\sim$ is
an equivalence relation on $E$.

In ordinary language
\begin{itemize}
\item
$E$ is  the set of {\it oriented edges} of the graph
\item
$\s_1$ is the orientation-reversing involution of $E$,
so that the set of unoriented edges is $E_1:=E/\s_1$
\item
two oriented edges are equivalent if and only if they
come out from the same vertex, so that the set $V$ of
vertices is $E/\!\sim$ and the valence of $v\in E/\!\sim$
is exactly $|v|$.
\end{itemize}

A {\it ribbon graph} $\GG$ is a triple $(E,\s_0,\s_1)$,
where $E$ is a (finite) set, $\s_1:E\rar E$ is a fixed-point-free
involution and $\s_1:E\rar E$ is a permutation.
Define $\s_\infty:=\s_1\circ \s_0^{-1}$ and call $E_t$
the set of orbits of $\s_t$ and $[\cdot]_t:E\rar E_t$
the natural projection.
A disjoint union of two ribbon graphs is defined in
the natural way.

\begin{remark}
Given a ribbon graph $\GG$, the underlying ordinary graph
$G=\GG^{ord}$ is obtained by declaring that oriented edges in the same
$\s_0$-orbit are equivalent and forgetting about the
precise action of $\s_0$.
\end{remark}

\begin{center}
\begin{figurehere}
\psfrag{x}{$x_i$}
\psfrag{e}{$\ora{e}$}
\psfrag{s0e}{$\s_0(\ora{e})$}
\psfrag{s0}{$\s_0$}
\psfrag{s1e'}{$\ola{e'}=\s_1(\ora{e'})$}
\psfrag{Te'}{$T_{\ora{e'}}$}
\psfrag{e'}{$\ora{e'}$}
\psfrag{s8}{$\s_\infty$}
\psfrag{s8e'}{$\s_\infty(\ora{e'})$}
\includegraphics[width=0.7\textwidth]{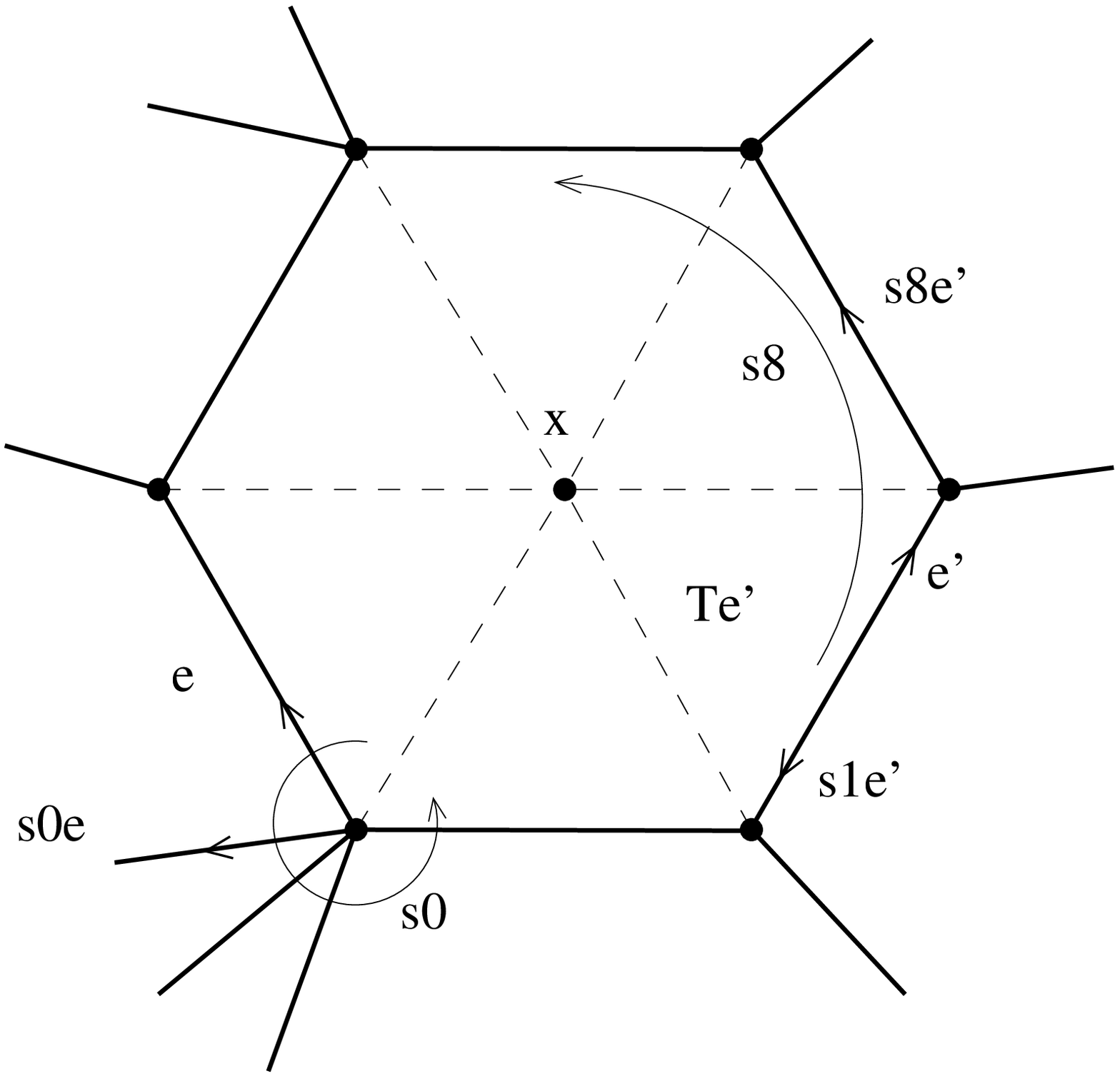}
\caption{Geometric representation of a ribbon graph}
\label{fig:ribbon}
\end{figurehere}
\end{center}

In ordinary language, a ribbon graph is an ordinary
graph endowed with a cyclic ordering of the oriented
edges outgoing from each vertex.

The $\s_\infty$-orbits are sometimes called {\it holes}.
A {\it connected component} of $\GG$ is an orbit
of $E(\GG)$ under the action of $\langle \s_0,\s_1 \rangle$.

The {\it Euler characteristic} of a ribbon graph $\GG$
is $\chi(\GG)=|E_0(\GG)|-|E_1(\GG)|$
and its {\it genus} is $g(\GG)=1+\frac{1}{2}
(|E_1(\GG)|-|E_0(\GG)|-|E_\infty(\GG)|)$.

A {\it (ribbon) tree} is a connected (ribbon) graph
of genus zero with one hole.

%An {\it $X$-marking} for $\GG$ is an injection
%$X\arr{\sim}{\lra} E_\infty(\GG)\cup E_0(\GG)$
%whose image contains all holes.
%
%%%%%%%%%%%%%%%%%%%%%%%%%%%%%%%%%%%%%%%%%%%%%%%%%%%%%%%%%%%%%%%%%%%%%
\subsubsection{Subgraphs and quotients.}
%%%%%%%%%%%%%%%%%%%%%%%%%%%%%%%%%%%%%%%%%%%%%%%%%%%%%%%%%%%%%%%%%%%%%

Let $\GG=(E,\s_0,\s_1)$ be a ribbon graph and let $Z\subsetneq E_1$
be a nonempty subset of edges.

The {\it subgraph $\GG_Z$} is given by $(\tilde{Z},\s^Z_0,\s^Z_1)$,
where $\tilde{Z}=Z\times_{E_1} E$ and $\s^Z_0,\s^Z_1$ are the induced
operators (that is, for every $e\in\tilde{Z}$ we define
$\s^Z_0(e)=\s_0^k(e)$,
where $k=\mathrm{min}\{k>0\,|\,\s_0^k(e)\in\tilde{Z}\}$).

Similarly, the {\it quotient $\GG/Z$} is $(\GG\setminus\tilde{Z},\s^{Z^c}_0,
\s^{Z^c}_1)$, where $\s^{Z^c}_1$ and $\s^{Z^c}_{\infty}$ are the operators
induced on $E\setminus\tilde{Z}$ and $\s^{Z^c}_0$ is defined accordingly.
A {\it new vertex} of $\GG/Z$ is a $\s^{Z^c}_0$-orbit of
$E \setminus\tilde{Z}\hra\GG$, which is not a $\s_0$-orbit.

%%%%%%%%%%%%%%%%%%%%%%%%%%%%%%%%%%%%%%%%%%%%%%%%%%%%%%%%%%%%%%%%%%%%%
\subsubsection{Bicolored graphs.}
%%%%%%%%%%%%%%%%%%%%%%%%%%%%%%%%%%%%%%%%%%%%%%%%%%%%%%%%%%%%%%%%%%%%%

A {\it bicolored graph} $\zeta$ is
a finite connected graph with a partition $V=V_+\cup V_-$
of its vertices.
We say that $\zeta$ is {\it reduced} if no two vertices of $V_-$
are adjacent. If not differently specified, we will always
understand that bicolored graphs are reduced.

If $\zeta$ contains an edge $z$ that joins $w_1,w_2\in V_-$, then
we can obtain a new graph $\zeta'$
{\it merging} $w_1$ and $w_2$ along $z$ into a new vertex
$w'\in V'_-$ (by simply forgetting
$\ora{z}$ and $\ola{z}$ and by declaring that vertices outgoing
from $w_1$ are equivalent to vertices outgoing from $w_2$).

If $\zeta$ comes equipped with a function $g:V_-\rar\N$, then
$g':V'_-\rar \N$ is defined so that $g'(w')=g(w_1)+g(w_2)$ if
$w_1\neq w_2$, or $g'(w')=g(w_1)+1$ if $w_1=w_2$.

As merging reduces the number of edges, we can iterate the process
only a finite number of times. The result is independent of
the choice of which edges to merge first and is a reduced
graph $\zeta^{red}$ (possibly with a $g^{red}$).

\begin{center}
\begin{figurehere}
\psfrag{t1}{$t_1$}
\psfrag{t2}{$t_2$}
\psfrag{t3}{$t_3$}
\psfrag{t4}{$t_4$}
\psfrag{s5}{$s_5$}
\psfrag{s6}{$s_6$}
\psfrag{0}{$0$}
\psfrag{3}{$3$}
\psfrag{1}{{\color{White}\boldmath$1$}}
\psfrag{2}{{\color{White}\boldmath$2$}}
\psfrag{x1}{$x_1$}
\psfrag{x2}{$x_2$}
\psfrag{x3}{$x_3$}
\includegraphics[width=0.7\textwidth]{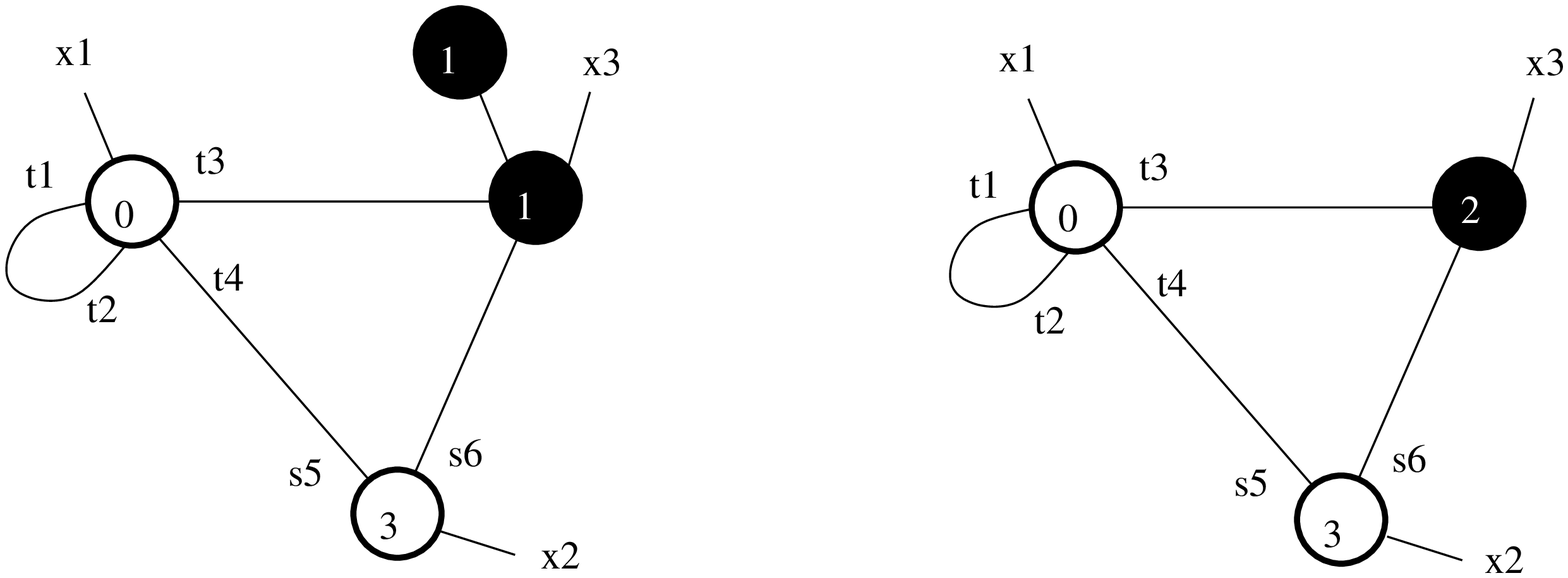}
\caption{A non-reduced bicolored graph (on the left) and
its reduction (on the right). Vertices in $V_-$ are black.}
\label{fig:bicolored}
\end{figurehere}
\end{center}

%%%%%%%%%%%%%%%%%%%%%%%%%%%%%%%%%%%%%%%%%%%%%%%%%%%%%%%%%%%%%%%%%%%%%
\subsubsection{Enriched ribbon graphs.}\label{sss:enriched}
%%%%%%%%%%%%%%%%%%%%%%%%%%%%%%%%%%%%%%%%%%%%%%%%%%%%%%%%%%%%%%%%%%%%%

An {\it enriched $X$-marked ribbon graph $\GG^{en}$} is the datum of
\begin{itemize}
\item
a connected bicolored graph $(\zeta,V_+)$
\item
a ribbon graph $\GG$ plus a bijection
$V_+\rar \{\text{connected components of $\GG$}\}$
\item
an (invisible) genus function $g:V_-\rar \N$
\item
a map $\dis m:X \rar 
V_-\cup E_{\infty}(\GG)$
such that the restriction
$m^{-1}(E_\infty(\GG))\rar
E_\infty(\GG)$ is bijective
\item
an injection
$s_v:\{\text{oriented edges of $\zeta$ outgoing from $v$}\}\rar E_0(\GG_v)$
(vertices of $\GG_v$ in the image are called {\it special})
for every $v\in V_+$
\end{itemize}
that satisfy the following properties:
\begin{itemize}
\item
for every $v\in V_+$ and $y\in E_0(\GG_v)$
we have $|m^{-1}(y)\cup s_v^{-1}(y)|\leq 1$
(i.e. no more than one marking or one node
at each vertex of $\GG_v$)
\item
$2g(v)-2+|\{\text{oriented edges of $\zeta$ outgoing from $v$}\}|+$\\
$+|\{\text{marked points on $v$} \}|> 0$
for every $v\in V$ (stability condition)
\item
every non-special vertex of $\GG_v$ must be at least trivalent
for all $v\in V_+$.
\end{itemize}
We say that $\GG^{en}$ is {\it reduced} if $\zeta$ is.

If the graph $\zeta$ is not reduced, then
we can merge two vertices of $\zeta$
along an edge of $\zeta$ and obtain a new enriched
$X$-marked ribbon graph.
$\GG^{en}_1$ and $\GG^{en}_2$ are considered equivalent
if they are related by a sequence of merging operations.
It is clear that each equivalence class can be identified
to its reduced representative.
Unless differently specified, we will always refer
to an enriched graph as the canonical reduced representative.

The total {\it genus} of $\GG^{en}$ is
$g(\GG^{en})=1-\chi(\zeta)+\sum_{v\in V_+} g(\GG_v)
+\sum_{w\in V_-}g(w)$.

\begin{example}
In Figure~\ref{fig:bicolored}, the genus
of each vertex is written inside, $x_1$ and $x_2$ are
marking the two holes of $\GG$
(sitting in different components),
whereas $x_3$ is an invisible marked point.
Moreover, $t_1,t_2,t_3,t_4$ (resp. $s_5,s_6$)
are distinct (special) vertices
of the visible component of genus $0$ (resp. of genus $3$).
The total genus of the associated $\GG^{en}$ is $7$.
\end{example}

\begin{remark}
If an edge $z$ of $\zeta$ joins $v\in V_+$ and $w\in V_-$
and this edge is marked by the special vertex $y\in E_0(\GG_v)$,
then we will say, for brevity, that $z$ joins $w$ to $y$. 
\end{remark}

An enriched $X$-marked ribbon graph is {\it nonsingular}
if $\zeta$ consists of a single visible vertex.
Equivalently, an enriched nonsingular
$X$-marked ribbon graph consists of a connected
ribbon graph $\GG$
together with an injection $X \hra E_\infty(\GG)\cup E_0(\GG)$,
whose image is exactly $E_\infty(\GG)\cup\{\text{special vertices}\}$,
such that non-special vertices are at least trivalent
and $\chi(\GG)-|\{\text{marked vertices}\}|<0$.

%%%%%%%%%%%%%%%%%%%%%%%%%%%%%%%%%%%%%%%%%%%%%%%%%%%%%%%%%%%%%%%%%%%%%
\subsubsection{Category of nonsingular ribbon graphs.}
%%%%%%%%%%%%%%%%%%%%%%%%%%%%%%%%%%%%%%%%%%%%%%%%%%%%%%%%%%%%%%%%%%%%%

A {\it morphism} of nonsingular $X$-marked
ribbon graphs $\GG_1 \rar \GG_2$
is an injective map $f:E(\GG_2)\hra E(\GG_1)$
such that
\begin{itemize}
\item
$f$ commutes with $\s_1$, $\s_\infty$ and respects the $X$-marking
\item
$\GG_{1,Z}$ is a disjoint union of trees,
where $Z=E_1(\GG_1)\setminus E_1(\GG_2)$.
\end{itemize}
Notice that, as $f$ preserves the $X$-markings (which are
{\it injections} $X\hra E_\infty(\GG_i)\cup E_0(\GG_i)$),
then each component of $Z$ may contain at most one special vertex.

Vice versa,
if $\GG$ is a nonsingular $X$-marked ribbon graph
and $\emptyset\neq Z\subsetneq E_1(\GG)$ such that
$\GG_Z$ is a disjoint union of trees (each one
containing at most a special vertex), then
the inclusion $f:E_1(\GG)\setminus\tilde{Z}\hra E_1(\GG)$ induces
a morphism of nonsingular ribbon graphs $\GG\rar \GG/Z$.

\begin{remark}
A morphism is an {\it isomorphism} if and only
if $f$ is bijective.
\end{remark}

$\mathfrak{RG}_{X,ns}$ is the small category whose objects are
nonsingular $X$-marked ribbon
graphs $\GG$ (where we assume that $E(\GG)$ is contained
in a fixed countable set) with the morphisms defined above.
We use the symbol $\mathfrak{RG}_{g,X,ns}$ to denote the
full subcategory of ribbon graphs of genus $g$.

%%%%%%%%%%%%%%%%%%%%%%%%%%%%%%%%%%%%%%%%%%%%%%%%%%%%%%%%%%%%%%%%%%%%%
\subsubsection{Topological realization of nonsingular ribbon graphs.}
\label{sss:ns-ribbon_graphs}
%%%%%%%%%%%%%%%%%%%%%%%%%%%%%%%%%%%%%%%%%%%%%%%%%%%%%%%%%%%%%%%%%%%%%

A topological realization $|G|$ of the graph $G=(E,\sim,\s_1)$
is the one-dimensional CW-complex
obtained from $\dis I\times E$
(where $I=[0,1]$) by identifying
\begin{itemize}
\item
$(t,\ora{e})\sim (1-t,\ola{e})$ for all $t\in I$ and $\ora{e}\in E$
\item
$(0,\ora{e})\sim (0,\ora{e'})$ whenever $e\sim e'$.
\end{itemize}

A {\it topological realization} $|\GG|$
of the nonsingular $X$-marked
ribbon graph $\GG=(E,\s_0,\s_1)$
is the oriented surface
obtained from $\dis T\times E$
(where $T=I\times[0,\infty]/I\times\{\infty\}$) by identifying
\begin{itemize}
\item
$(t,0,\ora{e})\sim (1-t,0,\ola{e})$
for all $\ora{e}\in E$
\item
$(1,y,\ora{e})\sim (0,y,s_{\infty}(\ora{e}))$
for all $\ora{e}\in E$ and $y\in [0,\infty]$.
\end{itemize}
If $G$ is the ordinary graph underlying $\GG$,
then there is a natural embedding $|G|\hra |\GG|$,
which we call the {\it spine}.

The points at infinity in $|\GG|$ are called {\it centers}
of the holes and can be identified to $E_\infty(\GG)$.
Thus, $|\GG|$ is naturally an $X$-marked surface.

Notice that a morphism of nonsingular
$X$-marked ribbon graphs $\GG_1\rar \GG_2$
induces an isotopy class of orientation-preserving
homeomorphisms $|\GG_1|\rar |\GG_2|$ that respect the $X$-marking.

%%%%%%%%%%%%%%%%%%%%%%%%%%%%%%%%%%%%%%%%%%%%%%%%%%%%%%%%%%%%%%%%%%%%%
\subsubsection{Nonsingular $(S,X)$-markings.}
%%%%%%%%%%%%%%%%%%%%%%%%%%%%%%%%%%%%%%%%%%%%%%%%%%%%%%%%%%%%%%%%%%%%%

An {\it $(S,X)$-marking} of the nonsingular $X$-marked
ribbon graph $\GG$
is an orientation-preserving
homeomorphism $f:S\rar |\GG|$, compatible with
$X\hra E_\infty(\GG)\cup E_0(\GG)$.

Define $\mathfrak{RG}_{ns}(S,X)$ to be the category whose objects are
$(S,X)$-marked nonsingular
ribbon graphs $(\GG,f)$ and whose morphisms
$(\GG_1,f_1)\rar (\GG_2,f_2)$ are morphisms $\GG_1\rar \GG_2$
such that $S\arr{f_1}{\lra} |\GG_1| \rar |\GG_2|$ is homotopic
to $f_2: S\rar |\GG_2|$.

As usual, there is a right action of
the mapping class group $\G(S,X)$ on $\mathfrak{RG}_{ns}(S,X)$
and $\mathfrak{RG}_{ns}(S,X)/\G(S,X)$ is equivalent to
$\mathfrak{RG}_{g,X,ns}$.

%%%%%%%%%%%%%%%%%%%%%%%%%%%%%%%%%%%%%%%%%%%%%%%%%%%%%%%%%%%%%%%%%%%%%
\subsubsection{Nonsingular arcs/graph duality.}\label{sss:ns-duality}
%%%%%%%%%%%%%%%%%%%%%%%%%%%%%%%%%%%%%%%%%%%%%%%%%%%%%%%%%%%%%%%%%%%%%

Let $\ua=\{\a_0,\dots,\a_k\}\in \Ao(S,X)$ be a proper arc system
and let $\s_0,\s_1,\s_\infty$ the corresponding operators on the
set of oriented arcs $E(\ua)$.
The {\it ribbon graph dual to $\ua$} is $\GG_{\ua}=(E(\ua),\s_0,\s_1)$,
which comes naturally equipped with an $X$-marking
(see~\ref{sss:sigma-arc}).

Define the $(S,X)$-marking
$f:S\rar |\GG|$ in the following way. Fix a point $c_v$ in each
component $v$ of $S\setminus \ua$ (which must be exactly the
marked point, if the component is a pointed disc) and let
$f$ send it to the corresponding vertex $v$ of $|\GG|$.
For each arc $\a_i\in\ua$, consider a transverse path
$\beta_i$
from $c_{v'}$ to $c_{v''}$
that joins the two components $v'$ and $v''$ separated by $\a_i$,
intersecting $\a_i$ exactly once, in such a way that 
$\beta_i\cap \beta_j=\emptyset$ if $i\neq j$.
Define $f$ to be a homeomorphism of $\beta_i$ onto
the oriented edge in $|\GG|$ corresponding to $\a_i$ that runs
from $v'$ to $v''$.

\begin{center}
\begin{figurehere}
\psfrag{x1}{$x_1$}
\psfrag{x2}{$x_2$}
\psfrag{S}{$S$}
\psfrag{c1}{{\color{Blue}$c_{v'}$}}
\psfrag{c2}{{\color{Blue}$c_{v''}$}}
\psfrag{ai}{{\color{Red}$\a_i$}}
\psfrag{bi}{{\color{Blue}$\beta_i$}}
\includegraphics[width=0.7\textwidth]{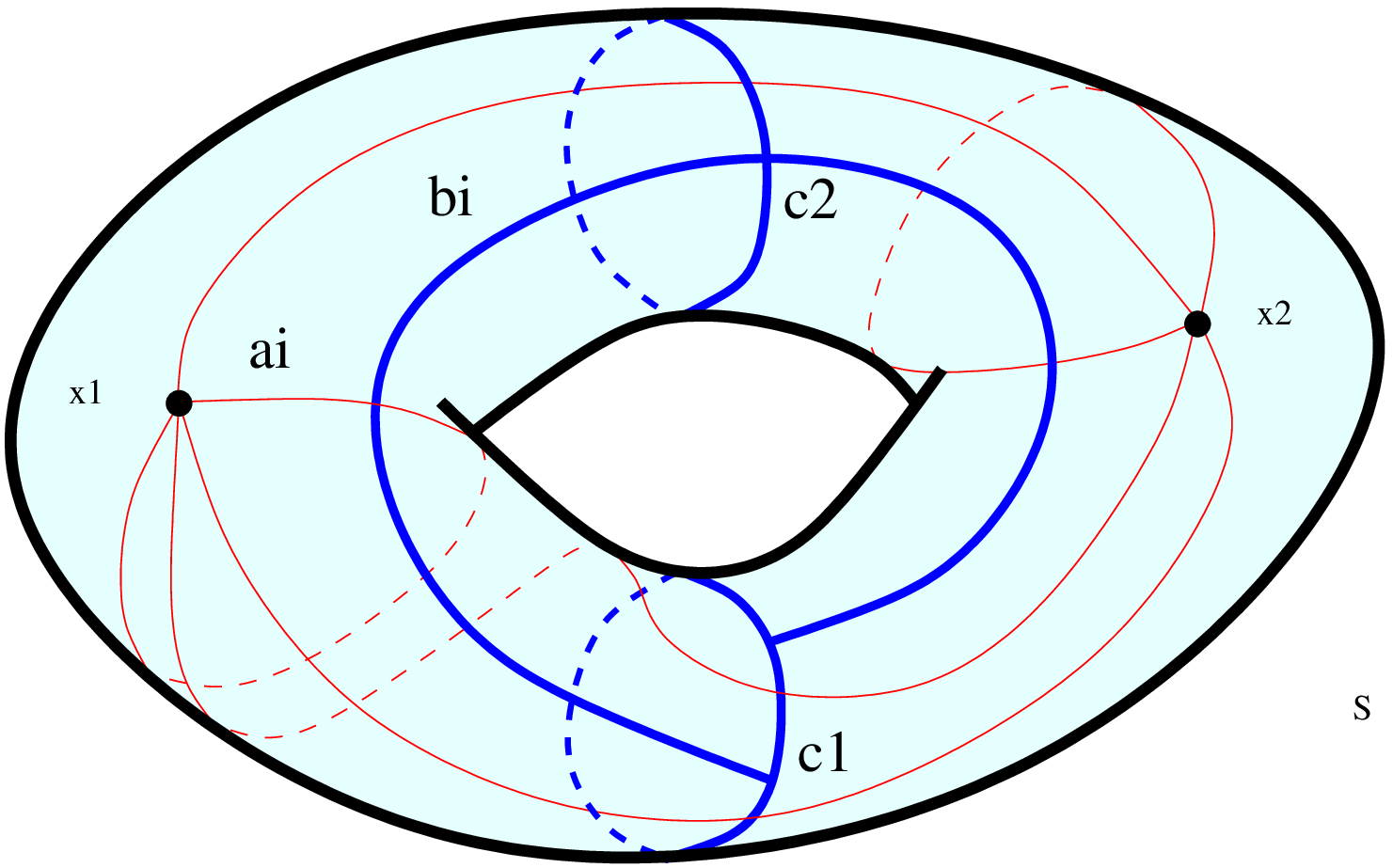}
\caption{Thick curves represent $f^{-1}(|G|)$ and thin ones
their dual arcs.}
\end{figurehere}
\end{center}

Because all components of $S\setminus\ua$
are discs (or pointed discs), it is easy to see that there is
a unique way of extending $f$ to a homeomorphism (up to isotopy).

\begin{proposition}
The association above defines
a $\G(S,X)$-equivariant equivalence of categories
\[
\wh{\Ao}(S,X) \lra \mathfrak{RG}_{ns}(S,X)
\]
where $\wh{\Ao}(S,X)$ is the category of proper arc systems,
whose morphisms are reversed inclusions.
\end{proposition}

In fact, an inclusion $\ua\hra \ub$ of proper systems induces a
morphism $\GG_{\ub}\rar \GG_{\ua}$ of nonsingular $(S,X)$-marked
ribbon graphs.

A pseudo-inverse is constructed as follows.
Let $f:S\rar |\GG|$ be a nonsingular
$(S,X)$-marked ribbon graph and let $|G|\hra |\GG|$
be the spine. The graph $f^{-1}(|G|)$ decomposes $S$ into a disjoint
union of one-pointed discs. For each edge $e$ of $|G|$, let $\a_e$
be the simple arc joining the points in the two discs separated by $e$.
Thus, we can associate the system of arcs $\{\a_e\,|\, e\in E_1(\GG)\}$
to $(\GG,f)$ and this defines a pseudo-inverse
$\mathfrak{RG}_{ns}(S,X) \lra \wh{\Ao}(S,X)$.

%%%%%%%%%%%%%%%%%%%%%%%%%%%%%%%%%%%%%%%%%%%%%%%%%%%%%%%%%%%%%%%%%%%%
\subsubsection{Metrized nonsingular ribbon graphs.}\label{sss:metrized}
%%%%%%%%%%%%%%%%%%%%%%%%%%%%%%%%%%%%%%%%%%%%%%%%%%%%%%%%%%%%%%%%%%%%%

A {\it metric} on a ribbon graph $\GG$ is a map
$\ell:E_1(\GG)\rar \R_+$.
Given a simple closed curve $\g\in\mathcal{C}(S,X)$
and an $(S,X)$-marked nonsingular ribbon graph $f:S\rar|\GG|$,
there is a unique simple closed
curve $\tilde{\g}=|e_{i_1}|\cup\dots\cup |e_{i_k}|$
contained inside $|G|\subset |\GG|$
such that $f^{-1}(\tilde{\g})$ is freely homotopic to $\g$.

If $\GG$ is metrized, then we can define the {\it length}
$\ell(\g)$ to be $\ell(\tilde{\g})=\ell(e_{i_1})+\dots+\ell(e_{i_k})$.
Consequently, the {\it systol} is given by
$\mathrm{inf}\{\ell(\g)\,|\,\g\in\mathcal{C}(S,X)\}$.

Given a proper weighted arc system $w\in|\Ao(S,X)|_{\R}$,
supported on $\ua\in \Ao(S,X)$, we can endow the corresponding
ribbon graph $\GG_{\ua}$ with a {\it metric}, by simply setting
$\ell(\a_i)=w(\a_i)$. Thus, one can extend the
correspondence to proper weighted arc systems and
metrized $(S,X)$-marked nonsingular ribbon graphs.
Moreover, the notions of length and systol agree with
those given in \ref{sss:compact}.

Notice the similarity between Lemma~\ref{lemma:compact}
and Mumford-Mahler criterion for compactness in $\M_{g,n}$.

%%%%%%%%%%%%%%%%%%%%%%%%%%%%%%%%%%%%%%%%%%%%%%%%%%%%%%%%%%%%%%%%%%%%%
\subsubsection{Category of enriched ribbon graphs.}
%%%%%%%%%%%%%%%%%%%%%%%%%%%%%%%%%%%%%%%%%%%%%%%%%%%%%%%%%%%%%%%%%%%%%

An {\it isomorphism} of enriched $X$-marked
ribbon graphs $\GG^{en}_1 \rar \GG^{en}_2$
is the datum of compatible
isomorphisms of their (reduced) graphs $c:\zeta_1\rar \zeta_2$
and of the ribbon graphs $\GG_1\rar \GG_2$,
such that $c(V_{1,+})=V_{2,+}$
and they respect the rest of the data.

Let $\GG^{en}$ be an enriched $X$-marked ribbon graph
and let $e\in E_1(\GG_v)$, where $v\in V_+$.
Assume that $|V_+|>1$ or that $|E_1(\GG_v)|>1$.
We define $\GG^{en}/e$ in the following way.
\begin{itemize}
\item[(a)]
If $e$ is the only edge of $\GG_v$, then we
just turn $v$ into an invisible component
and we define $g(v):=g(\GG_v)$ and $m(x_i)=v$
for all $x_i\in X$ that marked a hole or a vertex
of $\GG_v$.
In what follows, suppose that $|E_1(\GG_v)|>1$.
\item[(b)]
If $[\ora{e}]_0$ and $[\ola{e}]_0$ are distinct and not both special,
then we obtain $\GG^{en}/e$ from $\GG^{en}$ by simply replacing
$\GG_v$ by $\GG_v/e$.
\item[(c)]
If $[\ora{e}]_0=[\ola{e}]_0$ not special, then replace
$\GG_v$ by $\GG_v/e$. If $\{\ora{e}\}$ was a hole marked
by $x_j$, then mark the new vertex of $\GG_v/e$ by $x_j$.
Otherwise, add an edge to $\zeta$
that joins the two new vertices of $\GG_v/e$
(which may or may not split into two visible components).
\item[(d)]
In this last case, add a new invisible component $w$ of
genus $0$ to $\zeta$,
replace $\GG_v$ by $\GG_v/e$
(if $\GG_v/e$ is disconnected, the vertex $v$ splits)
and join $w$ to the new vertices
(one or two) of $\GG_v/e$ and to the old edges
$s_v^{-1}([\ora{e}]_0)\cup s_v^{-1}([\ola{e}]_0)$.
Moreover, if $\{\ora{e}\}$ was a hole
marked by $x_j$, then mark $w$ by $x_j$.
\end{itemize}
Notice that $\GG^{en}/e$ can be not reduced, so we may want
to consider the reduced enriched graph
$\widetilde{\GG^{en}/e}$ associated to it. 
We define $\GG^{en}\rar \widetilde{\GG^{en}/e}$ to be an
{\it elementary contraction}.

$X$-marked enriched ribbon graphs form a (small) category
$\mathfrak{RG}_X$, whose morphisms are compositions of
isomorphisms and elementary contractions.
Call $\mathfrak{RG}_{g,X}$ the full subcategory of
$\mathfrak{RG}_X$ whose objects are ribbon graphs of genus $g$.

\begin{remark}
Really, the automorphism group of an enriched ribbon graph
must be defined as the product of the automorphism group
as defined above by
$\dis\prod_{v\in V_-} \Aut(v)$, where $\Aut(v)$ is the group
of automorphisms of the generic Riemann surface of type
$(g(v),n(v))$
(where $n(v)$ is the number of oriented edges of $\zeta$ outgoing
from $v$). Fortunately, $\Aut(v)$ is almost always trivial,
except if $g(v)=n(v)=1$, when $\Aut(v)\cong \Z/2\Z$.
\end{remark}

%%%%%%%%%%%%%%%%%%%%%%%%%%%%%%%%%%%%%%%%%%%%%%%%%%%%%%%%%%%%%%%%%%%%%
\subsubsection{Topological realization of enriched ribbon graphs.}
%%%%%%%%%%%%%%%%%%%%%%%%%%%%%%%%%%%%%%%%%%%%%%%%%%%%%%%%%%%%%%%%%%%%%

The {\it topological realization} of the
enriched $X$-marked ribbon graph $\GG^{en}$ is
the nodal $X$-marked oriented surface $|\GG^{en}|$
obtained as a quotient of
\[
\left(\coprod_{v\in V_+} |\GG_v|\right)
\coprod \left( \coprod_{w\in V_-} S_w \right)
\]
by a suitable equivalence relation, where
$S_w$ is 
a compact oriented surface of genus $g(w)$
with marked points given by
$m^{-1}(w)$ and by the oriented
edges of $\zeta$ outgoing from $w$.
The equivalence relation identifies couples
of points (two special vertices of $\GG$
or a special vertex on a visible component
and a point on an invisible one) corresponding
to the same edge of $\zeta$.

As in the nonsingular case, for each $v\in V_+$
the positive component $|\GG_v|$ naturally
contains an embedded {\it spine} $|G_v|$.
Notice that there is an obvious correspondence
between edges of $\zeta$ and nodes of $|\GG^{en}|$.

Moreover, the elementary contraction
$\GG^{en}\rar \GG^{en}/e$ to the non-reduced $\GG^{en}/e$
defines a unique homotopy class of maps
$|\GG^{en}|\rar |\GG^{en}/e|$, which may shrink a circle
inside a positive component of $|\GG^{en}|$ to a point
(only in cases (c) and (d)), and which are homeomorphisms elsewhere.

If $\widetilde{\GG^{en}/e}$ is the reduced graph associated
to $\GG^{en}/e$, then we also have a
map $|\widetilde{\GG^{en}/e}|\rar |\GG^{en}/e|$
that shrinks some circles inside the invisible components
to points and is a homeomorphism elsewhere.

\[
\xymatrix@R=0.3in{
|\GG^{en}| \ar[rd] && |\widetilde{\GG^{en}/e}| \ar[ld] \\
& |\GG^{en}/e|
}
\]

%%%%%%%%%%%%%%%%%%%%%%%%%%%%%%%%%%%%%%%%%%%%%%%%%%%%%%%%%%%%%%%%%%%%%
\subsubsection{$(S,X)$-markings of $\GG^{en}$.}
%%%%%%%%%%%%%%%%%%%%%%%%%%%%%%%%%%%%%%%%%%%%%%%%%%%%%%%%%%%%%%%%%%%%%

An {\it $(S,X)$-marking} of an enriched $X$-marked
ribbon graph $\GG^{en}$
is a map
$f:S\rar |\GG^{en}|$ compatible with $X\hra E_\infty(\GG)
\cup E_0(\GG)$
such that
$f^{-1}(\{\text{nodes}\})$ is a disjoint union of circles
and $f$ is an orientation-preserving homeomorphism elsewhere.
The subsurface $S_+:=f^{-1}(|\GG|\setminus
\{\text{special points}\})$ is the {\it visible subsurface}.

An {\it isomorphism} of $(S,X)$-marked (reduced)
enriched ribbon graphs
is an isomorphism $\GG^{en}_1\rar \GG^{en}_2$ such that
$S\arr{f_1}{\lra}|\GG^{en}_1|\rar |\GG^{en}_2|$ is homotopic
to $f_2:S\rar |\GG^{en}_2|$.

Given $(S,X)$-markings $f:S\rar |\GG^{en}|$ and
$f':S\rar |\widetilde{\GG^{en}/e}|$ such that
$S\arr{f}{\lra} |\GG^{en}|\rar |\GG^{en}/e|$
is homotopic to $S\arr{f'}{\lra} |\widetilde{\GG^{en}/e}|
\rar |\GG^{en}/e|$, then we define
$(\GG^{en},f)\rar (\widetilde{\GG^{en}/e},f')$ to be an
{\it elementary contraction} of $(S,X)$-marked enriched
ribbon graphs.

Define $\mathfrak{RG}(S,X)$ to be the category whose objects are
(equivalence classes of) $(S,X)$-marked enriched
ribbon graphs $(\GG^{en},f)$ and whose morphisms
are compositions of isomorphisms and elementary contractions.

Again, the mapping class group $\G(S,X)$ acts
on $\mathfrak{RG}(S,X)$ and the quotient $\mathfrak{RG}(S,X)/
\G(S,X)$ is equivalent to $\mathfrak{RG}_{g,X}$.

%%%%%%%%%%%%%%%%%%%%%%%%%%%%%%%%%%%%%%%%%%%%%%%%%%%%%%%%%%%%%%%%%%%%%
\subsubsection{Arcs/graph duality.}\label{sss:arc-graph}
%%%%%%%%%%%%%%%%%%%%%%%%%%%%%%%%%%%%%%%%%%%%%%%%%%%%%%%%%%%%%%%%%%%%%

Let $\ua=\{\a_0,\dots,\a_k\}\in \Ao(S,X)$ be an arc system
and let $\s_0,\s_1,\s_\infty$ the corresponding operators on the
set of oriented arcs $E(\ua)$.

Define $V_+$ to be the set of connected components of $S(\ua)_+$
and $V_-$ the set of components of $S(\ua)_-$.
Let $\zeta$ be a graph whose vertices are $V=V_+\cup V_-$ and
whose edges correspond to connected components
of $S\setminus(S(\ua)_+\cup
S(\ua)_-)$, where an edge connects $v$ and $w$
(possibly $v=w$) if the associated component bounds $v$ and $w$.

Define $g:V_-\rar\N$ to be the genus function associated to
the connected components of $S(\ua)_-$.

Call $S_v$ the subsurface associated to $v\in V_+$
and let $\hat{S}_v$ be the quotient of $\ol{S}_v$ obtained by
identifying each component of $\pa S_v$ to a point.
As $\ua\cap \hat{S}_v$ quasi-fills $\hat{S}_v$, we can
construct a dual ribbon graph $\GG_v$ and a homeomorphism
$\hat{S}_v\rar |\GG_v|$ that sends $\pa S_v$ to special
vertices of $|\GG_v|$
and marked points on $\hat{S}_v$ to centers of $|\GG_v|$.
These homeomorphisms glue to give a map $S\rar |\GG^{en}|$
that shrinks circles and cylinders in
$S\setminus (S(\ua)_+\cup S(\ua)_-)$ to nodes and is
a homeomorphism elsewhere, which is thus homotopic to
a marking of $|\GG^{en}|$.

We have obtain an enriched $(S,X)$-marked
(reduced) ribbon graph $\GG^{en}_{\ua}$ {\it dual to $\ua$}.

\begin{proposition}
The construction above defines a $\G(S,X)$-equivariant
equivalence of categories
\[
\wh{\Af}(S,X) \lra \mathfrak{RG}(S,X)
\]
where $\wh{\Af}(S,X)$ is the category of proper arc systems,
whose morphisms are reversed inclusions.
\end{proposition}

As before,
an inclusion $\ua\hra \ub$ of systems of arcs induces a
morphism $\GG^{en}_{\ub}\rar \GG^{en}_{\ua}$
of nonsingular $(S,X)$-marked enriched
ribbon graphs.

To construct a pseudo-inverse, start with $(\GG^{en},f)$
and call $\hat{S}_v$ the surface obtained from
$f^{-1}(|\GG_v|)$ by shrinking each boundary circle
to a point. By nonsingular duality,
we can construct a system of arcs $\ua_v$ inside $\hat{S}_v$ dual to
$f_v:\hat{S}_v \rar |\GG_v|$. As the arcs miss the vertices
of $f_v^{-1}(|G_v|)$ by construction, $\ua_v$ can
be lifted to $S$.
The wanted arc system on $S$ is
$\ua=\bigcup_{v\in V_+}\ua_v$.

%%%%%%%%%%%%%%%%%%%%%%%%%%%%%%%%%%%%%%%%%%%%%%%%%%%%%%%%%%%%%%%%%%%%%%%%%
\subsubsection{Metrized enriched ribbon graphs.}
%%%%%%%%%%%%%%%%%%%%%%%%%%%%%%%%%%%%%%%%%%%%%%%%%%%%%%%%%%%%%%%%%%%%%%%%%
A metric on $\GG^{en}$ is a map $\ell:E_1(\GG)\rar\R_+$.
Given $\g\in\mathcal{C}(S,X)$ and an $(S,X)$-marking
$f:S\rar |\GG^{en}|$, we can define $\g_+:=\g\cap S_+$.
As in the nonsingular case, there is a unique $\tilde{\g}_+
=|e_{i_1}|\cup\dots\cup |e_{i_k}|$ inside $|G|\subset |\GG|$
such that $f^{-1}(\tilde{\g}_+)\simeq \g_+$.

Hence, we can define $\ell(\g):=\ell(\g_+)=\ell(e_{i_1})+
\dots+\ell(e_{i_k})$. Clearly, $\ell(\g)=i(\g,w)$, where $w$
is the weight function supported on the arc system dual
to $(\GG^{en},f)$.
Thus, the arc-graph duality also establishes a correspondence
between weighted arc systems on $(S,X)$
and metrized $(S,X)$-marked enriched ribbon graphs.

%%%%%%%%%%%%%%%%%%%%%%%%%%%%%%%%%%%%%%%%%%%%%%%%%%%%%%%%%%%%%%%%%%%%%%%%%
%%%%%%%%%%%%%%%%%%%%%%%%%%%%%%%%%%%%%%%%%%%%%%%%%%%%%%%%%%%%%%%%%%%%%%%%%
\section{Differential and algebro-geometric point of view}\label{sec:alg}
%%%%%%%%%%%%%%%%%%%%%%%%%%%%%%%%%%%%%%%%%%%%%%%%%%%%%%%%%%%%%%%%%%%%%%%%%
\subsection{The Deligne-Mumford moduli space.}\label{ss:deligne-mumford}
%%%%%%%%%%%%%%%%%%%%%%%%%%%%%%%%%%%%%%%%%%%%%%%%%%%%%%%%%%%%%%%%%%%%%%%%%
\subsubsection{The Teichm\"uller space.}\label{sss:teichmueller}
%%%%%%%%%%%%%%%%%%%%%%%%%%%%%%%%%%%%%%%%%%%%%%%%%%%%%%%%%%%%%%%%%%%%%%%%%

Fix a compact oriented surface $S$ of genus $g$
and a subset $X=\{x_1,\dots,x_n\}\subset S$ such that $2g-2+n>0$.

A {\it smooth family} of $(S,X)$-marked Riemann surfaces is
a commutative diagram
\[
\xymatrix@R=0.3in{
B \times S \ar[rr]^f \ar[rrd] && \mathcal{C} \ar[d]^\pi \\
&& B
}
\]
where $f$ is an relatively (over $B$) oriented diffeomorphism,
$B\times S \rar B$ is the projection on the first factor and
the fibers $\mathcal{C}_b$
of $\pi$ are Riemann surfaces, whose complex
structure varies smoothly with $b\in B$.

Two families $(f_1,\pi_1)$ and $(f_2,\pi_2)$ over $B$
are {\it isomorphic} if there exists a continuous
map $h:\mathcal{C}_1 \rar \mathcal{C}_2$ such that
\begin{itemize}
\item
$h_b\circ f_{1,b}:S \rar \mathcal{C}_{2,b}$ is homotopic to $f_{2,b}$
for every $b\in B$
\item
$h_b:\mathcal{C}_{1,b}\rar\mathcal{C}_{2,b}$ is biholomorphic
for every $b\in B$.
\end{itemize}

The functor $\Teich(S,X):(\text{manifolds})\rar (\text{sets})$
defined by
\[
B \mapsto \left\{\substack{\dis\text{smooth families of $(S,X)$-marked}\\
\dis\text{Riemann surfaces over $B$}}\right\}/\text{iso}
\]
is represented by the
{\it Teichm\"uller space} $\Teich(S,X)$.

It is a classical result
that $\Teich(S,X)$ is a complex-analytic
manifold of (complex) dimension $3g-3+n$
(Ahlfors \cite{ahlfors:complex}, Bers
\cite{bers:simultaneous} and Ahlfors-Bers
\cite{ahlfors-bers:riemann})
and is diffeomorphic to a ball
(Teichm\"uller \cite{teichmueller:collected}).

%%%%%%%%%%%%%%%%%%%%%%%%%%%%%%%%%%%%%%%%%%%%%%%%%%%%%%%%%%%%%%%%%%%%%%%%%
\subsubsection{The moduli space of Riemann surfaces.}\label{sss:moduli}
%%%%%%%%%%%%%%%%%%%%%%%%%%%%%%%%%%%%%%%%%%%%%%%%%%%%%%%%%%%%%%%%%%%%%%%%%

A {\it smooth family} of $X$-marked Riemann surfaces of genus $g$ is
\begin{itemize}
\item
a submersion $\pi:\mathcal{C}\rar B$
\item
a smooth embedding $s:X\times B\rar \mathcal{C}$
\end{itemize}
such that the fibers $\mathcal{C}_b$
are Riemann surfaces of genus $g$, whose
complex structure varies smoothly in $b\in B$,
and $s_{x_i}:B\rar\mathcal{C}$
is a section for every $x_i\in X$.

Two families $(\pi_1,s_1)$ and $(\pi_2,s_2)$ over $B$
are isomorphic if
there exists a diffeomorphism $h:\mathcal{C}_1\rar\mathcal{C}_2$
such that $\pi_2\circ h=\pi_1$, the restriction of $h$ to each
fiber $h_b:\mathcal{C}_{1,b}\rar\mathcal{C}_{2,b}$ is a biholomorphism
and $h\circ s_1=s_2$.

The existence of Riemann surfaces with nontrivial automorphisms
(for $g\geq 1$) prevents the functor
\[
\M_{g,X}:
\xymatrix@R=0in{
(\text{manifolds}) \ar[r] & (\text{sets}) \\
B \ar@{|->}[r] & \left\{\substack{\dis
\text{smooth families of $X$-marked}\\
\dis\text{Riemann surfaces over $B$}}\right\}/\text{iso}
}
\]
from being representable. However, Riemann surfaces
with $2g-2+n>0$
have finitely many automorphisms and so $\M_{g,X}$ is actually
represented by an orbifold, which is in fact $\Teich(S,X)/\G(S,X)$
(in the orbifold sense). In the algebraic category, we would rather
say that $\M_{g,X}$ is a Deligne-Mumford stack with quasi-projective
coarse space.

%%%%%%%%%%%%%%%%%%%%%%%%%%%%%%%%%%%%%%%%%%%%%%%%%%%%%%%%%%%%%%%%%%%%%%%%%
\subsubsection{Stable curves.}
%%%%%%%%%%%%%%%%%%%%%%%%%%%%%%%%%%%%%%%%%%%%%%%%%%%%%%%%%%%%%%%%%%%%%%%%%

Enumerative geometry is traditionally reduced to intersection
theory on suitable moduli spaces. In our case, $\M_{g,X}$ is not a
compact orbifold. To compactify it in an algebraically meaningful way,
we need to look at how algebraic families of complex projective
curves can degenerate.

In particular, given a holomorphic family $\mathcal{C}^*\rar \Delta^*$
of algebraic curves over the punctured disc,
we must understand how to complete
the family over $\Delta$.

\begin{example}
Consider the family $\mathcal{C}^*=\{(b,[x:y:z])\in\Delta^*\times\C
\mathbb{P}^2\,|\,
y^2 z=x(x-bz)(x-2z)\}$ of curves of genus $1$ with the marked point
$[2:0:1]\in \C\mathbb{P}^2$, parametrized by $b\in\Delta^*$.
Notice that the projection $\mathcal{C}^*_b\rar \C\mathbb{P}^1$
given by $[x:y:z]\mapsto [x:z]$
(where $[0:1:0]\mapsto [1:0]$)
is a $2:1$ cover, branched
over $\{0,b,2,\infty\}$.
Fix a $\ol{b}\in\Delta^*$ and
consider a closed curve $\g\subset \C\mathbb{P}^1$
that separates $\{\ol{b},2\}$
from $\{0,\infty\}$ and pick one of the two (simple
closed) lifts $\tilde{\g}\subset \mathcal{C}^*_{\ol{b}}$.

This $\tilde{\g}$ determines a nontrivial
element of $H_1(\mathcal{C}^*_{\ol{b}})$.
A quick analysis tells us that
the endomorphism
$T:H_1(\mathcal{C}^*_{\ol{b}})\rar H_1(\mathcal{C}^*_{\ol{b}})$
induced by the monodromy
around a generator of $\pi_1(\Delta^*,\ol{b})$ is nontrivial.
Thus, the family $\mathcal{C}^*\rar\Delta^*$ cannot be completed
over $\Delta$ as smooth family (because it would have
trivial monodromy).
\end{example}

If we want to compactify our moduli space, we must allow
our curves to acquire some singularities.
Thus, it makes no longer sense
to ask them to be submersions. Instead, we will 
require them to be {\it flat}.

Given an open subset $0\in B\subset \C$,
a flat family of connected projective
curves $\mathcal{C}\rar B$
may typically look like (up to shrinking $B$)
\begin{itemize}
\item
$\Delta\times B\rar B$ around a smooth point
of $\mathcal{C}_0$
\item
$\{(x,y)\in\C^2\,|\, xy=0\}\times B\rar B$ around
a node of $\mathcal{C}_0$
that persists on each $\mathcal{C}_b$
\item
$\{(b,x,y)\in B\times\C^2\,|\, xy=b\}\rar B$
around a node of $\mathcal{C}_0$
that does not
persist on the other curves $\mathcal{C}_b$
with $b\neq 0$
\end{itemize}
in local analytic coordinates.

Notice that the (arithmetic) genus of each fiber
$g_b=1-\frac{1}{2}[\chi(\mathcal{C}_b)-\nu_b]$
is constant in $b$.

To prove that allowing nodal curves is enough to
compactify $\M_{g,X}$, one must show
that it is always possible to complete any family $\mathcal{C}^*\rar
\Delta^*$ to a family over $\Delta$.
However, because nodal curves may have nontrivial automorphisms,
we shall consider also the case in which $0\in\Delta$ is an orbifold
point. Thus, it is sufficient to be able to complete
not exactly the family $\mathcal{C}^*\rar\Delta^*$ but
its pull-back under a suitable map $\Delta^*\rar\Delta^*$
given by $z\mapsto z^k$. This is exactly the {\it semi-stable
reduction theorem}.

One can observe that it is always possible to avoid producing
genus $0$ components with $1$ or $2$ nodes. Thus, we can consider
only {\it stable curves}, that is nodal projective (connected)
curves such that all irreducible components have finitely
many automorphisms (equivalently, no irreducible
component is a sphere
with less than three nodes/marked points).

The {\it Deligne-Mumford compactification} $\Mbar_{g,X}$
of $\M_{g,X}$ is the moduli space of $X$-marked stable curves
of genus $g$, which is 
a compact orbifold (algebraically, a Deligne-Mumford
stack with projective coarse moduli space).

Its underlying topological space is a projective variety
of complex dimension $3g-3+n$.

%%%%%%%%%%%%%%%%%%%%%%%%%%%%%%%%%%%%%%%%%%%%%%%%%%%%%%%%%%%%%%%%%%%%%%%%
\subsection{The system of moduli spaces of curves}\label{ss:system}
%%%%%%%%%%%%%%%%%%%%%%%%%%%%%%%%%%%%%%%%%%%%%%%%%%%%%%%%%%%%%%%%%
\subsubsection{Boundary maps.}
%%%%%%%%%%%%%%%%%%%%%%%%%%%%%%%%%%%%%%%%%%%%%%%%%%%%%%%%%%%%%%%%%

Many facts suggest that one should not look at the moduli
spaces of $X$-pointed genus $g$
curves $\Mbar_{g,X}$ each one separately, but one
must consider the whole system $(\Mbar_{g,X})_{g,X}$.
An evidence is given by the existence
of three families of maps that relate different moduli spaces.
\begin{enumerate}
\item
The {\it forgetful map} is a
projective flat morphism
\[
\pi_q: \Mbar_{g,X\cup\{q\}} \lra \Mbar_{g,X}
\]
that forgets the point $q$ and stabilizes the curve
(i.e. contracts a possible two-pointed sphere).
This map can be identified to
the universal family and so is endowed with
natural sections
\[
\vartheta_{0,\{x_i,q\}}: \Mbar_{g,X} \rar \Mbar_{g,X\cup\{q\}}
\]
for all $x_i\in X$.
\item
The {\it boundary map} corresponding to irreducible curves is
the finite map
\[
\vartheta_{irr}: \Mbar_{g-1,X\cup\{x',x''\}} \lra \Mbar_{g,X}
\]
(defined for $g>0$) that glues $x'$ and $x''$ together. It is
generically $2:1$ and
its image sits in the boundary of $\Mbar_{g,X}$.
\item
The {\it boundary maps} corresponding to reducible curves are
the finite maps
\[
\vartheta_{g',I}:\Mbar_{g',I\cup\{x'\}} \times \Mbar_{g-g',I^c\cup\{x''\}}
\lra \Mbar_{g,X}
\]
(defined for every $0\leq g'\leq g$ and $I\subseteq X$
such that the spaces involved are nonempty)
that take two curves and glue them together identifying $x'$ and $x''$.
They are generically $1-1$ (except in the case
$g=2g'$ and $X=\emptyset$, when
the map is generically $2:1$) and
their images sit in the boundary of $\Mbar_{g,X}$ too.
\end{enumerate}
Let $\delta_{0,\{x_i,q\}}$ be the Cartier divisor in $\Mbar_{g,X\cup\{q\}}$
corresponding to the image of the tautological section
$\vartheta_{0,\{x_i,q\}}$ and call $D_q:=\sum_i \delta_{0,\{x_i,q\}}$.
%
%%%%%%%%%%%%%%%%%%%%%%%%%%%%%%%%%%%%%%%%%%%%%%%%%%%%%%%%%%%%%%%%%
\subsubsection{Stratification by topological type.}
%%%%%%%%%%%%%%%%%%%%%%%%%%%%%%%%%%%%%%%%%%%%%%%%%%%%%%%%%%%%%%%%

We observe that $\Mbar_{g,X}$ has a natural {\it stratification}
by topological type of the complex curve. In fact, we can
attach to every stable curve $\Si$ its
{\it dual graph} $\zeta_{\Si}$, whose vertices $V$
correspond to irreducible components
and whose edges correspond to nodes of $\Si$.
Moreover, we can define a genus function $g:V\rar \N$
such that $g(v)$ is the genus of the normalization
of the irreducible component corresponding to $v$
and a marking function $m:X \rar V$ (determined
by requiring that $x_i$ is marking a point on the
irreducible component corresponding to $m(x_i)$).
Equivalently, we will also say that
the vertex $v\in V$ is {\it labeled} by
$(g(v),X_v:=m^{-1}(v))$.
Call $Q_v$ the singular points of $\Si_v$.

For every such labeled graph $\zeta$, we can construct a
boundary map
\[
\vartheta_{\zeta}:\prod_{vin V} \Mbar_{g_v,X_v\cup Q_v} \lra \Mbar_{g,X}
\]
which is a finite morphism.

%%%%%%%%%%%%%%%%%%%%%%%%%%%%%%%%%%%%%%%%%%%%%%%%%%%%%%%%%%%%%%%%%%%%%
\subsection{Augmented Teichm\"uller space}\label{ss:augmented}
%%%%%%%%%%%%%%%%%%%%%%%%%%%%%%%%%%%%%%%%%%%%%%%%%%%%%%%%%%%%%%%%%%%%%
\subsubsection{Bordifications of $\Teich(S,X)$.}
%%%%%%%%%%%%%%%%%%%%%%%%%%%%%%%%%%%%%%%%%%%%%%%%%%%%%%%%%%%%%%%%%%%%%
Fix $S$ a compact oriented surface of genus $g$ and let
$X=\{x_1,\dots,x_n\}\subset S$ such that $2g-2+n>0$.

It is natural to look for natural
{\it bordifications} of $\Teich(S,X)$: that is, we look for a space
$\ol{\Teich}(S,X)\supset \Teich(S,X)$ that contains $\Teich(S,X)$
as a dense subspace and such that
the action of the mapping class group
$\G(S,X)$ extends to $\ol{\Teich}(S,X)$.

A remarkable example is given by Thurston's compactification
$\ol{\Teich}^{Th}(S,X)=\Teich(S,X)\cup \mathbb{P}\mathcal{ML}(S,X)$,
in which points at infinity are projective measured lamination
with compact support in $S\setminus X$.
Thurston showed that $\mathbb{P}\mathcal{ML}(S,X)$ is compact
and homeomorphic to a sphere. As $\G(S,X)$ is infinite and discrete,
this means that the quotient
$\ol{\Teich}^{Th}(S,X)/\G(S,X)$ cannot be too good and so
this does not sound like a convenient way to compactify $\M_{g,X}$.

We will see in Section~\ref{sec:triangulation} that
$\Teich(S,X)$ can be identified to $|\Ao(S,X)|$.
Thus, another remarkable example will be given by $|\Af(S,X)|$.

A natural question is how to define a bordification
$\ol{\Teich}(S,X)$ such that $\ol{\Teich}(S,X)/\G(S,X)\cong\Mbar_{g,X}$.

%%%%%%%%%%%%%%%%%%%%%%%%%%%%%%%%%%%%%%%%%%%%%%%%%%%%%%%%%%%%%%%%%%%%%
\subsubsection{Deligne-Mumford augmentation.}
%%%%%%%%%%%%%%%%%%%%%%%%%%%%%%%%%%%%%%%%%%%%%%%%%%%%%%%%%%%%%%%%%%%%%

A {\it (continuous)
family of stable $(S,X)$-marked curves} is a diagram
\[
\xymatrix@R=0.3in{
B \times S \ar[rr]^f \ar[rrd] && \mathcal{C} \ar[d]^\pi \\
&& B
}
\]
where $B\times S \rar B$ is the projection on the first factor
and
\begin{itemize}
\item
the family $\pi$ is obtained as a pull-back of
a flat stable family of $X$-marked curves
$\mathcal{C}'\rar B'$ through a continuous map
$B\rar B'$
\item
if $N_b\subset \mathcal{C}_b$ is the subset of nodes,
then $f^{-1}(\nu)$ is a smooth loop in $S\times\{b\}$
for every $\nu\in N_b$
\item
for every $b\in B$ the restriction
$f_b:S\setminus f^{-1}(N_b) \rar
\mathcal{C}_b\setminus N_b$ is an orientation-preserving
homeomorphism, compatible with the $X$-marking.
\end{itemize}
Isomorphisms of such families are defined in the obvious way.

\begin{example}
A way to construct such families is to start with
a flat family $\mathcal{C}'\rar\Delta$ such that
$\mathcal{C}'_b$ are all homeomorphic for $b\neq 0$.
Then consider the path $B=[0,\e)\subset \Delta$
and call $\mathcal{C}:=\mathcal{C}'\times_\Delta B$.
Over $(0,\e)$, the family $\mathcal{C}$ is topologically
trivial, whereas $\mathcal{C}_0$ may contain some new nodes.

Consider a marking $S \rar \mathcal{C}_{\e/2}$ that pinches
circles to nodes, is an oriented homeomorphism elsewhere and is
compatible with $X$.
The map $S\times (0,\e)\rar \mathcal{C}_{\e/2}\times(0,\e)
\arr{\sim}{\lra}\mathcal{C}$
extends $S\times [0,\e)\rar \mathrm{Bl}_{\mathcal{C}_0}\mathcal{C}
\rar \mathcal{C}$, which is the wanted $(S,X)$-marking.
\end{example}

The {\it Deligne-Mumford augmentation} of $\Teich(S,X)$
is the topological space $\ol{\Teich}^{DM}(S,X)$ that 
classifies families of stable $(S,X)$-marked curves.

It follows easily that $\ol{\Teich}^{DM}(S,X)/\G(S,X)=
\Mbar_{g,X}$ as topological spaces. However,
$\ol{\Teich}^{DM}(S,X)\rar\Mbar_{g,X}$ has infinite
ramification at $\pa^{DM}\Teich(S,X)$, due to the Dehn
twists around the pinched loops.

%%%%%%%%%%%%%%%%%%%%%%%%%%%%%%%%%%%%%%%%%%%%%%%%%%%%%%%%%%%%%%%%%%%%%
\subsubsection{Hyperbolic length functions.}
%%%%%%%%%%%%%%%%%%%%%%%%%%%%%%%%%%%%%%%%%%%%%%%%%%%%%%%%%%%%%%%%%%%%%

Let $[f:S\rar \Si]$ be a point of $\Teich(S,X)$.
As $\chi(S\setminus X)=2-2g-n<0$, the uniformization
theorem provides a universal cover
$\mathbb{H}\rar \Si\setminus f(X)$, which endows $\Si\setminus f(X)$
with a hyperbolic metric of finite volume, with cusps at $f(X)$.

In fact, we can interpret $\Teich(S,X)$ as the classifying
space of $(S,X)$-marked families of hyperbolic surfaces.
It is clear that continuous variation of the complex structure
corresponds to continuous variation of the hyperbolic metric
(uniformly on the compact subsets, for instance),
and so to continuity
of the holonomy map $H:\pi_1(S\setminus X)\times\Teich(S,X)
\rar \mathrm{PSL}_2(\R)$.

In particular, for every $\g\in\pi_1(S\setminus X)$
the function $\ell_\g:\Teich(S,X)\rar \R$
that associates to $[f:S\rar \Si]$ the length of the unique
geodesic in the free homotopy class $f_*\g$ is continuous.
As $\cosh(\ell_\g/2)=|\mathrm{Tr}(H_{\g}/2)|$, one can check
that $H$ can be reconstructed from sufficiently (but finitely)
many length functions. So that the continuity of these
is equivalent to the continuity of the family.

%%%%%%%%%%%%%%%%%%%%%%%%%%%%%%%%%%%%%%%%%%%%%%%%%%%%%%%%%%%%%%%%%%%%%
\subsubsection{Fenchel-Nielsen coordinates.}
%%%%%%%%%%%%%%%%%%%%%%%%%%%%%%%%%%%%%%%%%%%%%%%%%%%%%%%%%%%%%%%%%%%%%

Let $\ug=\{\g_1,\dots,\g_{N}\}$
be a maximal system of disjoint
simple closed curves of $S\setminus X$
(and so $N=3g-3+n$)
such that no $\g_i$ is contractible in $S\setminus X$
or homotopic to a puncture
and no couple $\g_i,\g_j$ bounds a cylinder contained in $S\setminus X$.

The system $\ug$ induces a {\it pair of pants decomposition} of $S$,
that is $S\setminus(\g_1\cup\dots\cup\g_{N})=P_1\cup P_2\cup\dots\cup
P_{2g-2+n}$,
and each $P_i$ is a pair of pants
(i.e. a surface of genus $0$ with $\chi(P_i)=-1$).

Given $[f:S\rar\Si]\in\Teich(S,X)$, we have {\it lengths}
$\ell_i(f)=\ell_{\g_i}(f)$ for $i=1,\dots,N$, which
determine the hyperbolic type of all pants $P_1,\dots,P_{2g-2+n}$.
The information about how the pants are glued together is
encoded in the {\it twist parameters} $\tau_i=\tau_{\g_i}\in \R$,
which are well-defined up to some choices.
What is important is that, whatever choices we make,
the difference $\tau_i(f_1)-\tau_i(f_2)$ is the same and it is
well-defined.

The {\it Fenchel-Nielsen coordinates} $(\ell_i,\tau_i)_{i=1}^{N}$
exhibit a real-analytic
diffeomorphism $\Teich(S,X)\arr{\sim}{\lra}(\R_+\times\R)^{N}$
(which clearly depends on the choice of $\ug$).

%%%%%%%%%%%%%%%%%%%%%%%%%%%%%%%%%%%%%%%%%%%%%%%%%%%%%%%%%%%%%%%%%%%%%
\subsubsection{Fenchel-Nielsen coordinates around nodal curves.}
%%%%%%%%%%%%%%%%%%%%%%%%%%%%%%%%%%%%%%%%%%%%%%%%%%%%%%%%%%%%%%%%%%%%%

Points of $\pa^{DM}\Teich(S,X)$ are $(S,X)$-marked
stable curves or, equivalently (using the uniformization
theorem componentwise), $(S,X)$-marked
hyperbolic surfaces with nodes, i.e. homotopy classes
of maps $f:S\rar\Si$, where $\Si$ is a hyperbolic surface
with nodes $\nu_1,\dots,\nu_k$, the fiber $f^{-1}(\nu_j)$ is a simple
closed curve $\g_j$ and $f$ is an oriented diffeomorphism
outside the nodes.

Complete $\{\g_1,\dots,\g_k\}$ to a maximal set $\ug$ of simple closed
curves in $(S,X)$ and consider the associated Fenchel-Nielsen
coordinates $(\ell_j,\tau_j)$ on $\Teich(S,X)$.
As we approach the point $[f]$, the holonomies $H_{\g_1},\dots,H_{\g_k}$
tend to parabolics and so
the lengths $\ell_1,\dots,\ell_k$
tend to zero. In fact, the hyperbolic metric on surface $\Si$
has a pair of cusps at each node $\nu_j$.

This shows that the lengths functions $\ell_1,\dots,\ell_k$ extend
to zero at $[f]$ with continuity. On the other hand, the twist
parameters $\tau_1(f),\dots,\tau_k(f)$ make no longer sense.

If we look at what happens on $\Mbar_{g,X}$, we may notice
that the couples $(\ell_j,\tau_j)_{j=1}^k$ behave like polar
coordinate around $[\Si]$, so that is seems natural to set
$\vartheta_m=2\pi\tau_m/\ell_m$
for all $m=1,\dots,N$
and define consequently a map
$F_{\ug}:(\R^2)^{N}\rar \Mbar_{g,X}$,
that associates to $(\ell_1,\vartheta_1,\dots,\ell_N,\vartheta_{N})$
the surface with Fenchel-Nielsen coordinates $(\ell_m,
\tau_m=\ell_m\vartheta_m/2\pi)$.
Notice that the map is well-defined, because a twist along $\g_j$
by $\ell_j$ is a diffeomorphism of the surface (a Dehn twist).

The map $F_{\ug}$ is an orbifold cover
$F_{\ug}:\R^{2N}\rar F_{\ug}(\R^{2N})\subset\Mbar_{g,X}$
and its image contains $[\Si]$. 
Varying $\ug$, we can cover the whole $\Mbar_{g,X}$ and thus
give it a {\it Fenchel-Nielsen smooth structure}.

The bad news, analyzed by Wolpert \cite{wolpert:geometry},
is that the Fenchel-Nielsen smooth structure is different
(at $\pa\M_{g,X}$) from the Deligne-Mumford one.
In fact, if a boundary divisor is locally described by $\{z_1=0\}$,
then the length $\ell_\g$
of the corresponding vanishing geodesic is related
to $z_1$ by $|z_1|\approx \exp(-1/\ell_\g)$, which shows that
the identity map
$\Mbar^{FN}_{g,X}\rar \Mbar^{DM}_{g,X}$ is Lipschitz, but
its inverse it not H\"older-continuous.

%%%%%%%%%%%%%%%%%%%%%%%%%%%%%%%%%%%%%%%%%%%%%%%%%%%%%%%%%%%%%%%%%%%%
\subsubsection{Weil-Petersson metric.}
%%%%%%%%%%%%%%%%%%%%%%%%%%%%%%%%%%%%%%%%%%%%%%%%%%%%%%%%%%%%%%%%%%%%

Let $\Si$ be a Riemann surface of genus $g$ with marked points
$X\hra \Si$ such that $2g-2+n>0$.
First-order deformations of the complex structure
can be rephrased in terms of $\ol{\pa}$ operator
as $\ol{\pa}+\e\mu\pa+o(\e)$, where
the {\it Beltrami differential} $\mu\in \Omega^{0,1}(T_\Si(-X))$
can be locally written as
$\dis \mu(z)\frac{d\ol{z}}{dz}$ with respect to some
holomorphic coordinate $z$ on $\Si$ and $\mu(z)$ vanishes at $X$.

Given a smooth vector field $\dis V=V(z)\frac{\pa}{\pa z}$
on $\Si$ that vanishes at $X$,
the deformations induced by $\mu$ and $\mu+\ol{\pa}V$
differ only by an isotopy of $\Si$ generated by $V$
(which fixes $X$).

Thus, the {\it tangent space} $T_{[\Si]}\M_{g,X}$
can be identified to $H^{0,1}(\Si,T_\Si(-X))$. As a consequence,
the {\it cotangent space} $T^{\vee}_{[\Si]}\M_{g,X}$ identifies
to the space $\mathcal{Q}(\Si,X)$
of integrable
holomorphic quadratic differentials on $\Si\setminus X$,
that is, which are allowed to have a simple pole at each $x_i\in X$.
The duality between $T_{[\Si]}\M_{g,X}$ and $T^{\vee}_{[\Si]}\M_{g,X}$
is given by
\[
\xymatrix@R=0in{
H^{0,1}(\Si,T_{\Si}(-X))\times H^0(\Si,K_{\Si}^{\otimes 2}(X)) \ar[rr] && \C \\
(\mu,\varphi) \ar@{|->}[rr] && \dis\int_{\Si} \mu\varphi
}
\]
If $\Si\setminus X$ is given the hyperbolic metric $\lambda$, then elements
in $H^{0,1}(\Si,T_\Si(-X))$ can be identified to the space of {\it harmonic
Beltrami differentials} $\mathcal{H}(\Si,X)=\{\ol{\varphi}/\lambda\,|
\,\varphi\in\mathcal{Q}(\Si,X)\}$.

The {\it Weil-Petersson Hermitean metric} $h=g+i\omega$
(defined by Weil \cite{weil:onthemoduli} using Petersson's
pairing of modular forms) is
\[
h(\mu,\nu):=\int_{\Si} \mu \ol{\nu}\cdot\lambda
\]
for $\mu,\nu \in \mathcal{H}(\Si,X)\cong T_{\Si}\M_{g,X}$.

This metric has a lot of properties: it is K\"ahler (Weil
\cite{weil:onthemoduli} and  Ahlfors
\cite{ahlfors:someremarks})
and it is mildly divergent at $\pa\M_{g,X}$, so that
the Weil-Petersson distance extends to a non-degenerate
distance on $\Mbar_{g,X}$ and
all points of $\pa\M_{g,X}$ are at finite distance (Masur
\cite{masur:extension} , Wolpert \cite{Wolpert:diameter}).

Because $\Mbar_{g,X}$ is compact and so WP-complete,
the lifting of the Weil-Petersson metric on
to $\ol{\Teich}(S,X)$ is also complete.
Thus, $\ol{\Teich}(S,X)$ can be seen as the
{\it Weil-Petersson completion} of $\Teich(S,X)$.

%%%%%%%%%%%%%%%%%%%%%%%%%%%%%%%%%%%%%%%%%%%%%%%%%%%%%%%%%%%%%%%%%%%%%
\subsubsection{Weil-Petersson form.}
%%%%%%%%%%%%%%%%%%%%%%%%%%%%%%%%%%%%%%%%%%%%%%%%%%%%%%%%%%%%%%%%%%%%%

We should emphasize that the Weil-Petersson symplectic
form $\omega_{WP}$ depends more directly
on the hyperbolic metric on the surface than on its
holomorphic structure.

In particular, Wolpert \cite{wolpert:symplectic}
has shown that
\[
\omega_{WP}=\sum_i d\ell_i\wedge d\tau_i
\]
on $\Teich(S,X)$, where $(\ell_i,\tau_i)$ are Fenchel-Nielsen
coordinates associated to any pair of pants decomposition of $(S,X)$.

On the other hand, if we identify $\Teich(S,X)$ with an
open subset of
$\mathrm{Hom}(\pi_1(S\setminus X),\mathrm{SL}_2(\R))/\mathrm{SL}_2(\R)$,
then points of $\Teich(S,X)$ are associated
$\mathfrak{g}$-local systems $\rho$ on $S\setminus X$ (with parabolic
holonomies at $X$ and hyperbolic holonomies otherwise),
where $\mathfrak{g}=\mathfrak{sl}_2(\R)$ is endowed with
the symmetric bilinear form $\la \a,\b \ra=\mathrm{Tr}(\a\b)$.

Goldman \cite{goldman:symplectic} has proved that,
in this description,
the tangent space to $\Teich(S,X)$ at
$\rho$ is naturally $H^1(S,X;\mathfrak{g})$
and that $\omega_{WP}$ is given by
$\omega(\mu,\nu)=\mathrm{Tr}(\mu\smile \nu)\cap [S]$.

\begin{remark}
Another description of $\omega$ in terms of shear
coordinates and Thurston's symplectic form on measured
laminations is given by Bonahon-S\"ozen
\cite{bonahon-soezen:shear}.
\end{remark}

One can feel that the complex structure $J$
on $\Teich(S,X)$ inevitably shows up whenever
we deal with the Weil-Petersson metric, as
$g(\cdot,\cdot)=\omega(\cdot,J\cdot)$.
On the other hand, the knowledge of
$\omega$ is sufficient to compute
volumes and characteristic classes.

%%%%%%%%%%%%%%%%%%%%%%%%%%%%%%%%%%%%%%%%%%%%%%%%%%%%%%%%%%%%%%%%%%%%%
%%%%%%%%%%%%%%%%%%%%%%%%%%%%%%%%%%%%%%%%%%%%%%%%%%%%%%%%%%%%%%%%%%%%%
\subsection{Tautological classes}\label{ss:tautological}
%%%%%%%%%%%%%%%%%%%%%%%%%%%%%%%%%%%%%%%%%%%%%%%%%%%%%%%%%%%%%%%%%%%%%
\subsubsection{Relative dualizing sheaf.}
%%%%%%%%%%%%%%%%%%%%%%%%%%%%%%%%%%%%%%%%%%%%%%%%%%%%%%%%%%%%%%%%%%%%%

All the maps between moduli spaces
we have defined are in some sense tautological
as they are very naturally constructed and they reflect intrinsic
relations among the various moduli spaces.
It is evident that
one can look at these
as classifying maps to the Deligne-Mumford
stack $\Mbar_{g,X}$ (which obviously descend to maps between
coarse moduli spaces). Hence, we can consider all the cycles
obtained by pushing forward or pulling back via these maps as
being ``tautologically'' defined.

Moreover, there is an ingredient we have
not considered yet: it is the relative dualizing sheaf of
the universal curve $\pi_q:\Mbar_{g,X\cup\{q\}}
\rar \Mbar_{g,X}$. One expects that it carries many
informations and that it can produce many classes of interest.

The relative dualizing sheaf $\omega_{\pi_q}$
is the sheaf on $\Mbar_{g,X\cup\{q\}}$, whose
local sections are (algebraically varying)
Abelian differentials that are allowed to
have simple poles at the nodes, provided
the two residues at each node are opposite.
The local sections of $\omega_{\pi_q}(D_q)$
(the logarithmic variant of $\omega_{\pi_q}$)
are sections of $\omega_{\pi_q}$ 
that may have simple poles at the $X$-marked points.

%%%%%%%%%%%%%%%%%%%%%%%%%%%%%%%%%%%%%%%%%%%%%%%%%%%%%%%%%%%%%%%%%%%%
\subsubsection{MMMAC classes.}
%%%%%%%%%%%%%%%%%%%%%%%%%%%%%%%%%%%%%%%%%%%%%%%%%%%%%%%%%%%%%%%%%%%%

The {\it Miller classes} are
\[
\psi_{x_i}:=c_1(\mathcal{L}_i)
\in CH^1(\Mbar_{g,X})_{\Q}
\]
where $\mathcal{L}_i:=\vartheta_{0,\{x_i,q\}}^*\omega_{\pi_q}$
and the modified (by Arbarello-Cornalba)
{\it Mumford-Morita classes} as
\[
\k_j:=(\pi_q)_*(\psi_q^{j+1})
\in CH^j(\Mbar_{g,X})_{\Q}.
\]
One could moreover define the $l$-th {\it Hodge bundle} as
\mbox{$\mathbb{E}_l:=(\pi_q)_*(\omega_{\pi_q}^{\otimes l})$} and consider
the Chern classes of these bundles (for example, the $\lambda$
classes \mbox{$\l_i:=c_i(\mathbb{E}_1)$}).
However, using Grothendieck-Riemann-Roch,
Mumford \cite{mumford:towards} and Bini \cite{bini:algorithm} proved 
that $c_i(\mathbb{E}_j)$ can be expressed as a linear combination of
Mumford-Morita classes up to elements in the boundary, so that
they do not introduce anything really new.

When there is no risk of ambiguity,
we will denote in the same way the classes $\psi$ and $\k$
belonging to different $\Mbar_{g,X}$'s as it is now traditional.

\begin{remark}
Wolpert has proven \cite{wolpert:homology} that, on $\Mbar_g$,
we have $\k_1=[\omega_{WP}]/\pi^2$ and that the amplitude of
$\k_1\in A^1(\Mbar_g)$ (and so the projectivity of $\Mbar_g$)
can be recovered from the fact that
$[\omega_{WP}/\pi^2]$ is an integral K\"ahler class
\cite{wolpert:positive}. He also showed that
the cohomological identity
$[\omega_{WP}/\pi^2]=\k_1=(\pi_q)_*\psi_q^2$
admits a beautiful
pointwise interpretation \cite{wolpert:chern}.
\end{remark}

%%%%%%%%%%%%%%%%%%%%%%%%%%%%%%%%%%%%%%%%%%%%%%%%%%%%%%%%%%%%%%%%%%%%%%%%%%
\subsubsection{Tautological rings.}
%%%%%%%%%%%%%%%%%%%%%%%%%%%%%%%%%%%%%%%%%%%%%%%%%%%%%%%%%%%%%%%%%%%%%%%%%%

Because of the natural definition of $\k$ and $\psi$ classes,
as explained before,
the subring $R^*(\M_{g,X})$ of $CH^*(\M_{g,X})_{\Q}$
they generate is called the {\it tautological ring}
of $\M_{g,X}$. Its image $RH^*(\M_{g,X})$ through the cycle class map
is called cohomology tautological ring.

From an axiomatic point of view,
the {\it system of tautological rings}
$(R^*(\Mbar_{g,X}))$ is the minimal
system of subrings of
$(CH^*(\Mbar_{g,X}))$ is
the minimal system of subrings such that
\begin{itemize}
\item
every $R^*(\Mbar_{g,X})$ contains the fundamental class
$[\Mbar_{g,X}]$
\item
the system is closed under push-forward maps $\pi_*$,
$(\vartheta_{irr})_*$ and $(\vartheta_{g',I})_*$.
\end{itemize}

$R^*(\M_{g,X})$ is defined to be the image of
the restriction map
$R^*(\Mbar_{g,X})\rar CH^*(\M_{g,X})$.
The definition for the rational cohomology
is analogous
(where the role of $[\Mbar_{g,X}]$ is here played
by its Poincar\'e dual $1\in H^0(\Mbar_{g,X};\Q)$).

It is a simple fact to remark that all tautological rings contain
$\psi$ and $\kappa$ classes and in fact that
$R^*(\M_{g,X})$ is generated by them.
Really, this was the original definition of
$R^*(\M_{g,X})$.

%%%%%%%%%%%%%%%%%%%%%%%%%%%%%%%%%%%%%%%%%%%%%%%%%%%%%%%%%%%%%%%
\subsubsection{Faber's formula.}
%%%%%%%%%%%%%%%%%%%%%%%%%%%%%%%%%%%%%%%%%%%%%%%%%%%%%%%%%%%%%%%

The $\psi$ classes interact reasonably well
with the forgetful maps. In fact
\begin{align*}
(\pi_q)_*(\psi_{x_1}^{r_1}\cdots\psi_{x_n}^{r_n})=
\sum_{\{i| r_i>0\}} \psi_{x_1}^{r_1}\cdots\psi_{x_i}^{r_i-1}
\cdots\psi_{x_n}^{r_n} \\
(\pi_q)_*(\psi_{x_1}^{r_1}\cdots\psi_{x_n}^{r_n}\psi_q^{b+1})=
\psi_{x_1}^{r_1}\cdots\psi_{x_n}^{r_n}\k_b
\end{align*}
where the first one is the so-called {\it string equation} and
the second one for $b=0$ is the {\it dilaton equation}
(see \cite{witten:intersection}).
They have been generalized by Faber
for maps that forget more than one point:
Faber's formula (which we are going to describe below)
can be proven using the second equation above and
the relation
$\pi_q^*(\k_j)=\k_j-\psi_q^j$
(proven in \cite{arbarello-cornalba:combinatorial}).

Let $Q:=\{q_1,\dots,q_m\}$
and let $\pi_Q:\Mbar_{g,X\cup Q}\rar\Mbar_{g,X}$ be the forgetful map.
Then
\[
(\pi_Q)_*(\psi_{x_1}^{r_1}\cdots\psi_{x_n}^{r_n}
\psi_{q_1}^{b_1+1}\cdots\psi_{q_m}^{b_m+1})=
\psi_{x_1}^{r_1}\cdots\psi_{x_n}^{r_n}
K_{b_1\cdots b_m}
\]
where $K_{b_1\cdots b_m}=\sum_{\s\in\mathfrak{S}_m} \k_{b(\s)}$ and
$\k_{b(\s)}$ is defined in the following way.
%
%\index[not]{$K_{b_1\cdots b_m}$, $k_{b(\s)}$}
%
If $\g=(c_1,\dots,c_l)$ is a cycle,
then set $b(\g):=\sum_{j=1}^l b_{c_j}$.
If $\s=\g_1\cdots\g_{\nu}$ is the decomposition in
disjoint cycles (including 1-cycles), then we let
$k_{b(\s)}:=\prod_{i=1}^{\nu} \k_{b(\g_i)}$.
We refer to \cite{kmz:volumes} for more details on Faber's
formula, to \cite{arbarello-cornalba:combinatorial} and
\cite{arbarello-cornalba:algebraic} for more properties
of tautological classes and to \cite{faber:conjectures} for
a conjectural description (which is
now partially proven) of the tautological rings.

%%%%%%%%%%%%%%%%%%%%%%%%%%%%%%%%%%%%%%%%%%%%%%%%%%%%%%%%%%%%%%%%%%%%%
\subsection{Kontsevich's compactification}\label{ss:konts}
%%%%%%%%%%%%%%%%%%%%%%%%%%%%%%%%%%%%%%%%%%%%%%%%%%%%%%%%%%%%%%%%%%%%%
\subsubsection{The line bundle $\mathbb{L}$.}
%%%%%%%%%%%%%%%%%%%%%%%%%%%%%%%%%%%%%%%%%%%%%%%%%%%%%%%%%%%%%%%%%%%%%

It has been observed by Witten \cite{witten:intersection} that
the intersection theory of $\kappa$ and $\psi$
classes can be reduced to that of
$\psi$ classes only by using the push-pull formula with respect
to the forgetful morphisms. Moreover recall that
\[
\psi_{x_i}=c_1(\omega_{\pi_{x_i}}(D_{x_i}))
\]
on $\Mbar_{g,X}$, where $D_{x_i}=\sum_{j\neq i} \d_{0,\{x_i,x_j\}}$
(as shown in \cite{witten:intersection}).
So, in order to find a ``minimal'' projective compactification
of $\M_{g,X}$ where to compute the intersection numbers of
the $\psi$ classes, it is natural to look at the maps induced
by the linear system $\mathbb{L}:=\sum_{x_i\in X}\omega_{\pi_{x_i}}(D_{x_i})$.
It is well-known that $\mathbb{L}$ is nef and big (Arakelov
\cite{arakelov:families} and Mumford \cite{mumford:towards}),
so that the problem is to decide whether $\mathbb{L}$ is semi-ample
and to determine its exceptional locus $Ex(\mathbb{L}^{\otimes d})$
for $d\gg 0$.

It is easy to see that $\mathbb{L}^{\otimes d}$ pulls back to
the trivial line bundle via the boundary map
\mbox{$\Mbar_{g',\{x'\}}\times\{C\}\lra\Mbar_{g,X}$},
where $C$ is a fixed curve of genus $g-g'$ with a $X\cup\{x''\}$-marking
and the map glues $x'$ with $x''$. Hence the map induced by the
linear system $\mathbb{L}^{\otimes d}$ (if base-point-free)
should restrict to the projection
\mbox{$\Mbar_{g,\{x'\}}\times\Mbar_{g-g',X\cup\{x''\}}
\lra \Mbar_{g-g',X\cup\{x''\}}$}
on these boundary components.

Whereas $\mathbb{L}$ is semi-ample in characteristic $p>0$, it is not so in
characteristic $0$ (Keel \cite{keel:basepoint}). However, one can still
topologically contract the exceptional (with respect to $\mathbb{L}$)
curves to obtain Kontsevich's map
\[
\xi':\Mbar_{g,X}\lra \Mbar^K_{g,X}
\]
which is a proper continuous surjection of orbispaces.
A consequence of Keel's result is that the coarse
$\ol{M}^K_{g,P}$ cannot be given
a scheme structure such that the contraction map is a morphism.
This is in some sense unexpected, because the morphism behaves as if
it were algebraic: in particular, the fiber product
$\ol{M}_{g,X}\times_{\ol{M}^K_{g,X}}\ol{M}_{g,X}$ is projective.

\begin{remark}
$\Mbar^K_{g,X}$ can be given the structure of a stratified
orbispace, where the stratification is again
by topological
type of the generic curve in the fiber of $\xi'$.
Also, the stabilizer of a point $s$ in $\Mbar^K_{g,X}$ will be
the same as the stabilizer of the generic point in $(\xi')^{-1}(s)$.
\end{remark}

%%%%%%%%%%%%%%%%%%%%%%%%%%%%%%%%%%%%%%%%%%%%%%%%%%%%%%%%%%%%%%%%%%%%%%%%
\subsubsection{Visibly equivalent curves.}
%%%%%%%%%%%%%%%%%%%%%%%%%%%%%%%%%%%%%%%%%%%%%%%%%%%%%%%%%%%%%%%%%%%%%%%%

So now we leave the realm of algebraic geometry and proceed
topologically to construct and describe this different compactification.
In fact we introduce a slight modification of Kontsevich's
construction (see \cite{kontsevich:intersection}).
We realize it as a quotient of $\Mbar_{g,X}\times \D_X$
by an equivalence relation, where
$\D_X$ is the standard simplex in $\R^X$

If $(\Si,\up)$ is an element of $\Mbar_{g,X}\times \D_X$, then
we say that an irreducible component of $\Si$
(and so the associated vertex of the dual graph $\zeta_{\Si}$) is
{\it visible} with respect to $\up$ if it contains a point
$x_i \in X$ such that $p_{i}>0$.

Next, we declare that $(\Si,\up)$ is equivalent to $(\Si',\up')$ if
$\up=\up'$ and there is a homeomorphism of pointed surfaces
$\Si \arr{\sim}{\lra} \Si'$, which is biholomorphic on the visible
components of $\Si$.
As this relation would not give back a Hausdorff space
we consider its closure, which we are now going to describe.

Consider the following two moves on the dual graph $\zeta_\Si$:
\begin{enumerate}
\item
if two non-positive vertices $w$ and $w'$ are joined by
an edge $e$, then we can build a new graph discarding $e$,
merging $w$ and $w'$ along $e$, thus obtaining
a new vertex $w''$, which we label
with $(g_{w''},X_{w''}):=(g_w+g_{w'},X_w\cup X_{w'})$
\item
if a non-positive vertex $w$ has a loop $e$, we can
make a new graph discarding $e$ and relabeling $w$
with $(g_w+1,X_w)$.
\end{enumerate}
Applying these moves to $\zeta_\Si$ iteratively until the process ends,
we end up with a
{\it reduced dual graph} $\zeta_{\Si,\up}^{red}$.
Call
$V_-(\Si,\up)$ the subset of invisible
vertices and
$V_+(\Si,\up)$ the subset of visible
vertices of $\zeta_{\Si,\up}^{red}$.

For every couple $(\Si,\up)$ denote by $\ol{\Si}$ the
quotient of $\Si$ obtained collapsing every non-positive
component to a point.

We say that $(\Si,\up)$ and $(\Si',\up')$ are {\it visibly equivalent}
if $\up=\up'$ and
there exist 
a homeomorphism
$\ol{\Si} \arr{\sim}{\lra} \ol{\Si}'$,
whose restriction to each component is analytic,
and a compatible isomorphism
$f^{red}: \zeta_{\Si,p}^{red} \arr{\sim}{\lra} \zeta_{\Si',p'}^{red}$
of reduced dual graphs.

\begin{remark}
In other words, $(\Si,\up),(\Si',\up')$
are visibly equivalent if and only if
$\up=\up'$
there exists a stable $\Si''$ and maps $h:\Si''\rar \Si$
and $h':\Si''\rar \Si'$ such that $h,h'$ are biholomorphic
on the visible components and are a stable
marking on the invisible components of $(\Si'',\up)$
(that is, they may shrink some disjoint simple closed curves
to nodes and are homeomorphisms elsewhere).
\end{remark}

Finally call
\[
\xi: \Mbar_{g,X}\times \D_X \lra
\Mbar^\Delta_{g,X}:=\Mbar_{g,X}\times\D_X/\!\sim
\]
the quotient map and remark that $\Mbar^\Delta_{g,X}$ is compact
and that $\xi$ commutes with the projection onto $\D_X$.

Similarly, one can say that two $(S,X)$-marked stable
surfaces $([f:S\rar \Si],\up)$ and $([f':S\rar \Si'],\up')$
are visibly equivalent if there exists a stable $(S,X)$
marked $[f'':S\rar \Si'']$ and maps
$h:\Si''\rar \Si$ and $h':\Si''\rar \Si'$
such that $h\circ f''\simeq f$, $h'\circ f''\simeq f'$
and $(\Si,\up),(\Si',\up')$ are visibly
equivalent through $h,h'$ (see the remark above).
Consequently, we can define $\ol{\Teich}^{\Delta}(S,X)$
as the quotient of $\ol{\Teich}(S,X)\times\Delta_X$
obtained by identifying visibly equivalent $(S,X)$-marked
surfaces.

For every $\up$ in $\D_X$, we will
denote by $\Mbar^\Delta_{g,X}(\up)$ the subset of points
of the type $[\Si,\up]$.
%and we will write $\Mbartri_{g,X}(L)$
%for $\cup_{l\in L}\Mbartri_{g,P}(p)$ where $L\subset \D_P$.
Then it is clear that $\Mbar^\Delta_{g,X}(\D^\circ_X)$
is in fact homeomorphic to a product $\Mbar^\Delta_{g,X}(\up)\times\D^\circ_X$
for any given $\up \in \D^\circ_X$.
Observe that $\Mbar^\Delta_{g,X}(\up)$ is isomorphic to $\Mbar^K_{g,X}$
for all \mbox{$\up \in \D^\circ_X$} in such a way that
\[
\xi_{\up}:\Mbar_{g,X}\cong\Mbar_{g,X}\times\{\up\}\lra\Mbar^\Delta_{g,X}(\up)
\]
identifies to $\xi'$.

Notice, by the way, that the fibers of $\xi$ are isomorphic to moduli spaces.
More precisely consider a point $[\Si,\up]$ of $\Mbar^\Delta_{g,X}$.
For every $w \in V_-(\Si,\up)$, call $Q_v$ the subset of edges
of $\zeta_{\Si,\up}^{red}$ outgoing from $w$.
Then we have the natural isomorphism
\[
\xi^{-1}([\Si,\up])\cong \prod_{w\in V_-(\Si,\up)} \Mbar_{g_w,X_w\cup Q_w}
\]
according to the fact that
\mbox{$\ol{M}_{g,X}\times_{\ol{M}^K_{g,X}}\ol{M}_{g,X}$}
is projective.
%
%%%%%%%%%%%%%%%%%%%%%%%%%%%%%%%%%%%%%%%%%%%%%%%%%%%%%%%%%%%%%%%%%%%%%%%%%
%%%%%%%%%%%%%%%%%%%%%%%%%%%%%%%%%%%%%%%%%%%%%%%%%%%%%%%%%%%%%%%%%%%%%%%%%
\section{Cell decompositions of the
moduli space of curves}\label{sec:triangulation}
%%%%%%%%%%%%%%%%%%%%%%%%%%%%%%%%%%%%%%%%%%%%%%%%%%%%%%%%%%%%%%%%%%%%%%%%%
%
%%%%%%%%%%%%%%%%%%%%%%%%%%%%%%%%%%%%%%%%%%%%%%%%%%%%%%%%%%%%%%%%%%%%%%%%%
\subsection{Harer-Mumford-Thurston construction}\label{ss:hmt}
%%%%%%%%%%%%%%%%%%%%%%%%%%%%%%%%%%%%%%%%%%%%%%%%%%%%%%%%%%%%%%%%%%%%%%%%%

One traditional way to associate a weighted arc system to a Riemann
surface endowed with weights at its marked points is to look at
critical trajectories of Jenkins-Strebel quadratic differentials.
Equivalently, to decompose the punctured
surface into a union of semi-infinite
flat cylinders with assigned lengths of their circumference.

%%%%%%%%%%%%%%%%%%%%%%%%%%%%%%%%%%%%%%%%%%%%%%%%%%%%%%%%%%%%%%%%%%%%%%%%%
\subsubsection{Quadratic differentials.}
%%%%%%%%%%%%%%%%%%%%%%%%%%%%%%%%%%%%%%%%%%%%%%%%%%%%%%%%%%%%%%%%%%%%%%%%%

Let $\Si$ be a compact Riemann surface and let $\varphi$ be a
{\it meromorphic quadratic differential}, that is $\varphi=\varphi(z)dz^2$
where $z$ is a local holomorphic coordinate and $\varphi(z)$ is a
meromorphic function. Being a quadratic differential means
that, if $w=w(z)$ is another local coordinate, then
$\dis \varphi=\varphi(w)\left(\frac{dz}{dw}\right)^2 dw^2$.

{\it Regular points} of $\Si$ for $\varphi$ are points where $\varphi$
has neither a zero nor a pole; {\it critical points} are zeroes
or poles of $\varphi$.

We can attach a metric to $\varphi$, by simply setting
$\dis |\varphi|:=\sqrt{\varphi\ol{\varphi}}$. In coordinates,
$|\varphi|=|\varphi(z)|dz\,d\ol{z}$.
The metric is well-defined and flat at the regular points
and it has conical singularities (with angle $\alpha=(k+2)\pi$)
at simple poles ($k=-1$) and at zeroes of order $k$.
Poles of order $2$ or higher are at infinite distance.

If $P$ is a regular point,
we can pick a local holomorphic coordinate $z$ at
$P\in U\subset\Si$ such that
$z(P)=0$ and
$\varphi=dz^2$ on $U$. The choice of $z$ is unique up sign.
Thus, $\{Q\in U\,|\,z(Q)\in \R\}$ defines a real-analytic curve through $P$
on $\Si$, which is called {\it horizontal trajectory} of $\varphi$.
Similarly, $\{Q\in U\,|\,z(Q)\in i\R\}$ defines the {\it vertical
trajectory} of $\varphi$ through $P$.

Horizontal (resp. vertical) trajectories $\tau$ are intrinsically
defined by asking that the restriction of $\varphi$ to $\tau$
is a positive-definite (resp. negative-definite) symmetric bilinear
form on the tangent bundle of $\tau$.

If $\varphi$ has at worst double poles,
then the local aspect of horizontal trajectories is as in
Figure~\ref{fig:trajectories}
(horizontal trajectories through $q$ are drawn thicker).

\begin{center}
\begin{figurehere}
\psfrag{q}{$q$}
\psfrag{0}{$f(z)=dz^2$}
\psfrag{1}{$f(z)=z\,dz^2$}
\psfrag{2}{$f(z)=z^2\,dz^2$}
\psfrag{-1}{$\dis f(z)=\frac{dz^2}{z}$}
\psfrag{-2-}{$\dis f(z)=-a\frac{dz^2}{z^2}$}
\psfrag{a}{$a>0$}
\psfrag{-2+}{$\dis f(z)=a\frac{dz^2}{z^2}$}
\includegraphics[width=0.8\textwidth]{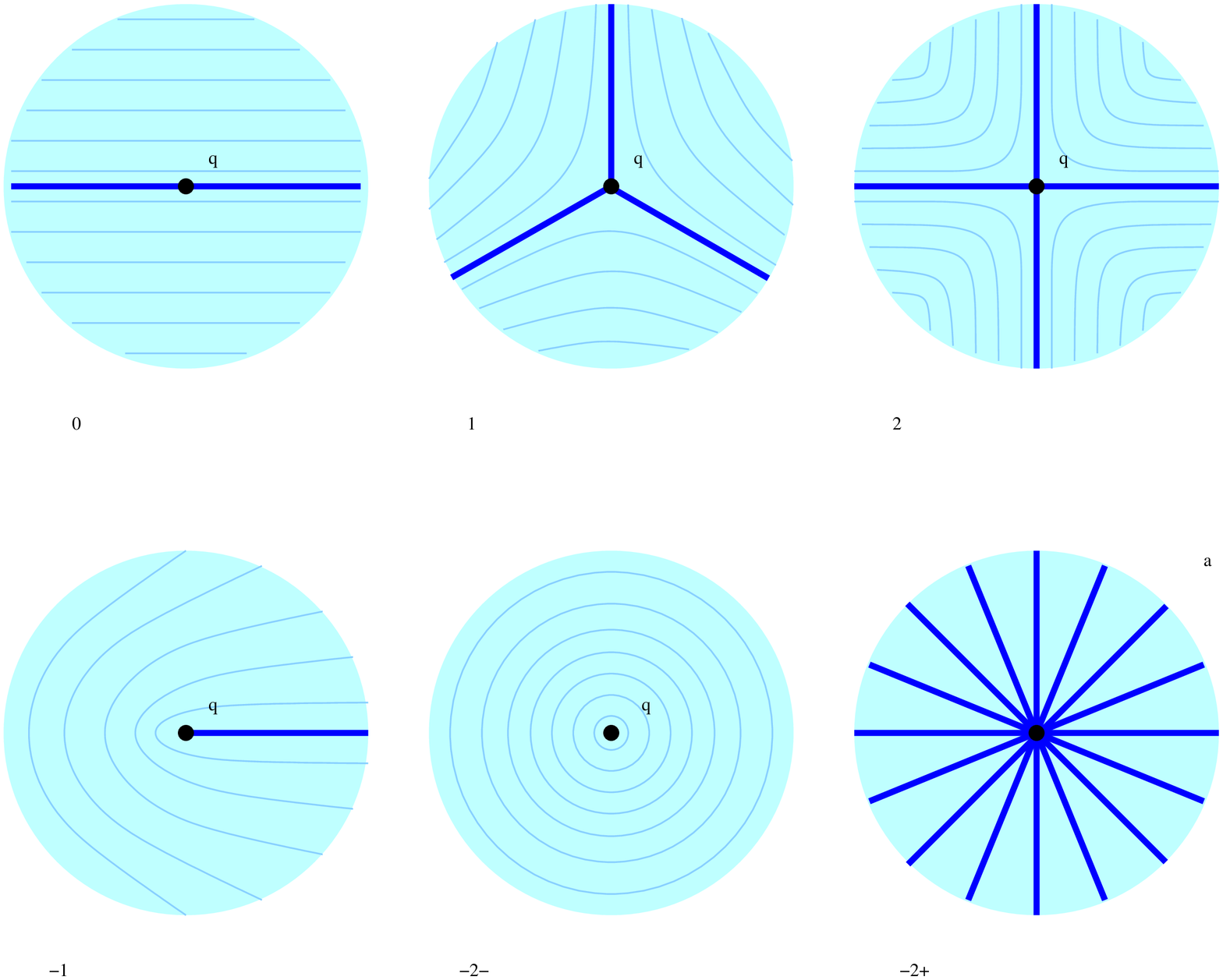}
\caption{Local structure of horizontal trajectories.}
\label{fig:trajectories}
\end{figurehere}
\end{center}

Trajectories are called {\it critical} if they meet a critical point.
It follows from the general classification (see \cite{strebel:book})
that
\begin{itemize}
\item
a trajectory
is {\it closed} if and only if it is either periodic or it
starts and ends at a critical point;
\item
if a horizontal trajectory $\tau$ 
is periodic, then there exists a maximal open annular
domain $A\subset \Si$ and a number $c>0$ such that
\[
\left(A,\varphi\Big|_A\right)
\arr{\sim}{\lra}
\left(\{z\in\C\,|\,r<|z|<R\},-c\frac{dz^2}{z^2}\right)
\]
and, under this identification,
$\tau=\{z\in\C\,|\, h=|z|\}$ for some $h\in(r,R)$;
\item
if all horizontal trajectories are closed of finite length,
then $\varphi$ has at worst double poles, where
it has negative quadratic residue (i.e. at a double pole,
it looks like $\dis -a \frac{dz^2}{z^2}$, with $a>0$).
\end{itemize}

%%%%%%%%%%%%%%%%%%%%%%%%%%%%%%%%%%%%%%%%%%%%%%%%%%%%%%%%%%%%%%%%%%%%%%%%%
\subsubsection{Jenkins-Strebel differentials.}
%%%%%%%%%%%%%%%%%%%%%%%%%%%%%%%%%%%%%%%%%%%%%%%%%%%%%%%%%%%%%%%%%%%%%%%%%

There are many theorems about existence
and uniqueness of quadratic differentials $\varphi$
with specific
behaviors of their trajectories and about their
characterization using extremal properties of
the associated metric $|\varphi|$ (see Jenkins
\cite{jenkins:existence}).
The following result is the one we are interested in.

\begin{theorem}[Strebel \cite{strebel:quadratic}]\label{thm:strebel}
Let $\Si$ be a compact Riemann surface of genus $g$
and $X=\{x_1,\dots,x_n\}
\subset\Si$ such that $2g-2+n>0$.
For every $(p_1,\dots,p_n)\in\R_+^X$
there exists a unique quadratic differential $\varphi$
such that
\begin{itemize}
\item[(a)]
$\varphi$ is holomorphic on $\Si\setminus X$
\item[(b)]
all horizontal trajectories of $\varphi$ are closed
\item[(c)]
it has a double pole at $x_i$ with quadratic
residue $\dis -\left(\frac{p_i}{2\pi} \right)^2$
\item[(d)]
the only annular domains of $\varphi$ are pointed
discs at the $x_i$'s.
\end{itemize}
Moreover, $\varphi$ depends continuously on $\Si$
and on $\up=(p_1,\dots,p_n)$.
\end{theorem}

\begin{center}
{\large
\begin{figurehere}
\psfrag{xi}{$x_i$}
\includegraphics[width=0.8\textwidth]{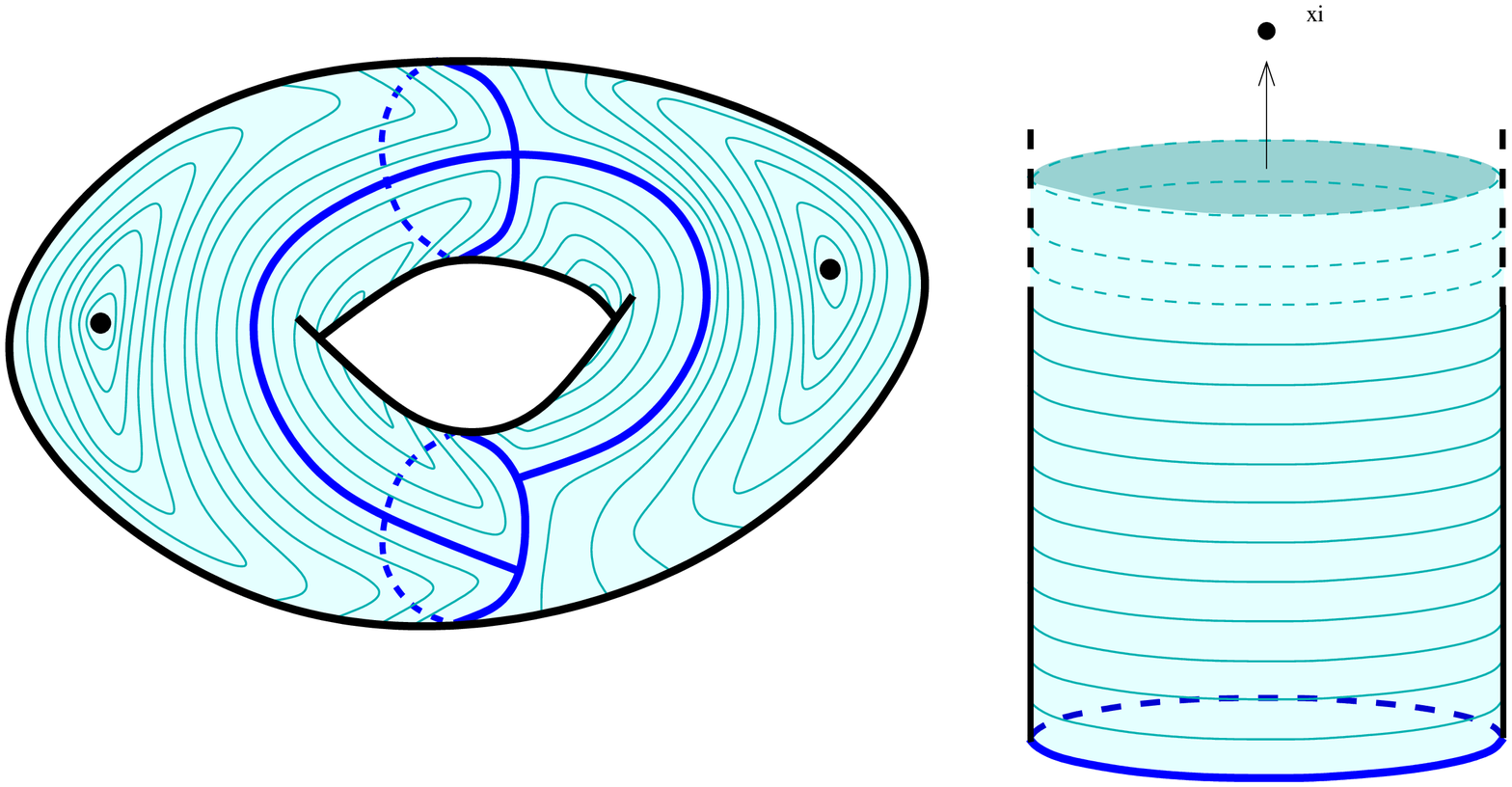}
\caption{Example of horizontal foliation of a Jenkins-Strebel differential.}
\label{fig:JS}
\end{figurehere}
}
\end{center}

\begin{remark}
Notice that the previous result establishes the
existence of a continuous map
\[
\R_+^X \lra \{\text{continuous sections of
$\mathcal{Q}(S,2X)\rar\Teich(S,X)$} \}
\]
where $\mathcal{Q}(S,2X)$ is the vector
bundle, whose fiber over $[f:S\rar\Si]$ is
the space of quadratic differentials on $\Si$,
which can have double poles at $X$ and
are holomorphic elsewhere.
Hubbard and Masur \cite{hubbard-masur:foliations} proved
(in a slightly different case, though) that the
sections of $\mathcal{Q}(S,2X)$ are piecewise
real-analytic and gave precise equations for
their image.
\end{remark}

Quadratic differentials that satisfy (a) and (b)
are called {\it Jenkins-Strebel differentials}.
They are particularly easy to understand because
their critical trajectories form a graph $G=G_{\Si,\up}$
embedded inside the surface $\Si$ and $G$
decomposes $\Si$ into a union of cylinders
(with respect to the flat metric $|\varphi|$),
of which horizontal trajectories are the circumferences.

Property (d) is telling us that $\Si\setminus X$
retracts by deformation onto $G$, flowing along
the vertical trajectories out of $X$.

\begin{remark}
It can be easily seen that Theorem~\ref{thm:strebel}
still holds for $p_1,\dots,p_n \geq 0$ but
$\up\neq 0$. Condition (d) can be rephrased
by saying that every annular domain corresponds
to some $x_i$ for which $p_i>0$, and that $x_j\in G$
if $p_j=0$. It is still true that $\Si\setminus X$
retracts by deformation onto $G$.
\end{remark}

We sketch the traditional existence proof of
Theorem~\ref{thm:strebel}.

\begin{definition}
The {\it modulus} of a standard annulus $A(r,R)=\{z\in\C\,
|\, r<|z|<R\}$ is $\dis m(A(r,R))=\frac{1}{2\pi}\log(R/r)$
and the modulus of an annulus $A$ is defined to be that
of a standard annulus biholomorphic to $A$.
Given a simply connected domain $0\in U\subset \C$
and let $z$ be a holomorphic coordinate at $0$.
The {\it reduced modulus} of the annulus
$U^*=U\setminus\{0\}$ is $m(U^*,w)=m(U^*\cap\{|z|>\e\})+
\frac{1}{2\pi}\log(\e)$, which is independent
of the choice of a sufficiently small $\e>0$.
\end{definition}

Notice that the {\it extremal length} $\mathrm{E}_\g$
of a circumference $\g$ inside $A(r,R)$ is exactly
$1/m(A(r,R))$.

\begin{proof}[Existence of Jenkins-Strebel differential]
Fix holomorphic coordinates $z_1,\dots,z_n$ at $x_1,\dots,x_n$.
A {\it system of annuli} is a holomorphic injection
$s:\Delta\times X\hra \Si$ such that $s(0,x_i)=x_i$,
where $\Delta$ is the unit disc in $\C$.
Call $m_i(s)$ the reduced modulus $m(s(\Delta\times\{x_i\}),z_i)$
and
define the functional
\[
F:
\xymatrix@R=0in{
\{\text{systems of annuli}\} \ar[rr] && \R \\
s \ar@{|->}[rr] && \dis \sum_{i=1}^n p_i^2 m_i(s)
}
\]
which is bounded above, because $\Si\setminus X$ is hyperbolic.
A maximizing sequence $s_n$ converges (up to extracting a subsequence)
to a system of annuli $s_\infty$ and let $D_i=s_\infty(\Delta\times\{x_i\})$.
Notice that the restriction of $s_{\infty}$ to $\Delta\times\{x_i\}$
is injective if $p_i>0$ and is constantly $x_i$ if $p_i=0$.

Clearly, $s_\infty$ is maximizing for every choice of $z_1,\dots,z_n$
and so we can assume that, whenever $p_i>0$,
$z_i$ is the coordinate induced by $s_{\infty}$.

Define the $L^1_{loc}$-quadratic
differential $\varphi$ on $\Si$ as
$\dis \varphi:=\left(-\frac{p_i^2}{4\pi^2}
\frac{dz_i^2}{z_i^2}\right)$ on $D_i$
(if $p_i>0$) and $\varphi=0$ elsewhere.
Notice that $F(s_\infty)=\|\varphi\|_{red}$,
where the {\it reduced norm} is given by
\[
\|\varphi\|_{red}:=
\int_{\Si}\left[ |\varphi|^2-\sum_{i:p_i>0}\frac{p_i^2}{4\pi^2}
\frac{dz_i\,d\ol{z_i}}{|z_i|^2}\chi(|z_i|<\e_i)\right]+
\sum_{i=1}^n p_i\log(\e_i)
\]
which is independent of the choice of sufficiently small
$\e_1,\dots,\e_n>0$.

As $s_\infty$ is a stationary point for $F$,
so is for $\|\cdot\|_{red}$. Thus, for every smooth vector
field $V=V(z){\pa}/{\pa z}$ on $\Si$,
compactly supported on $\Si\setminus X$,
the first order variation of
\[
\|f_t^*(\varphi)\|_{red}=\|\varphi\|_{red}+2t
\int_S \mathrm{Re}(\varphi\ol{\pa}V)
+o(t)
\]
must vanish, where $f_t=\exp(tV)$.
Thus, $\varphi$ is holomorphic on $\Si\setminus X$
by Weyl's lemma and it satisfies all the requirements.
\end{proof}

%%%%%%%%%%%%%%%%%%%%%%%%%%%%%%%%%%%%%%%%%%%%%%%%%%%%%%%%%%%%%%%%%%%%%%%%%
\subsubsection{The nonsingular case.}
%%%%%%%%%%%%%%%%%%%%%%%%%%%%%%%%%%%%%%%%%%%%%%%%%%%%%%%%%%%%%%%%%%%%%%%%%

Using the construction described above,
we can attach to every $(\Si,X,\up)$
a graph $G_{\Si,\up}\subset\Si$ (and thus
an $(S,X)$-marked ribbon graph $\GG_{\Si,\up}$)
which is naturally metrized by $|\varphi|$.
By arc-graph duality (in the nonsingular case, see
\ref{sss:ns-duality}), we also
have a weighted proper system of arcs in $\Si$.
Notice that, because of (c), the boundary weights
are exactly $p_1,\dots,p_n$.

If $[f:S\rar \Si]$ is a point in $\Teich(S,X)$
and $\up\in(\R_{\geq 0}^X)\setminus\{0\}$,
then the previous construction (which is first
explicitly mentioned by Harer in \cite{harer:virtual},
where he attributes it to Mumford and Thurston)
provides a point in $|\Ao(S,X)|\times\R_+$.
It is however clear that, if $a>0$,
then the Strebel differential associated to
$(\Si,a\up)$ is $a\varphi$. Thus, we can just
consider $\up\in \mathbb{P}(\R_{\geq 0}^X)\cong \Delta_X$,
so that the corresponding weighted arc system
belongs to $|\Ao(S,X)|$ (after multiplying by a factor $2$).

Because of the continuous dependence of $\varphi$
on $\Si$ and $\up$, the map
\[
\Psi_{JS}: \Teich(S,X)\times\Delta_X \lra |\Ao(S,X)|
\]
is {\it continuous}.

We now show that a point $\ol{w}\in|\Ao(S,X)|$
determines exactly one $(S,X)$-marked surface,
which proves that $\Psi_{JS}$ is {\it bijective}.

By \ref{sss:metrized},
we can associate a metrized $(S,X)$-marked 
nonsingular ribbon graph $\GG_{\ua}$ to
each $w\in |\Ao(S,X)|_{\R}$ supported on $\ua$.
However, if we realize $|\GG_{\ua}|$ by gluing
semi-infinite tiles $T_{\ora{\a_i}}$ of the type
$[0,w(\a_i)]_x\times[0,\infty)_y\subset \hat{\C}_z$,
which naturally come together with a complex structure and
a quadratic differential $dz^2$, then $|\GG_{\ua}|$ 
becomes a Riemann surface endowed with the (unique)
Jenkins-Strebel quadratic differential $\varphi$ determined
by Theorem~\ref{thm:strebel}.
Thus, $\Psi_{JS}^{-1}(w)=([f:S\rar |\GG_{\ua}|],\up)$,
where $p_i$ is obtained from the
quadratic residue of $\varphi$ at $x_i$.
Moreover, the length function defined
on $|\Ao(S,X)|_{\R}$ exactly corresponds to the
$|\varphi|$-length function on $\M_{g,X}$.

Notice that $\Psi_{JS}$ is $\G(S,X)$-equivariant by
construction and so induces a continuous
bijection
$\ol{\Psi}_{JS}:\M_{g,X}\times\Delta_X\rar |\Ao(S,X)|/\G(S,X)$
on the quotient. If we prove that $\ol{\Psi}_{JS}$
is {\it proper}, then $\ol{\Psi}_{JS}$ is a homeomorphism.
To conclude that $\Psi_{JS}$ is a homeomorphism too,
we will use the following.

\begin{lemma}\label{lemma:isom}
Let $Y$ and $Z$ be metric spaces
acted on discontinuously
by a discrete group of isometries $G$
and let $h:Y\rar Z$ be a $G$-equivariant
continuous injection such that the induced map
$\ol{h}:Y/G\rar Z/G$ is a homeomorphism.
Then $h$ is a homeomorphism.
\end{lemma}
\begin{proof}
To show that $h$ is surjective, let $z\in Z$.
Because $\ol{h}$ is bijective, $\exists ! [y]\in Y/G$
such that $\ol{h}([y])=[z]$. Hence, $h(y)=z\cdot g$ for
some $g\in G$ and so $h(y\cdot g^{-1})=z$.

To prove that $h^{-1}$ is continuous,
let $(y_m)\subset Y$ be a sequence such that $h(y_m)\rar h(y)$
as $m\rar \infty$ for some $y\in Y$.
Clearly, $[h(y_m)]\rar[h(y)]$ in $Z/G$ and so $[y_m]\rar [y]$
in $Y/G$, because $\ol{h}$ is a homeomorphism.
Let $(v_m)\subset Y$ be a sequence such that $[v_m]=[y_m]$
and $v_m\rar y$ and call $g_m\in G$ the element such that
$y_m=v_m\cdot g_m$. By continuity of $h$, we have $d_Z(h(v_m), h(y))\rar 0$
and by hypothesis $d_Z(h(v_m)\cdot g_m,h(y))\rar 0$.
Hence, $d_Z(h(y),h(y)\cdot g_m)\rar 0$ and so
$g_m\in\mathrm{stab}(h(y))=\mathrm{stab}(h)$ for large $m$,
because $G$ acts discontinuously on $Z$.
As a consequence, $y_m\rar y$ and so $h^{-1}$ is continuous.
\end{proof}

The final step is the following.

\begin{lemma}
$\ol{\Psi}_{JS}:\M_{g,X}\times\Delta_X\rar |\Ao(S,X)|/\G(S,X)$ is proper.
\end{lemma}

\begin{proof}
Let $([\Si_m],\up_m)$ be a diverging sequence in $\M_{g,X}\times\Delta_X$
and call $\l_m$ the hyperbolic metric on $\Si\setminus X$.
By Mumford-Mahler criterion, there exist simple
closed hyperbolic geodesics $\g_m\subset\Si_m$ such that
$\ell_{\l_m}(\g_m)\rar 0$.
Because the hyperbolic length and the extremal length
are approximately proportional for short curves,
we conclude that extremal length $E(\g_m)\rar 0$.

Consider now the metric $|\varphi_m|$ induced by
the Jenkins-Strebel differential $\varphi_m$
uniquely determined by $(\Si_m,\up_m)$.
Call $\ell_{\varphi}(\g_m)$ the length of the unique
geodesic $\tilde{\g}_m$
with respect to the metric $|\varphi_m|$,
freely homotopic to $\g_m\subset\Si_m$.
Notice that $\tilde{\g}_m$ is a union of critical
horizontal trajectories.

Because $|\varphi_m|$ has infinite area, define
a modified metric $g_m$ on $\Si_m$ in the same conformal
class as $|\varphi_m|$ as follows.
\begin{itemize}
\item
$g_m$ agrees with $|\varphi_m|$
on the critical horizontal trajectories of $\varphi_m$
\item
Whenever $p_{i,m}>0$, consider a coordinate $z$ at $x_i$ such that
the annular domain of $\varphi_m$ at $x_i$ is
exactly $\Delta^*=\{z\in \C\,|\, 0<|z|<1\}$
and $\dis \varphi_m=-\frac{p_{i,m}^2dz^2}{4\pi^2 z^2}$.
Then define $g_m$ to agree with $|\varphi_m|$
on $\exp(-2\pi/p_{i,m})\leq |z|<1$ (which becomes
isometric to a cylinder of circumference $p_{i,m}$
and height $1$, so with area $p_{i,m}$) and to be the metric
of a flat Euclidean disc of circumference $p_{i,m}$
centered at $z=0$ (so with area $\pi p_{i,m}^2$)
on $|z|<\exp(-2\pi/p_{i,m})$.
\end{itemize}
Notice that the total area $A(g_m)$ is
$\pi(p_{1,m}^2+\dots+p_{n,m}^2)+(p_{1,m}+\dots+p_{n,m})\leq \pi+1$.

Call $\ell_g(\g_m)$ the length of the shortest $g_m$-geodesic
$\hat{\g}_m$ in the class of $\g_m$.
By definition, $\ell_g(\g_m)^2/A(g_m)\leq E(\g_m)\rar 0$
and so $\ell_g(\g_m)\rar 0$.
As a $g_m$-geodesic is either longer than $1$ or contained
in the critical graph of $\varphi$, then $\hat{\g}_m$
coincides with $\tilde{\g}_m$ for $m\gg 0$.

Hence, $\ell_\varphi(\g_m)\rar 0$ and so
$\mathrm{sys}(\ol{w}_m)\rar 0$.
By Lemma~\ref{lemma:compact}, we conclude that
$\ol{\Psi}_{JS}(\Si_m,\up_m)$ is diverging in $|\Ao(S,X)|/\G(S,X)$.
\end{proof}

\begin{remark}\label{rmk:zero}
Suppose that $([f_m:S\rar\Si_m],\up_m)$ is converging to
$([f:S\rar \Si],\up)
\in \ol{\Teich}_{g,X}\times\Delta_X$ and let $\Si'\subset\Si$
be an invisible component.  Then $S'=f^{-1}(\Si')$ is bounded
by simple closed curves $\g_1,\dots,\g_k\subset S$
and $\ell_{\varphi_m}(\g_i)\rar 0$ for $i=1,\dots,k$.
Just analyzing the shape of the critical graph of $\varphi_m$,
one can check that $\ell_{\varphi_m}(\g)\leq \sum_{i=1}^k
\ell_{\varphi_m}(\g_i)$ for all $\g\subset S'$.
Hence, $\ell_{\varphi_m}(\g)\rar 0$ and so $f_m^*\varphi_m$
tends to zero uniformly on the compact subsets of $(S')^\circ$.
\end{remark}

%%%%%%%%%%%%%%%%%%%%%%%%%%%%%%%%%%%%%%%%%%%%%%%%%%%%%%%%%%%%%%%%%%%%%%%%%
\subsubsection{The case of stable curves.}
%%%%%%%%%%%%%%%%%%%%%%%%%%%%%%%%%%%%%%%%%%%%%%%%%%%%%%%%%%%%%%%%%%%%%%%%%

We want to extend the map $\Psi_{JS}$ to Deligne-Mumford's augmentation:
will call still $\Psi_{JS}:\ol{\Teich}(S,X)\times\Delta_X\rar |\Af(S,X)|$
this extension.

Given $([f:S\rar \Si],\up)$, we can construct a Jenkins-Strebel
differential $\varphi$
on each visible component of $\Si$, by considering
nodes as marked points with zero weight. Extend $\varphi$
to zero over the invisible components.
Clearly, $\varphi$ is a holomorphic
section of $\omega_\Si^{\otimes 2}(2X)$ (the square of the
logarithmic dualizing sheaf on $\Si$):
call it the Jenkins-Strebel differential
associated to $(\Si,\up)$. Notice that it
clearly maximizes the functional $F$, used
in the proof of Theorem~\ref{thm:strebel}.

As $\varphi$ defines a metrized
ribbon graph for each visible component of $\Si$,
one can easily see that thus we have
an $(S,X)$-marked enriched ribbon graph $\GG^{en}$
(see \ref{sss:enriched}), where $\zeta$
is the dual graph of $\Si$ and $V_+$ is the set of
visible components of $(\Si,\up)$, $m$ is determined
by the $X$-marking and $s$ by the position of the nodes.

By arc-graph duality (see \ref{sss:arc-graph}),
we obtain a system of arcs
$\ua$ in $(S,X)$ and the metrics provide a system
of weights $\ol{w}$ with support on $\ua$.
This defines the set-theoretic extension of $\Psi_{JS}$.
Clearly, it is still $\G(S,X)$-equivariant and
it identifies visibly equivalent $(S,X)$-marked
surfaces. Thus, it descends to
a bijection
$\Psi_{JS}:\ol{\Teich}^\Delta(S,X)\rar |\Af(S,X)|$
and we also have
\[
\ol{\Psi}_{JS}:\Mbar^\Delta_{g,X} \lra |\Af(S,X)|/\G(S,X)
\]
where $|\Af(S,X)|/\G(S,X)$ can be naturally given
the structure of an orbispace (essentially, forgetting
the Dehn twists along curves of $S$ that are shrunk to points,
so that the stabilizer of an arc system just becomes
the automorphism group of the corresponding enriched
$X$-marked ribbon graph).

The only thing left to prove is that $\Psi_{JS}$ is continuous.
In fact, $\Mbar^\Delta_{g,X}$ is compact and $|\Af(S,X)|/\G(S,X)$
is Hausdorff: hence, $\ol{\Psi}_{JS}$
would be (continuous and)
automatically proper, and so a homeomorphism.
Using Lemma~\ref{lemma:isom} again (using a metric pulled
back from $\Mbar^\Delta_{g,X}$), we could
conclude that $\Psi_{JS}$ is a
homeomorphism too.

\begin{proof}[Continuity of $\Psi_{JS}$]
Consider a differentiable stable family
\[
\xymatrix@R=0.3in{
S \times[0,\e] \ar[r]^{\qquad f} \ar[rd] & \mathcal{C} \ar[d]\\
& [0,\e]
}
\]
of $(S,X)$-marked curves (that is, obtained restricting
to $[0,\e]$ a holomorphic family over the unit disc $\Delta$),
such that $g$ is topologically
trivial over $(0,\e]$ with fiber a curve with $k$ nodes.
Let also $\up:[0,\e]\rar \Delta_X$ be a differentiable family
of weights.

We can assume that there are disjoint
simple closed curves $\g_1,\dots,\g_k,\eta_1,\dots,\eta_h
\subset S$ such that $f(\g_i\times\{t\})$ is a node
for all $t$, that $f(\eta_j\times\{t\})$ is a node for $t=0$
and that $\mathcal{C}_t$ is smooth away from these nodes.

Fix $K$ a nonempty open relatively compact subset 
of $S\setminus(\g_1\cup \dots\cup \g_k\cup\eta_1\cup\dots
\cup \eta_h)$ that intersects every connected component.
Define a reduced $L^1$ norm of a section $\psi_t$ of
$\omega_{\mathcal{C}_t}^{\otimes 2}(2X)$ to be
$\|\psi\|_{red}=\int_{f_t(K)} |\psi|$.
Notice that $L^1$ convergence of holomorphic sections
$\psi_t$ as $t\rar 0$
implies uniform convergence of $f_t^*\psi_t$
on the compact subsets
of $S\setminus(\g_1\cup \dots\cup \g_k\cup\eta_1\cup\dots
\cup \eta_h)$.

Call $\varphi_t$ the Jenkins-Strebel differential associated
to $(\mathcal{C}_t,\up_t)$ with annular domains
$D_{1,t},\dots,D_{n,t}$.

As all the components of $\mathcal{C}_t$ are hyperbolic,
$\|\varphi_t\|_{red}$ is uniformly bounded and we
can assume (up to extracting a subsequence) that
$\varphi_t$ converges to a holomorphic section $\varphi'_0$
of $\omega_{\mathcal{C}_0}^{\otimes 2}(2X)$ in the reduced norm.
Clearly, $\varphi'_0$ will have double poles at $x_i$
with the prescribed residue.

Remark~\ref{rmk:zero} implies that $\varphi'_0$
vanishes on the invisible components of $\mathcal{C}_0$,
whereas it certainly does not on the visible ones.

For all those $(i,t)\in \{1,\dots,n\}\times[0,\e]$
such that $p_{i,t}>0$,
let $z_{i,t}$ be the coordinate at $x_i$ (uniquely
defined up to phase) given by $\dis z_{i,t}=u_{i,t}^{-1}\Big|_{D_{i,t}}$
and
\[
u_{i,t}:\ol{\Delta}\lra \ol{D}_{i,t}\subset \mathcal{C}_t
\]
is continuous on $\ol{\Delta}$ and biholomorphic in the interior
for all $t>0$
and $\dis \varphi_t\Big|_{D_{i,t}}=
-\frac{p_{i,t}^2 dz_{i,t}^2}{4\pi^2 z_{i,t}^2}$ for $t\geq 0$.
Whenever $p_{i,t}=0$, choose $z_{i,t}$ such that
$\varphi_t\Big|_{D_{i,t}}=z^k\,dz^2$, with $k=\mathrm{ord}_{x_i}\varphi_t$.
When $p_{i,t}>0$,
we can choose the phases of $u_{i,t}$ in such a way that
$u_{i,t}$ vary continuously with $t\geq 0$.

If $p_{i,0}=0$, then set $D_{i,0}=\emptyset$.
Otherwise, $p_{i,0}>0$ and so $D_{i,0}$ cannot shrink
to $\{x_i\}$ (because $F_t$ would go to $-\infty$
as $t\rar 0$).
In this case, call $D_{i,0}$
the region $\{|z_{i,0}|<1\}\subset \mathcal{C}_0$.
Notice that $\varphi'_0$ has a double pole at $x_i$ with
residue $p_{i,0}>0$ and
clearly $\dis \varphi'_0\Big|_{D_{i,0}}=
-\frac{p_{i,0}^2 dz_{i,0}^2}{4\pi^2 z_{i,0}^2}$.

We want to prove that the visible subsurface of $\mathcal{C}_0$
is covered by $\bigcup_i \ol{D}_{i,0}$ and so $\varphi'_0$
is a Jenkins-Strebel differential on each visible component
of $\mathcal{C}_0$. By uniqueness, it must coincide with $\varphi_0$.

Consider a point $y$ in
the interior of $f_0^{-1}(\mathcal{C}_{0,+})\setminus X$.
For every $t>0$ there exists
an $y_t \in S$ such that $f_t(y_t)$ does not belong to
the critical graph of $\varphi_t$ and the $f_t^*|\varphi_t|$-distance
$d_t(y,y_t)<t$. As $\varphi_t\rar\varphi_0$ in reduced norm
and $y,y_t \notin X$, then $d_0(y,y_t)\rar 0$ as $t\rar 0$.

We can assume (up to discarding some $t$'s) that
$f_{t}(y_{t})$ belongs to $D_{i,t}$ for a fixed $i$
and in particular that $f_t(y_t)=u_{i,t}(c_t)$
for some $c_t\in \Delta$.
Up to discarding some $t$'s, we can also assume that
$c_t\rar c_0\in \ol{\Delta}$. Call $y'_t$ the point
given by $f_0(y'_t)=u_{i,0}(c_t)$.
\begin{align*}
d_0(y'_t,y) & \leq d_0(y_t,y) + d_0(y'_t,y_t) 
\leq  d_0(y_t,y)  +  d_0(f_0^{-1}u_{i,0}(c_t),f_t^{-1}u_{i,t}(c_t))  \leq \\
 & \leq  d_0(y_t,y)  +  d_0(f_0^{-1}u_{i,0}(c_t),f_0^{-1}u_{i,0}(c_0)) +\\
& \qquad \quad +  d_0(f_0^{-1}u_{i,0}(c_0),f_t^{-1}u_{i,t}(c_0)) 
+ d_0(f_t^{-1}u_{i,t}(c_0),f_t^{-1}u_{i,t}(c_t))
\end{align*}
and all terms go to zero as $t\rar 0$.
Thus, every point in the smooth
locus $\mathcal{C}_{0,+}\setminus X$ is at $|\varphi_0|$-distance
zero from some $D_{i,0}$. Hence, $\varphi_0$ is a Jenkins-Strebel
differential on the visible components.

With a few simple considerations, one can easily conclude that
\begin{itemize}
\item
the zeroes of $\varphi_t$ move with continuity for $t\in[0,\e]$
\item
if $e_t$ is an edge of the critical graph of $\varphi_t$
which starts at $y_{1,t}$ and ends at $y_{2,t}$, and if
$y_{i,t}\rar y_{i,0}$ for $i=1,2$, then
$e_t\rar e_0$ the corresponding edge of the critical graph of $\varphi_0$
starting at $y_{1,0}$ and ending at $y_{2,0}$; moreover,
$\ell_{|\varphi_t|}(e_t)\rar \ell_{|\varphi_0|}(e_0)$
\item
the critical graph of $\varphi_t$ converges to that of $\varphi_0$
for the Gromov-Hausdorff distance.
\end{itemize}
Thus, the associated weighted arc systems $\ol{w}_t\in |\Af(S,X)|$
converge to $\ol{w}_0$ for $t\rar 0$. 
\end{proof}

Thus, we have proved the following result, claimed first
by Kontsevich in \cite{kontsevich:intersection}
(see Looijenga's \cite{looijenga:cellular}
and Zvonkine's \cite{zvonkine:strebel}).

\begin{proposition}
The map defined above
\[
\Psi_{JS}: \ol{\Teich}^\Delta(S,X) \lra |\Af(S,X)|
\]
is a $\G(S,X)$-equivariant homeomorphism, which
commutes with the projection onto $\Delta_X$.
Hence, $\ol{\Psi}_{JS}:\Mbar^\Delta_{g,X}\rar |\Af(S,X)|/\G(S,X)$
is a homeomorphism of orbispaces too.
\end{proposition}

A consequence of the previous proposition and of 
\ref{sss:arc-graph} is that the realization
$B\mathfrak{RG}_{g,X,ns}$ is the classifying space
of $\G(S,X)$ and that $B\mathfrak{RG}_{g,X}\rar \Mbar_{g,X}$
is a homotopy equivalence (in the orbifold category).

%%%%%%%%%%%%%%%%%%%%%%%%%%%%%%%%%%%%%%%%%%%%%%%%%%%%%%%%%%%%%%%%%%%%%%%%%%
%%%%%%%%%%%%%%%%%%%%%%%%%%%%%%%%%%%%%%%%%%%%%%%%%%%%%%%%%%%%%%%%%%%%%%%%%
\subsection{Penner-Bowditch-Epstein construction}\label{ss:pbe}
%%%%%%%%%%%%%%%%%%%%%%%%%%%%%%%%%%%%%%%%%%%%%%%%%%%%%%%%%%%%%%%%%%%%%%%%%

The other traditional way to obtain a weighted arc system out
of a Riemann surface with weighted marked points is to look at
the spine of the truncated surface obtained by removing horoballs
of prescribed circumference. Equivalently, to
decompose the surface into a union of hyperbolic cusps.

%%%%%%%%%%%%%%%%%%%%%%%%%%%%%%%%%%%%%%%%%%%%%%%%%%%%%%%%%%%%%%%%%%%%%%%%%
\subsubsection{Spines of hyperbolic surfaces.}
%%%%%%%%%%%%%%%%%%%%%%%%%%%%%%%%%%%%%%%%%%%%%%%%%%%%%%%%%%%%%%%%%%%%%%%%%

Let $[f:S\rar \Si]$ be an $(S,X)$-marked hyperbolic surface
and let $\up\in\Delta_X$.
Call $H_i\subset\Si$ the horoball at $x_i$ with circumference $p_i$
(as $p_i\leq 1$, the horoball is embedded in $\Si$)
and let $\Si_{tr}=\Si\setminus \bigcup_i H_i$ be the
{\it truncated surface}.
The datum $(\Si,\pa H_1,\dots,\pa H_n)$ is also called
a {\it decorated surface}.

For every $y\in \Si\setminus X$ at finite distance
from $\pa\Si_{tr}$, let the {\it valence} $\mathrm{val}(y)$
be the number of paths that realize $\mathrm{dist}(y,\pa\Si_{tr})$,
which is generically $1$. We will call a {\it projection} of $y$
a point on $\pa\Si_{tr}$ which is at shortest distance from $y$:
clearly, there are $\mathrm{val}(y)$ of them.

Let the {\it spine} $\mathrm{Sp}(\Si,\up)$ be the locus of points
of $\Si$ which are at finite distance from $\pa\Si_{tr}$
and such that $\mathrm{val}(y)\geq 2\}$
(see Figure~\ref{fig:hcoord}).
In particular, $\mathrm{val}^{-1}(2)$
is a disjoint union of finitely many geodesic arcs
(the {\it edges}) and $\mathrm{val}^{-1}([3,\infty))$ is a finite
collection of points (the {\it vertices}).
If $p_i=0$, then we include $x_i$ in $\mathrm{Sp}(\Si,\up)$
and we consider it a vertex. Its valence is defined to be
the number of half-edges of the spine incident at $x_i$.

There is a deformation retraction of $\Si_{tr}\cap \Si_+$
(where $\Si_+$ is the visible subsurface) onto
$\mathrm{Sp}(\Si,\up)$, defined on $\mathrm{val}^{-1}(1)$
simply flowing away from $\pa\Si_{tr}$ along the unique
geodesic that realizes the distance from $\pa\Si_{tr}$.

This shows that $\mathrm{Sp}(\Si,\up)$ defines an $(S,X)$-marked
enriched ribbon graph $\GG^{en}_{sp}$. By arc-graph duality, we also have
an associated {\it spinal arc system} $\ua_{sp}\in \Af(S,X)$.

%%%%%%%%%%%%%%%%%%%%%%%%%%%%%%%%%%%%%%%%%%%%%%%%%%%%%%%%%%%%%%%%%%%%%%%%%
\subsubsection{Horocyclic lengths and weights.}
%%%%%%%%%%%%%%%%%%%%%%%%%%%%%%%%%%%%%%%%%%%%%%%%%%%%%%%%%%%%%%%%%%%%%%%%%

As $\Si$ is a hyperbolic surface, we could metrize $\mathrm{Sp}(\Si,\up)$
by inducing a length on each edge. However, the relation between
this metric and $\up$ would be a little involved.

Instead, for every edge $e$ of $\GG^{en}_{sp}$
(that is, of $\mathrm{Sp}(\Si,\up)$),
consider one of its two projections $pr(e)$ to $\pa\Si_{tr}$
and define $\ell(e)$ to the be hyperbolic length of $pr(e)$,
which clearly does not depend on the chosen projection.
Thus, the boundary weights vector $\ell_{\pa}$ is exactly $\up$.

\begin{center}
{\large
\begin{figurehere}
\psfrag{Str}{$\Si_{tr}$}
\psfrag{ai}{{\color{Red}$\a_i$}}
\psfrag{ei}{{\color{Blue}$e_i$}}
\psfrag{wi}{{\color{MidnightBlue}$w_i$}}
\includegraphics[width=0.7\textwidth]{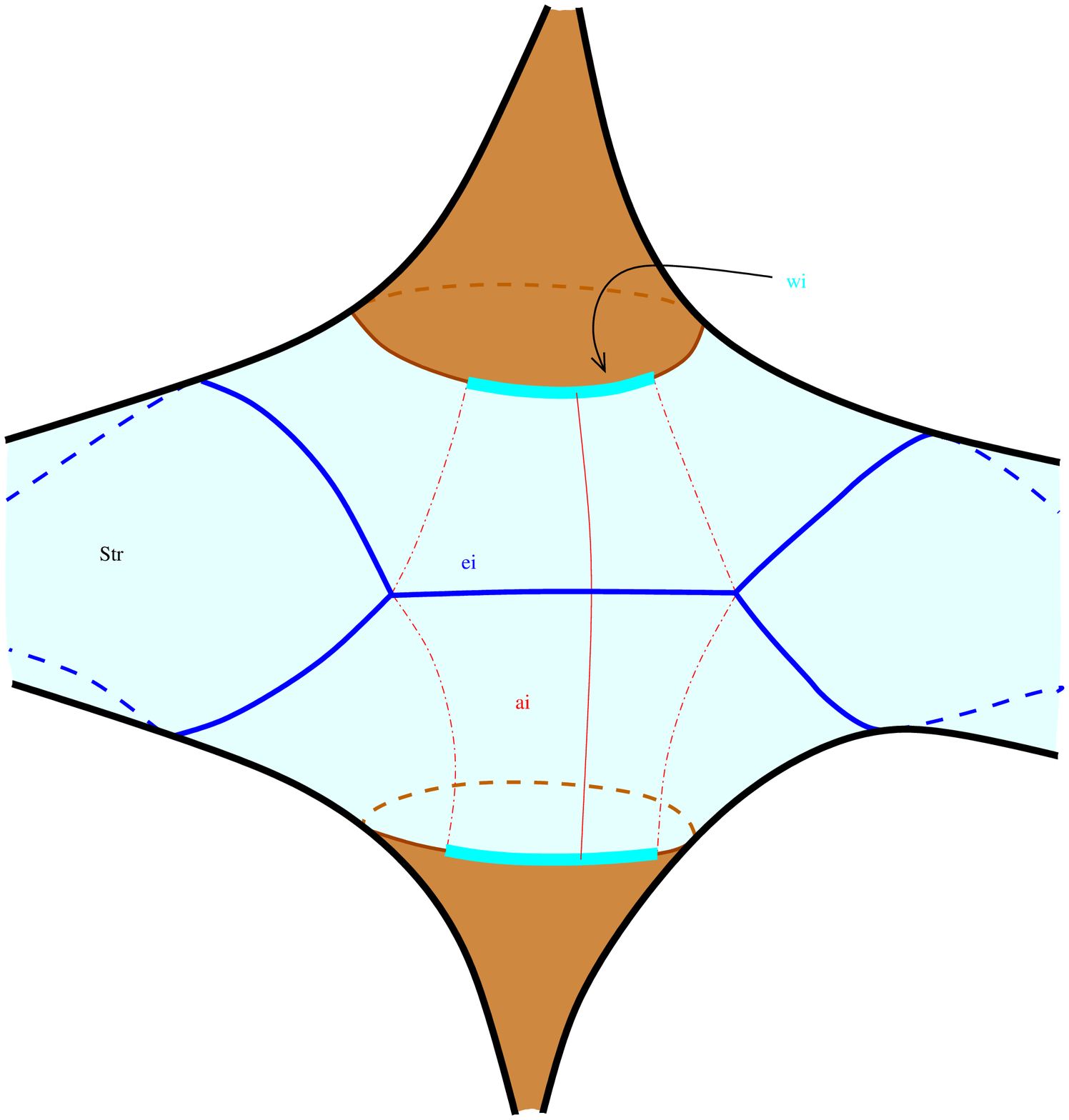}
\caption{Weights come from lengths of horocyclic arcs.}
\label{fig:hcoord}
\end{figurehere}
}
\end{center}

This endows $\GG^{en}_{sp}$ with a metric and so $\ua_{sp}$
with a projective weight $\ol{w}_{sp}\in |\Af(S,X)|$.
Notice that visibly equivalent surfaces are associated
the same point of $|\Af(S,X)|$.

This defines a $\G(S,X)$-equivariant map
\[
\Phi_{0}: \ol{\Teich}^\Delta(S,X) \lra |\Af(S,X)|
\]
that commutes with the projection onto $\Delta_X$.

Penner \cite{penner:decorated}
proved that the restriction of $\Phi_{0}$ to $\Teich(S,X)\times\Delta^\circ$
is a homeomorphism; the proof that $\Phi_{0}$ is a homeomorphism first
appears in
Bowditch-Epstein's \cite{bowditch-epstein:natural}
(and a very detailed treatment will appear in \cite{acgh:II}).
We refer to these papers for a proof of this result.

%%%%%%%%%%%%%%%%%%%%%%%%%%%%%%%%%%%%%%%%%%%%%%%%%%%%%%%%%%%%%%%%%%%%%%%%%%
%%%%%%%%%%%%%%%%%%%%%%%%%%%%%%%%%%%%%%%%%%%%%%%%%%%%%%%%%%%%%%%%%%%%%%%%%
\subsection{Hyperbolic surfaces with boundary}\label{ss:boundary}
%%%%%%%%%%%%%%%%%%%%%%%%%%%%%%%%%%%%%%%%%%%%%%%%%%%%%%%%%%%%%%%%%%%%%%%%%

The purpose of this informal subsection is to briefly illustrate
the bridge
between the cellular decomposition of the Teichm\"uller space
obtained using Jenkins-Strebel differentials and that obtained
using spines of decorated surfaces.

%%%%%%%%%%%%%%%%%%%%%%%%%%%%%%%%%%%%%%%%%%%%%%%%%%%%%%%%%%%%%%%%%%%%%%%%%
\subsubsection{Teichm\"uller and moduli space of hyperbolic surfaces.}
%%%%%%%%%%%%%%%%%%%%%%%%%%%%%%%%%%%%%%%%%%%%%%%%%%%%%%%%%%%%%%%%%%%%%%%%%

Fix $S$ a compact oriented surface as before and $X=\{x_1,\dots,x_n\}
\subset S$ a nonempty subset.

A (stable) hyperbolic surface $\Si$ is a nodal surface
such that $\Si\setminus\{\text{nodes}\}$ is hyperbolic
with geodesic boundary and/or cusps.
Notice that, by convention, $\pa\Si$ does not include the
possible nodes of $\Si$.

An $X$-marking
of a (stable) hyperbolic surface $\Si$ is
a bijection $X \rar \pi_0(\pa\Si)$.

An $(S,X)$-marking of the (stable) hyperbolic surface $\Si$ 
is an isotopy class of maps $f:S\setminus X \rar \Si$,
that may shrink disjoint simple
closed curves to nodes and are homeomorphisms onto
$\Si\setminus(\pa\Si\cup\{\text{nodes}\})$ elsewhere.

Let $\ol{\Teich}^{\pa}(S,X)$ be the Teichm\"uller space
of $(S,X)$-marked stable hyperbolic surfaces.
There is a natural map $\ell_{\pa}:\ol{\Teich}^{\pa}(S,X)\rar \R_{\geq 0}^X$
that associates to $[f:S\rar \Si]$ the boundary lengths of $\Si$,
which thus descends to $\ol{\ell}_{\pa}:\Mbar^{\pa}_{g,X}\rar \R_{\geq 0}^X$.
Call $\ol{\Teich}^{\pa}(S,X)(\up)$
(resp. $\Mbar^{\pa}_{g,X}(\up)$) the leaf $\ell_{\pa}^{-1}(\up)$
(resp. $\ol{\ell}_{\pa}^{-1}(\up)$).

There is an obvious identification between $\ol{\Teich}^{\pa}(S,X)(0)$
(resp. $\Mbar^{\pa}_{g,X}(0)$) and $\ol{\Teich}(S,X)$ (resp.
$\Mbar_{g,X}$).

Call $\widehat{\M}_{g,X}$ the blow-up of $\Mbar^{\pa}_{g,X}$ along
$\Mbar^{\pa}_{g,X}(0)$: the exceptional locus
can be naturally identified to
the space of decorated surfaces with cusps (which
is homeomorphic to $\Mbar_{g,X}\times\Delta_X$).
Define similarly, $\widehat{\Teich}(S,X)$.

%%%%%%%%%%%%%%%%%%%%%%%%%%%%%%%%%%%%%%%%%%%%%%%%%%%%%%%%%%%%%%%%%%%%%%%%%
\subsubsection{Tangent space to the moduli space.}
%%%%%%%%%%%%%%%%%%%%%%%%%%%%%%%%%%%%%%%%%%%%%%%%%%%%%%%%%%%%%%%%%%%%%%%%%

The conformal analogue of a hyperbolic surface with geodesic
boundary $\Si$ is a Riemann surface with real boundary.
In fact, the double of $\Si$ is a hyperbolic surface with
no boundary and an orientation-reversing involution,
that is a Riemann surface with an anti-holomorphic involution.
As a consequence, $\pa\Si$ is a real-analytic submanifold.

This means that first-order deformations are determined
by Beltrami differentials on $\Si$ which are real on $\pa \Si$,
and so $T_{[\Si]}\Mbar^{\pa}_{g,X}\cong H^{0,1}(\Si,T_{\Si})$,
where $T_{\Si}$ is the sheaf of tangent vector fields
$V=V(z)\pa/\pa z$, which are real on $\pa\Si$.

Dually, the cotangent space $T^{\vee}_{[\Si]}\Mbar^{\pa}_{g,X}$ is given
by the space $\mathcal{Q}(\Si)$ of holomorphic quadratic
differentials that are real on $\pa\Si$.
If we call $\mathcal{H}(\Si)=\{\ol{\varphi}/\lambda\,|\,
\varphi\in\mathcal{Q}(\Si)\}$, where $\lambda$ is the hyperbolic
metric on $\Si$, then $H^{0,1}(\Si,T_{\Si})$ identifies
to the space of harmonic Beltrami differentials $\mathcal{H}(\Si)$.

As usual, if $\Si$ has a node, then quadratic differentials
are allowed to have a double pole at the node, with the
same quadratic residue on both branches.

If a boundary component of $\Si$ collapses to a cusp $x_i$,
then the cotangent cone to $\Mbar^{\pa}_{g,X}$ at $[\Si]$
is given by quadratic differentials that may have at worst
a double pole at $x_i$ with positive residue.
The phase of the residue being zero corresponds to the fact that,
if we take Fenchel-Nielsen coordinates on the double of $\Si$
which are symmetric under the real involution, then
the twists along $\pa\Si$ are zero.

%%%%%%%%%%%%%%%%%%%%%%%%%%%%%%%%%%%%%%%%%%%%%%%%%%%%%%%%%%%%%%%%%%%%%%%%%
\subsubsection{Weil-Petersson metric.}
%%%%%%%%%%%%%%%%%%%%%%%%%%%%%%%%%%%%%%%%%%%%%%%%%%%%%%%%%%%%%%%%%%%%%%%%%

Mimicking what done for surfaces with cusps,
we can define Hermitean pairings on $\mathcal{Q}(\Si)$
and $\mathcal{H}(\Si)$, where $\Si$ is a hyperbolic surface
with boundary.
In particular,
\begin{align*}
h(\mu,\nu)        & =\int_{\Si} \mu\,\ol{\nu}\cdot\lambda\\
h^{\vee}(\varphi,\psi) & =\int_{\Si} \frac{\varphi\,\ol{\psi}}{\lambda}
\end{align*}
where $\mu,\nu\in\mathcal{H}(\Si)$ and $\varphi,\psi\in\mathcal{Q}(\Si)$.

Thus, if $h=g+i\omega$, then $g$ is the Weil-Petersson Riemannian metric
and $\omega$ is the Weil-Petersson form. Write similarly
$h^{\vee}=g^{\vee}+i\omega^{\vee}$, where $g^{\vee}$
is the cometric dual to $g$
and $\omega^{\vee}$ is the Weil-Petersson bivector field.

Notice that $\omega$ and $\omega^{\vee}$ are degenerate. This can
be easily seen, because Wolpert's formula
$\omega=\sum_i d\ell_i\wedge d\tau_i$ still holds.
We can also conclude that the symplectic leaves of $\omega^{\vee}$
are exactly the fibers of the boundary length map $\ell_{\pa}$.

%%%%%%%%%%%%%%%%%%%%%%%%%%%%%%%%%%%%%%%%%%%%%%%%%%%%%%%%%%%%%%%%%%%%%%%%%
\subsubsection{Spines of hyperbolic surfaces with boundary.}
%%%%%%%%%%%%%%%%%%%%%%%%%%%%%%%%%%%%%%%%%%%%%%%%%%%%%%%%%%%%%%%%%%%%%%%%%

The spine construction can be carried on, even in a more natural
way, on hyperbolic surfaces with geodesic boundary.

In fact, given such a $\Si$ whose boundary components
are called $x_1,\dots,x_n$, we can define the distance from
$\pa\Si$ and so the valence of a point in $\Si$ and consequently
the spine $\mathrm{Sp}(\Si)$, with no need of further information.

Similarly, if $\Si$ has also nodes (that is, some holonomy
degenerates to a parabolic element), then $\mathrm{Sp}(\Si)$
is embedded inside the {\it visible components of $\Si$},
i.e. those components of $\Si$ that contain a boundary circle
of positive length.

The weight of an arc $\a_i\in \ua_{sp}$ dual to the
edge $e_i$ of $\mathrm{Sp}(\Si)$ is still defined
as the hyperbolic length of one of the two
projections of $e_i$ to $\pa\Si$.
Thus, the construction above gives a point
$w_{sp} \in |\Af(S,X)|\times(0,\infty)$.

\begin{center}
{\large
\begin{figurehere}
\psfrag{S}{$\Si$}
\psfrag{ai}{{\color{Red}$\a_i$}}
\psfrag{ei}{{\color{Blue}$e_i$}}
\psfrag{wi}{{\color{MidnightBlue}$w_i$}}
\includegraphics[width=0.7\textwidth]{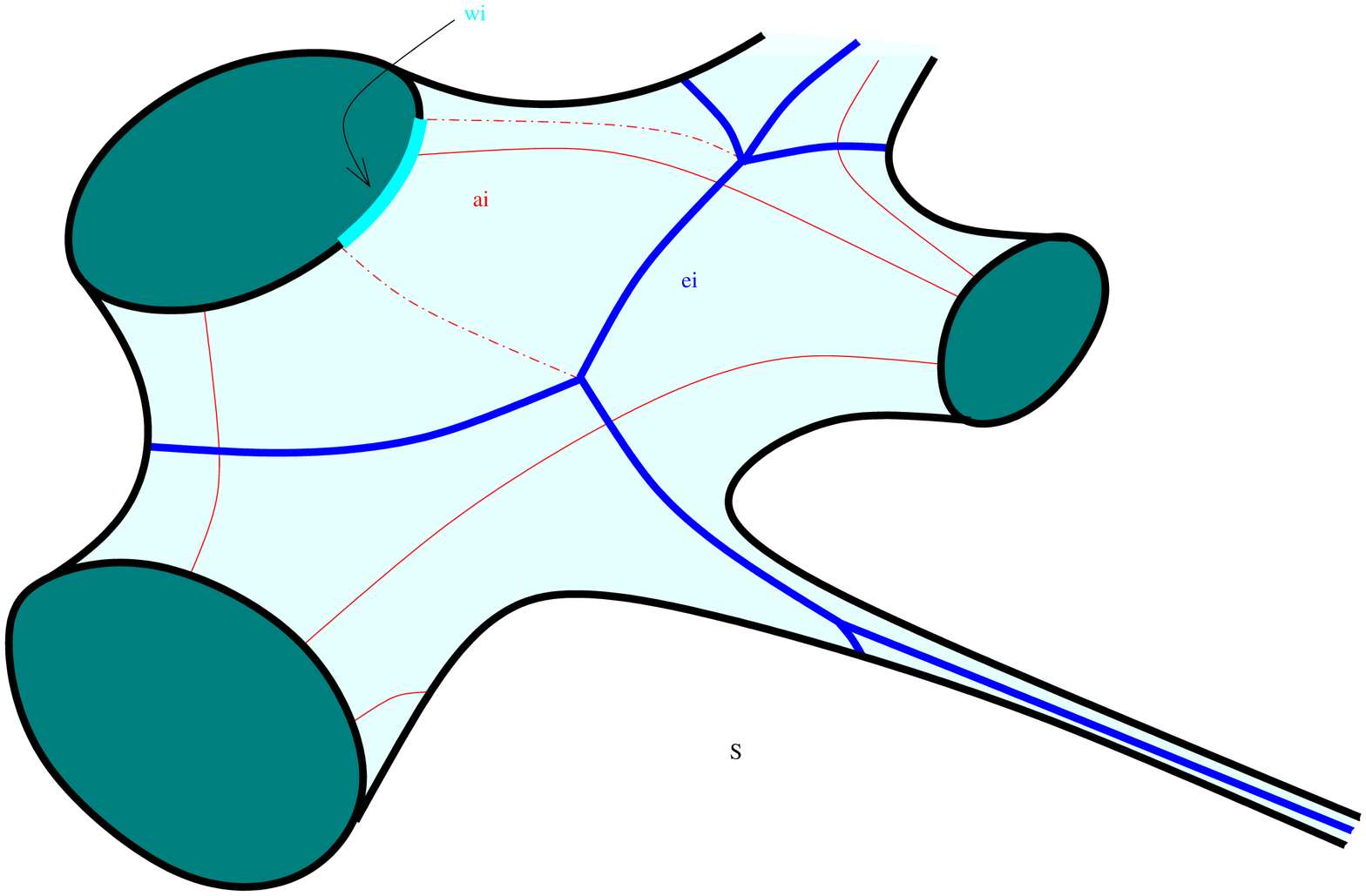}
\caption{Weights come from lengths of geodesic boundary arcs.}
\label{fig:surf-bound}
\end{figurehere}
}
\end{center}

It is easy to check (see \cite{mondello:wp} or \cite{mondello:triang})
that $w_{sp}$ converges to the $\ol{w}_{sp}$ defined above
when the hyperbolic surface with boundary converges to
a decorated surface with cusps in $\widehat{\Teich}(S,X)$.
Thus, the $\G(S,X)$-equivariant map
\[
\Phi:\widehat{\Teich}(S,X)\lra |\Af(S,X)|\times[0,\infty)
\]
reduces to $\Phi_0$ for decorated surfaces with cusps.

\begin{theorem}[Luo \cite{luo:decomposition}]
The restriction of $\Phi$ to smooth surfaces with no boundary
cusps gives a homeomorphism onto its image.
\end{theorem}

The continuity of the whole $\Phi$ is proven in \cite{mondello:triang},
using Luo's result.\\

The key point of Luo's proof is the following.
Pick a generic hyperbolic surface with geodesic boundary $\Si$
and suppose that the spinal arc system is the
ideal triangulation $\ua_{sp}=\{\a_1,\dots,\a_M\}\in\Ao(\Si,X)$
with weight $w_{sp}$.
We can define the length $\ell_{\a_i}$ as the hyperbolic length
of the shortest geodesic $\tilde{\a}_i$
in the free homotopy class of $\a_i$.

The curves $\{\tilde{\a}_i\}$ cut $\Si$ into hyperbolic hexagons,
which are completely determined by $\{\ell_{\b_1},\dots
\ell_{\b_{2M}}\}$, where the $\b_j$'s are the
sides of the hexagons lying on $\pa\Si$.
Unfortunately, going from the $\ell_{\b_j}$'s to $w_{sp}$
is much easier than the converse. In fact, $w_{\a_1},\dots,
w_{\a_M}$ can be written as explicit
linear combinations of the $\ell_{\b_j}$'s:
in matrix notation, $B=(\ell_{\b_j})$ is
a solution of the system $W=RB$, where
$R$ is a fixed $(M\times 2M)$-matrix
(that encodes the combinatorics is $\ua_{sp}$)
and $W=(w_{\a_i})$.
Clearly, there is a whole affine space $E_W$ of dimension
$M$ of solutions of $W=RB$.
The problem is that a random point in $E_W$ would
determine hyperbolic structures on the hexagons
of $\Si\setminus\ua_{sp}$ that do not glue, because
we are not requiring the two sides of each $\a_i$
to have the same length.

Starting from very natural quantities associated to hyperbolic
hexagons with right angles, Luo defines a functional
on the space $(b_1,\dots,b_{2M})\in \R_{\geq 0}^{2M}$.
For every $W$, the space $E_W$ is not empty
(which proves the surjectivity of $\Phi$)
and the restriction of Luo's functional to $E_W$
is strictly concave
and achieves
its (unique) maximum exactly when $B=(\ell_{\b_j})$
(which proves the injectivity of $\Phi$).

The geometric meaning of this functional is still not entirely
clear, but it seems related to some volume of a three-dimensional
hyperbolic manifold associated to $\Si$.
Quite recently, Luo \cite{luo:rigidity}
(see also \cite{guo:parametrizations}) has introduced
a modified functional $F_{c}$, which depends on a parameter $c\in \R$,
and he has produced other realizations
of the Teichm\"uller space as a polytope,
and so different systems of ``simplicial'' coordinates.

%%%%%%%%%%%%%%%%%%%%%%%%%%%%%%%%%%%%%%%%%%%%%%%%%%%%%%%%%%%%%%%%%%%%%%%%%
\subsubsection{Surfaces with large boundary components.}
%%%%%%%%%%%%%%%%%%%%%%%%%%%%%%%%%%%%%%%%%%%%%%%%%%%%%%%%%%%%%%%%%%%%%%%%%

To close the circle, we must relate the limit of
$\Phi$ for surfaces whose boundary lengths diverge
to $\Psi_{JS}$. This is the topic of \cite{mondello:triang}.
Here, we only sketch the main ideas. To simplify the
exposition, we will only deal with smooth surfaces.

Consider an $X$-marked
hyperbolic surface with geodesic boundary $\Si$.
Define $\gr8 (\Si)$ to be the surface
obtained by gluing semi-infinite flat cylinders at $\pa\Si$
of lengths $(p_1,\dots,p_n)=\ell_{\pa}(\Si)$.

Thus, $\gr8(\Si)$ has a hyperbolic core and flat ends and the
underlying conformal structure is that of an $X$-punctured
Riemann surface. This {\it grafting} procedure defines a map
\[
( \gr8 ,\ell_{\pa}):\Teich^{\pa}(S,X)
\lra \Teich(S,X)\times \R_{\geq 0}^N
\]

\begin{center}
{\large
\begin{figurehere}
\psfrag{S}{$\Si$}
\includegraphics[width=0.6\textwidth]{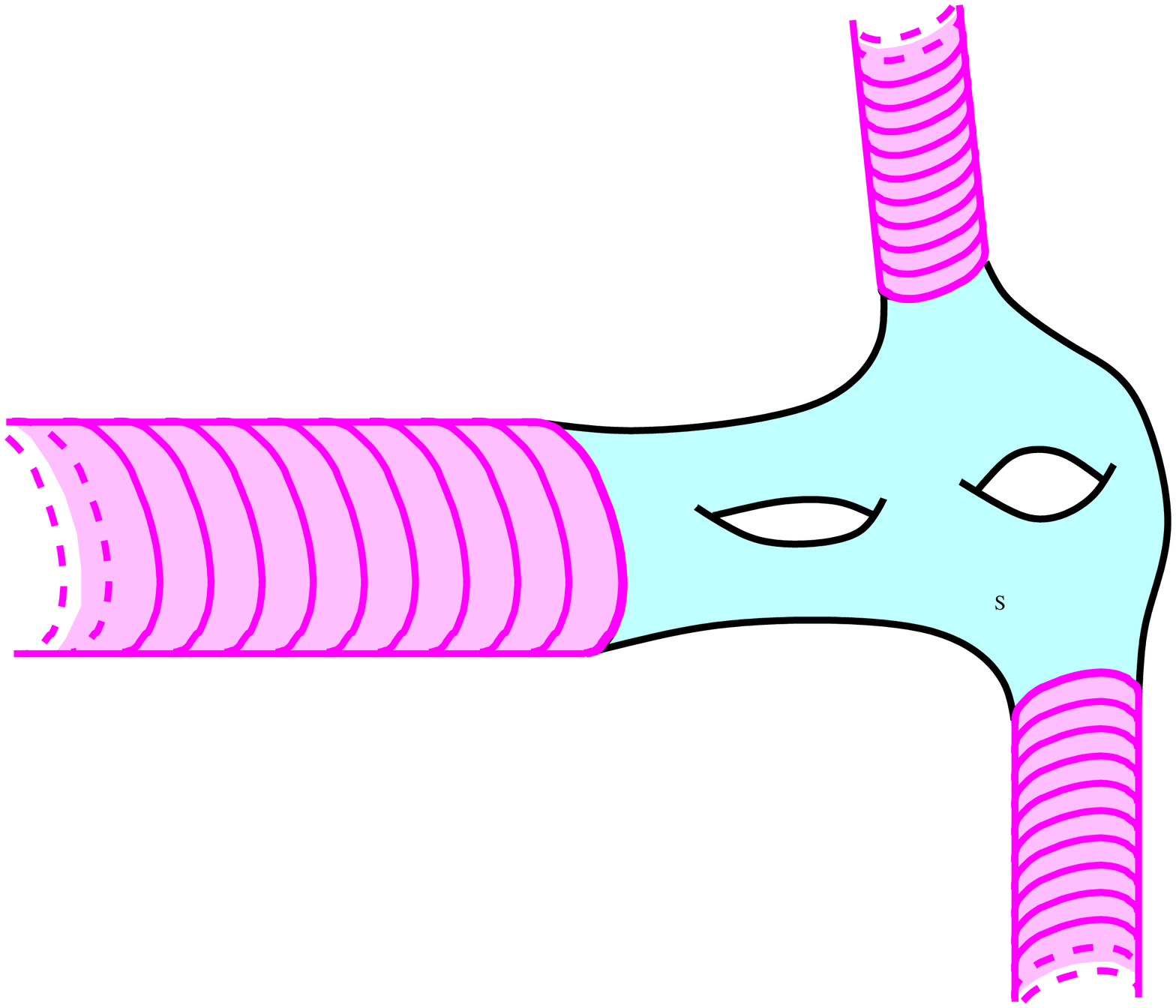}
\caption{A grafted surface $\gr8(\Si)$.}
\label{fig:grafted}
\end{figurehere}
}
\end{center}

\begin{proposition}[\cite{mondello:triang}]
The map $( \gr8 ,\ell_{\pa})$ is a $\G(S,X)$-equivariant
homeomorphism.
\end{proposition}

The proof is a variation of Scannell-Wolf's
\cite{scannell-wolf:grafting} that finite grafting is
a self-homeomorphism of the Teichm\"uller space.

Thus, the composition of $(\gr8,\ell_{\pa})^{-1}$ and
$\Phi$ gives (after blowing up the locus $\{\ell_{\pa}=0\}$)
the homeomorphism
\[
\Psi:\Teich(S,X)\times \Delta_X\times[0,\infty) \lra |\Ao(S,X)|\times[0,\infty)
\]

\begin{proposition}[\cite{mondello:triang}]
The map $\Psi$ extends to a $\G(S,X)$-equivariant homeomorphism
\[
\Psi:\Teich(S,X)\times\Delta_X\times[0,\infty]\lra |\Ao(S,X)|\times[0,\infty]
\]
and $\Psi_{\infty}$ coincides with Harer-Mumford-Thurston's $\Psi_{JS}$.
\end{proposition}

The main point is to show that a surface $\Si$ with large boundaries
and with spine $\mathrm{Sp}(\Si)$ is very close in $\Teich(S,X)$
to the flat surface whose Jenkins-Strebel differential has
critical graph isomorphic to $\mathrm{Sp}(\Si)$ (as metrized
ribbon graphs).

To understand why this is reasonable, consider a sequence
of hyperbolic surfaces $\Si_m$
whose spine has fixed isomorphism type $\GG$
and fixed projective metric
and such that $\ell_{\pa}(\Si_m)=c_m(p_1,\dots,p_n)$,
where $c_m$ diverges as $m\rar \infty$.
Consider the grafted surfaces $\gr8(\Si_m)$ and
rescale them so that $\sum_i p_i=1$. The flat metric
on the cylinders is naturally induced by a holomorphic
quadratic differential, which has negative quadratic
residue at $X$. Extend this differential to zero on
the hyperbolic core.

Because of the rescaling,
the distance between the flat cylinders and
the spine goes to zero and the differential converges
in $L^1_{red}$ to a Jenkins-Strebel differential.

Dumas \cite{dumas:schwarzian} has shown that an
analogous phenomenon occurs for closed surfaces
grafted along a measured lamination $t\lambda$
as $t\rar +\infty$.

%%%%%%%%%%%%%%%%%%%%%%%%%%%%%%%%%%%%%%%%%%%%%%%%%%%%%%%%%%%%%%%%%%%%%%%%%
\subsubsection{Weil-Petersson form and Penner's formula.}
%%%%%%%%%%%%%%%%%%%%%%%%%%%%%%%%%%%%%%%%%%%%%%%%%%%%%%%%%%%%%%%%%%%%%%%%%

Using Wolpert's result and hyperbolic geometry,
Penner \cite{penner:WP-volumes} has proved that
pull-back of
the Weil-Petersson form on the space of decorated hyperbolic
surfaces with cusps, which can be identified to
$\Teich(S,X)\times\Delta_X$, can be neatly written 
in the following way.
Fix a triangulation $\ua=\{\a_1,\dots,\a_M\}\in\Ao(S,X)$.
For every $([f:S\rar \Si],\up)\in\Teich(S,X)\times\Delta_X$,
let $\tilde{\a}_i$ be the geodesic representative in
the class of $f_*(\a_i)$ and
call $a_i:=\ell(\tilde{\a}_i\cap\Si_{tr})$,
where $\Si_{tr}$ be the truncated hyperbolic surface.
Then
\[
\pi^*\omega_{WP}=\sum_{t\in T}( da_{t_1}\wedge da_{t_2}+
da_{t_2}\wedge da_{t_3}+da_{t_3}\wedge da_{t_1} )
\]
where $\pi:\Teich(S,X)\times\Delta_X\rar \Teich(S,X)$ is the
projection, $T$ is the set of ideal triangles in which
the $\tilde{\a}_i$'s decompose $\Si$, and the sides of
$t$ are $(\a_{t_1},\a_{t_2},\a_{t_3})$ in the cyclic
order induced by the orientation of $t$ (see Figure~\ref{fig:triang}).

\begin{center}
{\large
\begin{figurehere}
\psfrag{Str}{$\Si_{tr}$}
\psfrag{ai}{{\color{Red}$\a_{t_1}$}}
\psfrag{aj}{{\color{Red}$\a_{t_3}$}}
\psfrag{ak}{{\color{Red}$\a_{t_2}$}}
\psfrag{t}{$t$}
\includegraphics[width=0.6\textwidth]{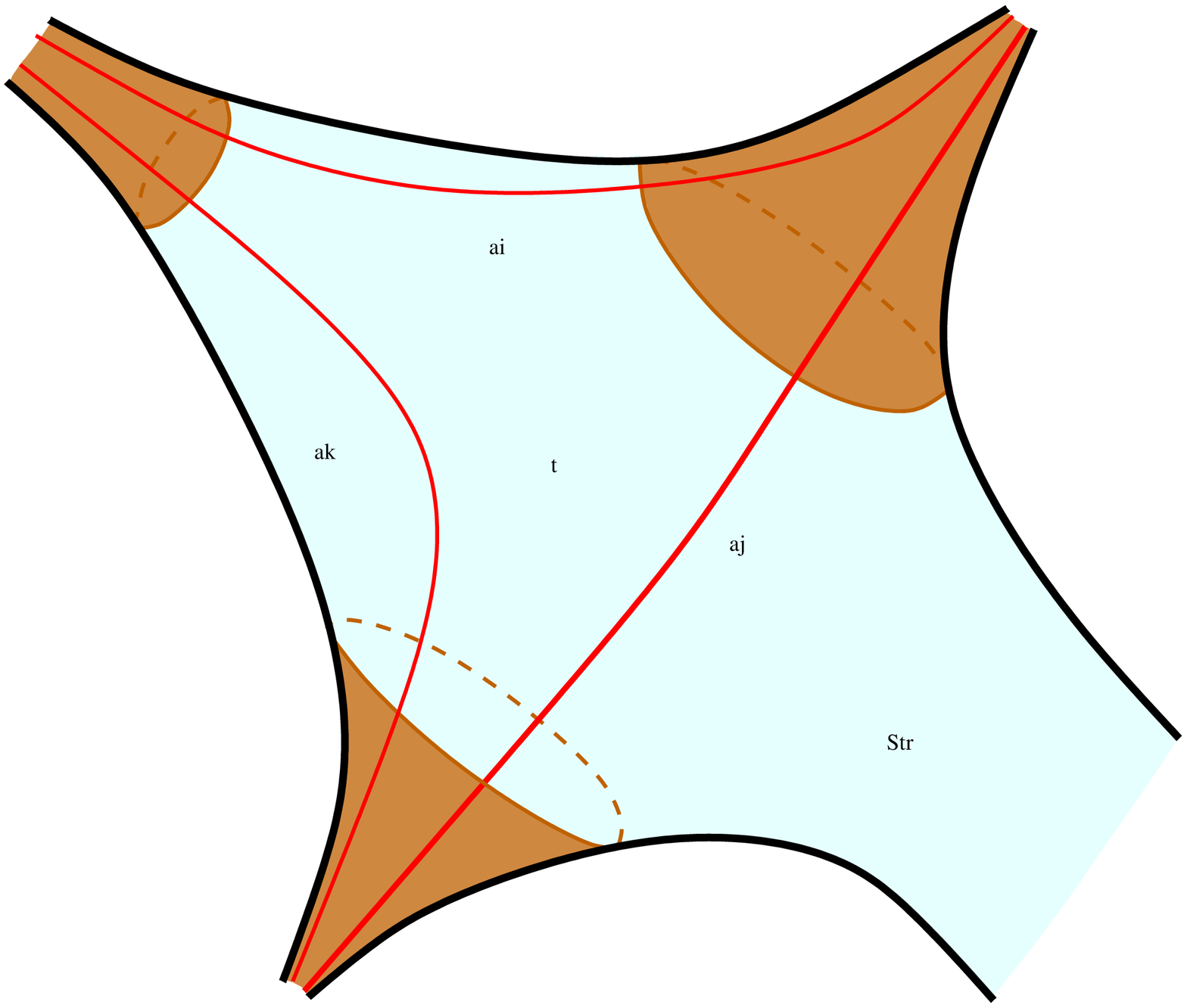}
\caption{An ideal triangle in $T$.}
\label{fig:triang}
\end{figurehere}
}
\end{center}

To work on $\M_{g,X}\times\Delta_X$ (for instance, to
compute Weil-Petersson volumes),
one can restrict to the interior of the cells
$\Phi_0^{-1}(|\ua|)$ whose associated system of arcs $\ua$
is triangulation and write the pull-back of $\omega_{WP}$
with respect to $\ua$.

%%%%%%%%%%%%%%%%%%%%%%%%%%%%%%%%%%%%%%%%%%%%%%%%%%%%%%%%%%%%%%%%%%%%%%%%%
\subsubsection{Weil-Petersson form for surfaces with boundary.}
%%%%%%%%%%%%%%%%%%%%%%%%%%%%%%%%%%%%%%%%%%%%%%%%%%%%%%%%%%%%%%%%%%%%%%%%%

Still using methods of Wolpert \cite{wolpert:symplectic},
one can generalize Penner's formula to hyperbolic
surfaces with boundary.
The result is better expressed
using the Weil-Petersson bivector field than the $2$-form.

\begin{proposition}[\cite{mondello:wp}]
Let $\Si$ be a hyperbolic surface with boundary components
$C_1,\dots,C_n$ and let $\ua=\{\a_1,\dots,\a_M\}$ be a
triangulation. Then the Weil-Petersson bivector
field can be written as
\[
\omega^{\vee}=\frac{1}{4}\sum_{b=1}^n
\sum_{\substack{y_i \in \a_i\cap C_b \\ y_j\in\a_j\cap C_b}}
\frac{\sinh(p_b/2-d_b(y_i,y_j))}{\sinh(p_b/2)}
\frac{\pa}{\pa a_i}\wedge \frac{\pa}{\pa a_j}
\]
where $a_i=\ell(\a_i)$ and $d_b(y_i,y_j)$ is the length of
the geodesic arc running from $y_i$ to $y_j$ along $C_b$
in the positive direction (according to the orientation
induced by $\Si$ on $C_b$).
\end{proposition}

The idea is to use Wolpert's formula
$\omega^{\vee}=-\sum_i \pa_{\ell_i}\wedge \pa_{\tau_i}$
on the double $d\Si$ of $\Si$ with the pair of pants
decomposition induced by doubling the arcs $\{\a_i\}$.
Then one must compute the (first-order) effect on
the $a_i$'s of twisting $d\Si$ along $\a_j$.

Though not immediate, the formula above can be shown
to reduce to Penner's, when the boundary lengths go to zero,
as we approximate $\sinh(x)\approx x$ for small $x$.
Notice that Penner's formula shows that $\omega$ linearizes
(with constant coefficients!) in the coordinates given
by the $a_i$'s.

More interesting is to analyze what happens for
$(\Si,t\up)$ with $\up\in\Delta_X$,
as $t\rar +\infty$.
Assume the situation is generic and so $\Psi_{JS}(\Si)$
is supported on a triangulation, whose dual graph is $\GG$.

Once again, the formula dramatically simplifies 
as we approximate $2\sinh(x)\approx \exp(x)$ for $x\gg 0$.
Under the rescalings $\dis \tilde{\omega}^{\vee}=c^2\omega^{\vee}$
and $\dis
\tilde{w}_i=w_i/c$ with $c=\sum_b p_b/2$,
we obtain that
\[
\lim_{t\rar \infty} \tilde{\omega}^{\vee}=
\omega^{\vee}_\infty:=\frac{1}{2}\sum_{v\in E_0(\GG)}
\left(
\frac{\pa}{\pa \tilde{w}_{v_1}}\wedge \frac{\pa}{\pa\tilde{w}_{v_2}}+
\frac{\pa}{\pa \tilde{w}_{v_2}}\wedge \frac{\pa}{\pa\tilde{w}_{v_3}}+
\frac{\pa}{\pa \tilde{w}_{v_3}}\wedge \frac{\pa}{\pa\tilde{w}_{v_1}}
\right)
\]
where $v=\{v_1,v_2,v_3\}$ and $\s_0(v_j)=v_{j+1}$ (and $j\in\Z/3\Z$).

\begin{center}
{\large
\begin{figurehere}
\psfrag{v}{$v$}
\psfrag{ei}{{\color{Blue}$e_{v_1}$}}
\psfrag{ej}{{\color{Blue}$e_{v_2}$}}
\psfrag{ek}{{\color{Blue}$e_{v_3}$}}
\includegraphics[width=0.4\textwidth]{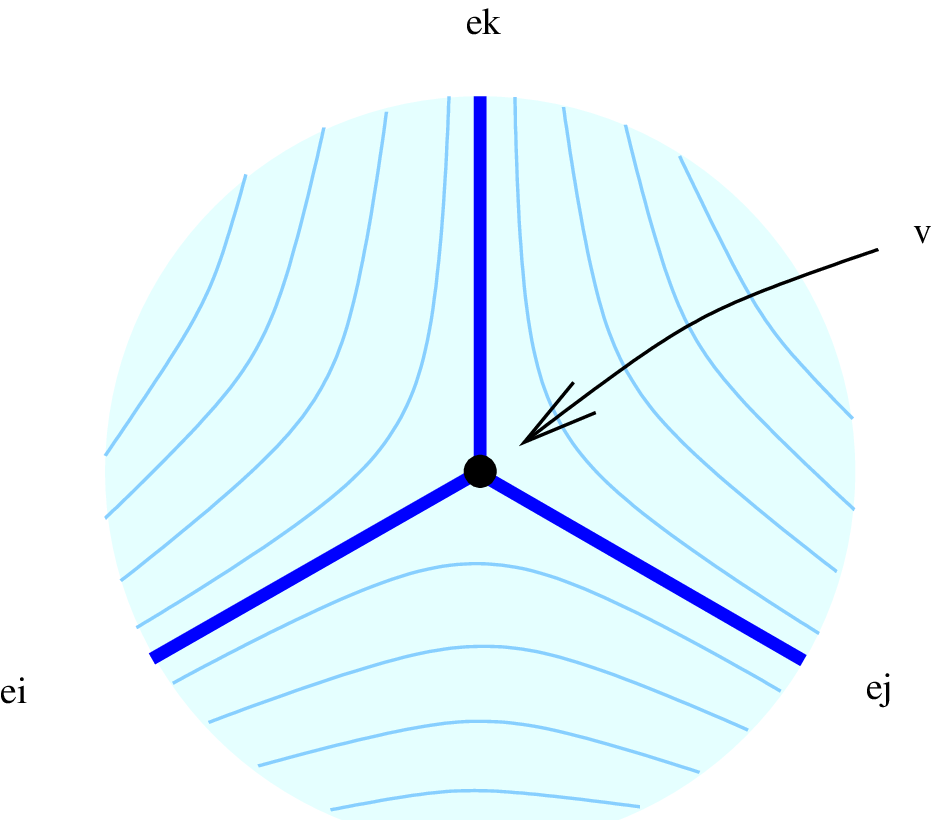}
\caption{A trivalent vertex $v$ of $\GG$.}
\label{fig:trivalent}
\end{figurehere}
}
\end{center}

Thus, the Weil-Petersson symplectic structure is again linearized
(and with constant coefficients!), but in the system of coordinates
given by the $w_j$'s, which are in some sense dual to the $a_i$'s.

It would be nice to exhibit a clear geometric argument for
the perfect symmetry of these two formulae.

%%%%%%%%%%%%%%%%%%%%%%%%%%%%%%%%%%%%%%%%%%%%%%%%%%%%%%%%%%%%%%%%%%%%%%%%%%%
%%%%%%%%%%%%%%%%%%%%%%%%%%%%%%%%%%%%%%%%%%%%%%%%%%%%%%%%%%%%%%%%%%%%%%%%%%%
\section{Combinatorial classes}\label{sec:combinatorial}
%%%%%%%%%%%%%%%%%%%%%%%%%%%%%%%%%%%%%%%%%%%%%%%%%%%%%%%%%%%%%%%%%%%%%%%%%%%
%%%%%%%%%%%%%%%%%%%%%%%%%%%%%%%%%%%%%%%%%%%%%%%%%%%%%%%%%%%%%%%%%%%%%%%%%%%
\subsection{Witten cycles.}\label{ss:witten}
%%%%%%%%%%%%%%%%%%%%%%%%%%%%%%%%%%%%%%%%%%%%%%%%%%%%%%%%%%%%%%%%%%%%%%%%%%%

Fix as usual a compact oriented surface $S$ of genus $g$ and a subset
$X=\{x_1,\dots,x_n\}\subset S$ such that $2g-2+n>0$.

We introduce some remarkable $\G(S,X)$-equivariant subcomplexes 
of $\Af(S,X)$, which define interesting cycles in the homology
of $\Mbar^K_{g,X}$ as well as in the Borel-Moore
homology of $\M_{g,X}$ and so, by Poincar\'e duality,
in the cohomology of $\M_{g,X}$ (that is, of $\G(S,X)$).

These subcomplexes are informally defined as the locus
of points of $|\Ao(S,X)|$, whose associated ribbon graphs
have prescribed odd valences of their vertices.
It can be easily shown that, if we assign even valence to some vertex,
the subcomplex we obtain is not a cycle (even with $\Z/2\Z$ coefficients!).

We follow Kontsevich (\cite{kontsevich:intersection}) for the
orientation of the combinatorial cycles, but an alternative
way is due to Penner \cite{penner:poincare} and
Conant and Vogtmann \cite{conant-vogtmann:onatheorem}.

Later, we will mention
a slight generalization of the combinatorial classes by
allowing some vertices to be marked.

Notice that we are going to use the cellularization of the moduli
space of curves given by $\Psi_{JS}$, and so we will identify
$\Mbar^\Delta_{g,X}$ with the orbispace $|\Af(S,X)|/\G(S,X)$.
As the arguments will be
essentially combinatorial/topological, any of the decompositions
described before would work.

%%%%%%%%%%%%%%%%%%%%%%%%%%%%%%%%%%%%%%%%%%%%%%%%%%%%%%%%%%%%%%%%%%%%%%%%%
\subsubsection{Witten subcomplexes.}
%%%%%%%%%%%%%%%%%%%%%%%%%%%%%%%%%%%%%%%%%%%%%%%%%%%%%%%%%%%%%%%%%%%%%%%%%

Let
$m_*=(m_0,m_1,\dots)$ be a sequence of nonnegative integers
such that
\[
\sum_{i\geq 0} (2i+1)m_i=4g-4+2n
\]
and define $(m_*)!:=\prod_{i\geq 0}m_i!$ and $r:=\sum_{i\geq 0} i\, m_i$.

\begin{definition}
The {\it combinatorial subcomplex} $\Af_{m_*}(S,X)\subset \Af(S,X)$
is the smallest simplicial subcomplex that contains
all proper simplices $\ua\in\Ao(S,X)$ such that
$S\setminus\ua$ is the disjoint union of exactly $m_i$
polygons with $2i+3$ sides.
\end{definition}

It is convenient to set $|\Af_{m_*}(S,X)|_{\R}:=|\Af_{m_*}(S,X)|\times\R_+$.
Clearly, this subcomplex is $\G(S,X)$-equivariant.
Hence, if we call
$\Mbarcomb_{g,X}:=\Mbar^\Delta_{g,X}\times\R_+\cong|\Af(S,X)|_{\R}/\G(S,X)$,
then we can define $\Mbarcomb_{m_*,X}$ to be the subcomplex
of $\Mbarcomb_{g,X}$ induced by $\Af_{m_*}(S,X)$.

\begin{remark}
We can introduce also univalent vertices by allowing
$m_{-1}>0$. It is still possible to define
the complexes $\Af_{m_*}(S,X)$
and $\mathfrak{A}^{\circ}_{m_*}(S,X)$, just
allowing (finitely many) contractible loops (i.e. unmarked tails
in the corresponding ribbon graph picture).
However, $\Af_{m_*}(S,X)$ would no longer be
a subcomplex of $\Af(S,X)$. Thus, we should construct
an associated family of Riemann surfaces
over $\Mbarcomb_{m_*,X}$ (which can be easily done)
and consider the
classifying map $\Mbarcomb_{m_*,X} \rar \Mbarcomb_{g,X}$,
whose existence is granted by the universal property
of $\Mbar_{g,X}$, but
which would no longer be cellular.
\end{remark}

For every $\up \in \Delta_X\times\R_+$
call
$\Mbarcomb_{g,X}(\up):=\bar{\ell}_{\pa}^{-1}(\up)
\subset \Mbarcomb_{g,X}$
and define $\Mbarcomb_{m_*,X}(\up):=\Mbarcomb_{m_*,X}\cap
\Mbarcomb_{g,X}(\up)$.

Notice that the dimensions of the slices are
the expected ones because in every cell they are described
by $n$ independent linear equations.
%
%%%%%%%%%%%%%%%%%%%%%%%%%%%%%%%%%%%%%%%%%%%%%%%%%%%%%%%%%%%%%%%%%%%
\subsubsection{Combinatorial $\psi$ classes.}
%%%%%%%%%%%%%%%%%%%%%%%%%%%%%%%%%%%%%%%%%%%%%%%%%%%%%%%%%%%%%%%%%%%

Define $L_i$ as the space of couples $(\GG,y)$, where
$\GG$ is a $X$-marked metrized ribbon graph
in $\Mbarcomb_{g,X}(\{p_i>0\})$ and
$y$ is a point of $|G|\subset |\GG|$ belonging to an edge that
borders the $x_i$-th hole.

Clearly $L_i \lra \Mbarcomb_{g,X}(\{p_i>0\})$ is a topological
bundle with fiber homeomorphic to $S^1$. It is easy to see that,
for a fixed $\up\in \D_X\times\R_+$ such that $p_i>0$,
the pull-back of $L_i$ via
\[
\xi_{\up}:\Mbar_{g,X} \lra \Mbarcomb_{g,X}(\up)
\]
is isomorphic (as a topological bundle) to the sphere
bundle associated to $\mathcal{L}^{\vee}_i$.

\begin{lemma}[\cite{kontsevich:intersection}]\label{lemma:eta}
Fix $x_i$ in $X$ and $\up\in \Delta_X\times\R_+$
such that $p_i>0$.
Then on every simplex $|\ua|(\up) \in \Mbarcomb_{g,X}(\up)$ define
\[
\ol{\eta}_i|_{|\ua|(\up)}:=\sum_{1\leq s<t\leq k-1}
d\tilde{e}_s \wedge d\tilde{e}_t
\]
where $\dis\tilde{e}_j=\frac{\ell(e_j)}{2 p_i}$
and
$x_i$ marks a hole with cyclically ordered sides
$(e_1,\dots,e_k)$. These $2$-forms glue to
give a piecewise-linear $2$-form
$\ol{\eta}_i$ on $\Mbarcomb_{g,X}(\up)$,
that represents $c_1(L_i)$. Hence,
the pull-back class
$\xi_{\up}^*[\ol{\eta}_i]$ is exactly
$\psi_i=c_1(\mathcal{L}_i)$ in $H^2(\Mbar_{g,X})$.
\end{lemma}
\begin{center}
\begin{figurehere}
\psfrag{xi}{$x_i$}
\psfrag{e1}{$e_1$}
\psfrag{e2}{$e_2$}
\psfrag{e3}{$e_3$}
\psfrag{e4}{$e_4$}
\psfrag{e5}{$e_5$}
\psfrag{e6}{$e_6$}
\psfrag{e7}{$e_7$}
\psfrag{p1}{$\ol{\phi}_1$}
\psfrag{p2}{$\ol{\phi}_2$}
\psfrag{p3}{$\ol{\phi}_3$}
\psfrag{p4}{$\ol{\phi}_4$}
\psfrag{p5}{$\ol{\phi}_5$}
\psfrag{p6}{$\ol{\phi}_6$}
\psfrag{p7}{$\ol{\phi}_7$}
\psfrag{y}{$y$}
\includegraphics[width=0.45\textwidth]{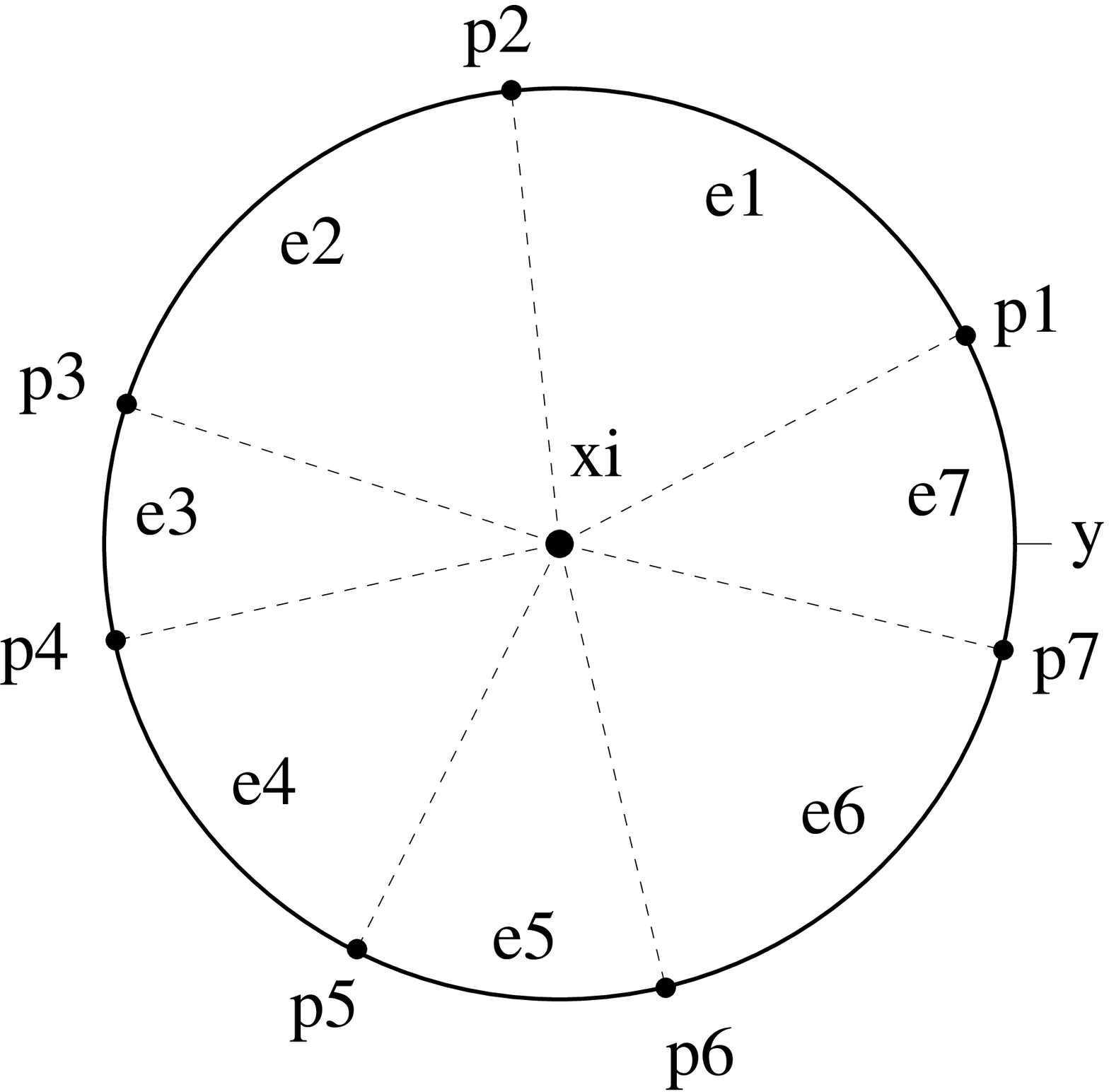}
\caption{A fiber of the bundle $L_i$ over a hole with $7$ sides.}
\label{fig:chern}
\end{figurehere}
\end{center}

The proof of the previous lemma is very easy.

%%%%%%%%%%%%%%%%%%%%%%%%%%%%%%%%%%%%%%%%%%%%%%%%%%%%%%%%%%%%%%%%%%%%%
\subsubsection{Orientation of Witten subcomplexes.}
%%%%%%%%%%%%%%%%%%%%%%%%%%%%%%%%%%%%%%%%%%%%%%%%%%%%%%%%%%%%%%%%%%%%%%

The following lemma says that the $\eta$ forms can be assembled
in a piecewise-linear ``symplectic form'', that can be used
to orient maximal cells of Witten subcomplexes.

\begin{lemma}[\cite{kontsevich:intersection}]\label{lemma:orient}
For every $\up\in\Delta_X\times\R_+$
the restriction of
\[
\ol{\Omega}:=\sum_{i=1}^n p_i^2 \ol{\eta}_i
\]
to the maximal simplices of $\Mbarcomb_{m_*,X}(\up)$
is a non-degenerate symplectic form. Hence,
$\ol{\Omega}^r$ defines an orientation on $\Mbarcomb_{m_*,X}(\up)$.
Also, $\ol{\Omega}^r\wedge\bar{\ell}_{\pa}^*\mathrm{Vol}_{\R^X}$
is a volume form on $\Mbarcomb_{m_*,X}$.
\end{lemma}
\begin{proof}
Let $|\ua|(\up)$ be a cell of $\Mbarcomb_{g,X}(\up)$, whose associated
ribbon graph $\GG_{\ua}$ has only vertices of odd valence.

On $|\ua|(\up)$,
the differentials $de_i$ span the cotangent space.
As the $p_i$'s are fixed, we have the
relation $dp_i=0$ for all $i=1,\dots,n$. Hence 
\[
T^{\vee}\Mbarcomb_{g,X}(\up)\Big|_{|\ua|(\up)}\cong |\ua|(\up)\times
\bigoplus_{e\in E_1(\ua)}\R\cdot de \Big/
\left(\sum_{[\ora{e}]_0=x_i} de\,\Big|\,i=1,\dots,n\right)
\]

On the other hand the tangent bundle is
\[
T\Mbarcomb_{g,X}(\up)\Big|_{|\ua|(\up)}\cong |\ua|(\up)
\times
\left\{
\sum_{e\in E_1(\ua)} b_e\frac{\pa}{\pa e}\ \Big|
\ \sum_{[\ora{e}]_0\in x_i}b_e=0
\quad\text{for all $i=1,\dots,n$}\,
\right\}.
\]
In order to prove that $\ol{\Omega}|_{\ua}:T|\ua|(\up)\lra
T^{\vee}|\ua|(\up)$
is non-degenerate,
we construct its right-inverse. Define $B:T^{\vee}|\ua|(\up)\lra
T|\ua|(\up)$ as
\[
B(de)=\sum_{i=1}^{2s} (-1)^i \frac{\pa}{\pa [\s_0^i(\ora{e})]_1}
+\sum_{j=1}^{2t} (-1)^j \frac{\pa}{\pa [\s_0^j(\ola{e})]_1}
\]
where $\vec{e}$ is any orientation of $e$, while $2s+1$ and $2t+1$
are the cardinalities of $[\ora{e}]_0$ and $[\ola{e}]_0$
respectively.
We want to prove that $\ol{\Omega}B(de)=4de$ for every $e\in E_1(\ua)$.

\begin{center}
{\large
\begin{figurehere}
\psfrag{f1}{$f_1$}
\psfrag{f2}{$f_2$}
\psfrag{f3}{$f_3$}
\psfrag{f4}{$f_4$}
\psfrag{h1}{$h_1$}
\psfrag{h2}{$h_2$}
\psfrag{H1}{$H_1$}
\psfrag{F1}{$F_1$}
\psfrag{F2}{$F_2$}
\psfrag{F3}{$F_3$}
\psfrag{E+}{$E_+$}
\psfrag{E-}{$E_-$}
\psfrag{e}{$\ora{e}$}
\includegraphics[width=0.6\textwidth]{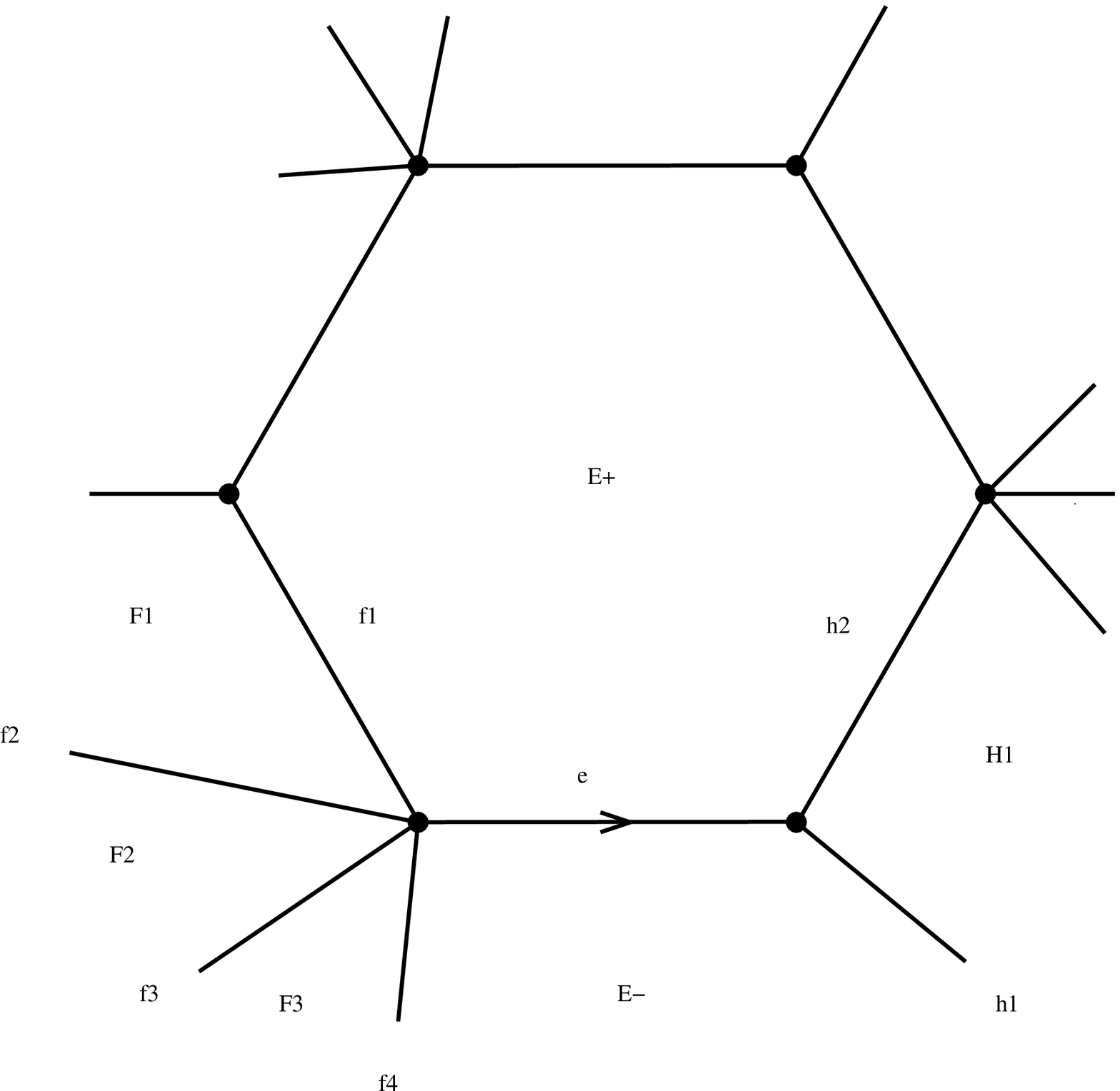}
\caption{An example with $s=2$ and $t=1$.}
\label{fig:orient}
\end{figurehere}
}
\end{center}

To shorten the notation, set
$f_i:=[\s_0^i(\ora{e})]_1$ and $h_j:=[\s_0^j(\ola{e})]_1$
and call $F_i:=[\s_0^i(\ora{e})]_{\infty}$ for $i=1,\dots,2s-1$ and
$H_j:=[\s_0^j(\ola{e})]_{\infty}$ for $j=1,\dots,2t-1$
the holes bordered
respectively by $\{f_i,f_{i+1}\}$ and $\{h_j,h_{j+1}\}$.
Finally call $E_+$ and $E_-$ the holes adjacent to $e$
as in Figure~\ref{fig:orient}.
%Consequently denote by $\ell_{_{F_i}}$ and $\ell_{_{H_j}}$ the lengths
%of the half-perimeters of the holes $F_i$ and $H_j$ respectively.
Remark that neither the edges $f$ and $h$ nor the holes $F$ and $H$
are necessarily distinct. This however has no importance in the
following computation.
\[
B(de)=-\sum_{i=1}^{2s}(-1)^i\frac{\pa}{\pa f_i}-
\sum_{j=1}^{2t}(-1)^j\frac{\pa}{\pa h_j}
\]
It is easy to see (using that the perimeters
are constant) that
\[
p_{F_i}^2\ol{\eta}_{{F_i}}
\left(
\frac{\pa}{\pa f_i}-\frac{\pa}{\pa f_{i+1}}
\right)
= df_i+df_{i+1}
\]
and analogously for the $h$'s.
Moreover
\[
p_{E_+}^2
\ol{\eta}_{{E_+}}
\left(
\frac{\pa}{\pa h_{2s}}-\frac{\pa}{\pa f_1}
\right)
=dh_{2s}+df_1+2de
\]
and similarly for $E_-$.
Finally, we obtain $\ol{\Omega}B(de)=4de$.
\end{proof}

\begin{remark}
Notice that $B$ is the piecewise-linear extension
of the Weil-Petersson bivector field $2\tilde{\omega}^{\vee}_\infty$.
Thus, $\ol{\Omega}$ is the piecewise-linear extension
of $2\tilde{\omega}_\infty$.
\end{remark}

Finally, we can show that the (cellular) chain obtained
by adding maximal simplices of Witten subcomplexes
(with the orientation determined by $\Omega$) is
in fact a cycle.

\begin{lemma}[\cite{kontsevich:intersection}]
With the given orientation $\Mbarcomb_{m_*,X}(\up)$ is a cycle
for all $\up\in\Delta_X\times \R_+$ and
$\Mbarcomb_{m_*,X}(\R_+^X)$ is
a cycle with non-compact support.
\end{lemma}
\begin{proof}
Given a top-dimensional cell $|\ua|(\up)$ in $\Mbarcomb_{m_*,X}(\up)$,
each face in the boundary $\pa|\ua|(\up)$
is obtained shrinking one edge of $\GG_{\ua}$.
This contraction may merge two vertices as in Fig.~\ref{fig:contraction}.

\begin{center}
{\large
\begin{figurehere}
\psfrag{e1}{{\color{Blue}$e_1$}}
\psfrag{e2}{{\color{Blue}$e_2$}}
\psfrag{e3}{{\color{Blue}$e_3$}}
\psfrag{e4}{{\color{Blue}$e_4$}}
\psfrag{e5}{{\color{Blue}$e_5$}}
\psfrag{e6}{{\color{Blue}$e_6$}}
\includegraphics[width=0.6\textwidth]{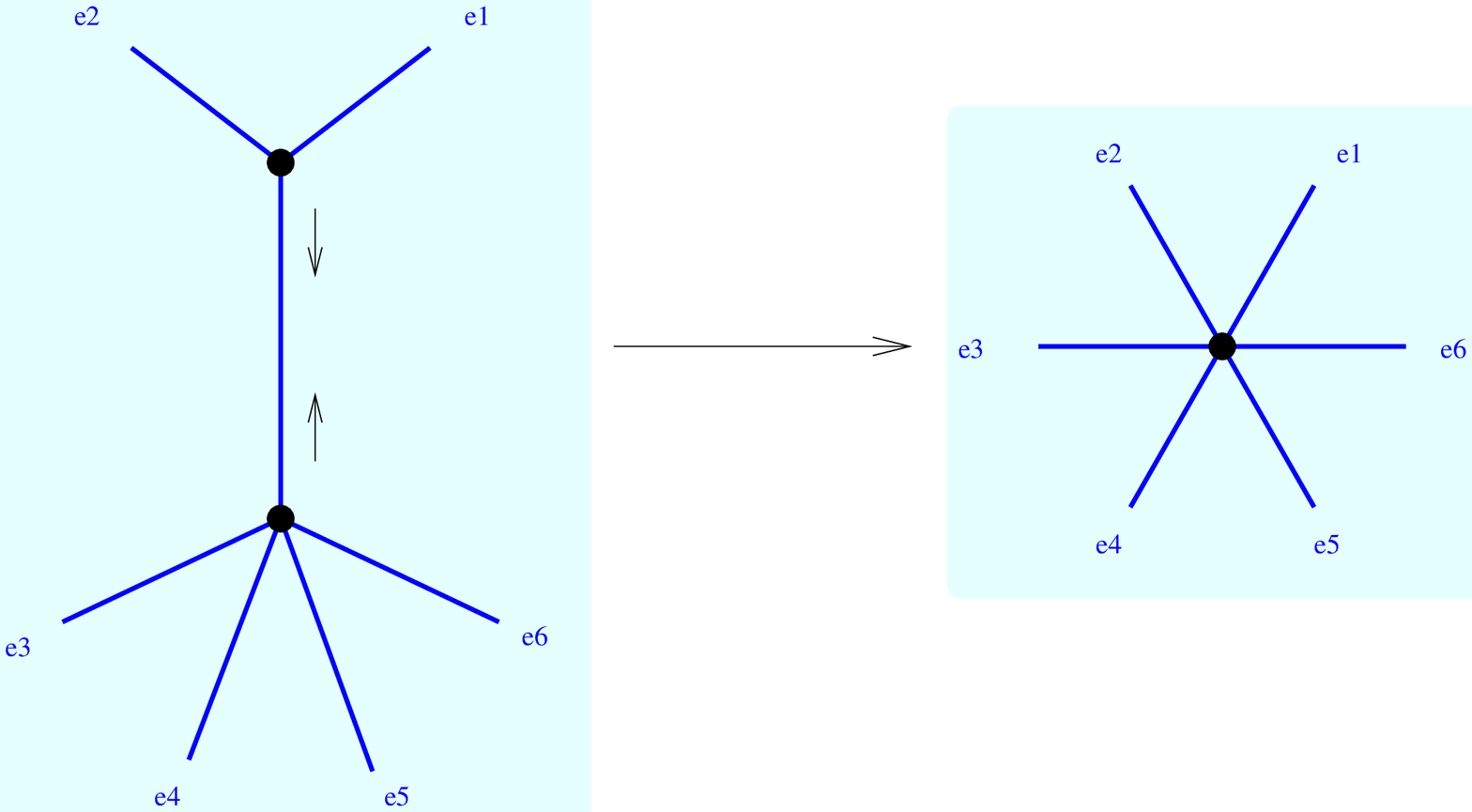}
\caption{A contraction that merges a $3$-valent and a $5$-valent vertex.}
\label{fig:contraction}
\end{figurehere}
}
\end{center}

Otherwise the shrinking produces a node,
as in Fig.~\ref{fig:contraction2}.

\begin{center}
{\large
\begin{figurehere}
\psfrag{e1}{{\color{Blue}$e_1$}}
\psfrag{e2}{{\color{Blue}$e_2$}}
\psfrag{e3}{{\color{Blue}$e_3$}}
\psfrag{e4}{{\color{Blue}$e_4$}}
\psfrag{e5}{{\color{Blue}$e_5$}}
\includegraphics[width=0.6\textwidth]{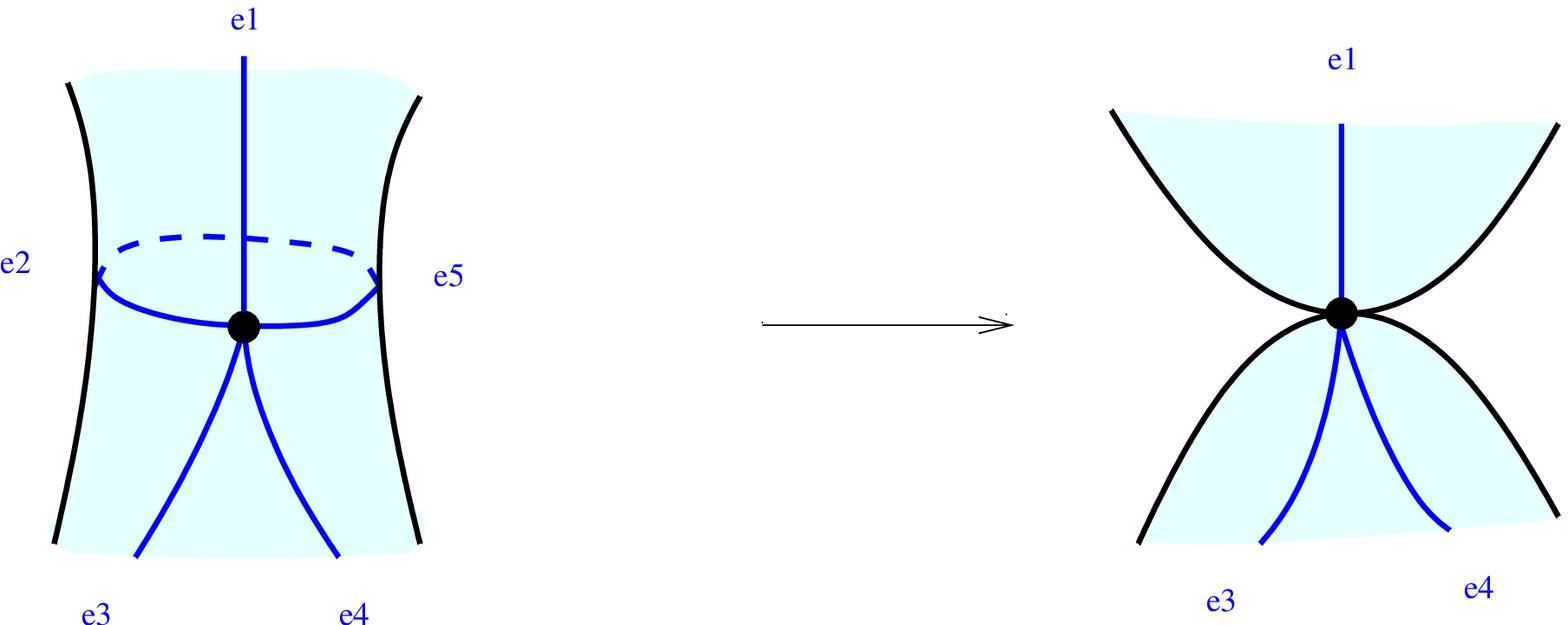}
\caption{A contraction produces a node.}
\label{fig:contraction2}
\end{figurehere}
}
\end{center}

Let $|\ua'|(\up)\in\pa |\ua|(\up)$
be the face of $|\ua|(\up)$ obtained by shrinking
the edge $e$. Then $\Lambda^{6g-7+2n-2r}T|\ua'|(\up)=
\Lambda^{6g-6+2n-2r}T|\ua|(\up)
\otimes N^{\vee}_{|\ua'|/|\ua|}$
and so the dual of the orientation form induced by
$|\ua|(\up)$ on $|\ua'|(\up)$
is $\i_{de}(B_{\ua}^{6g-6+2n-2r})=(6g-6+2n-2r)\i_{de}(B_{\ua})
\wedge B_{\ua}^{6g-8+2n-2r}$, where $B_{\ua}$ is the bivector field
on $|\ua|(\up)$ defined in Lemma~\ref{lemma:orient}.

Consider the graph $\GG_{\ua'}$ that occurs in the boundary
of a top-dimensional cell of $\Mbarcomb_{m_*,X}(\up)$.
Suppose it is obtained merging two
vertices of valences $2t_1+3$ and $2t_2+3$ in a vertex $v$
of valence $2(t_1+t_2)+4$. Then $|\ua'|(\up)$
is in the boundary of
exactly $2(t_1+t_2)+4$ cells of $\Mbarcomb_{m_*,X}(\up)$
or $t_1+t_2+2$ ones in the case $t_1=t_2$. In any case,
the number of cells $|\ua'|(\up)$ is bordered
by are even: we need to prove
that half of them induces on $|\ua'|(\up)$ an orientation and the
other half induces the opposite one. If $\GG_{\ua'}$ is obtained
from some $\GG_{\ua}$ contracting an edge $e$, then we just have
to compute the vector field $\i_{de}(B_{\ua})$, which turns
to be
\[
\i_{de}(B_{\ua})=\pm \sum_{i=1}^{2(t_1+t_2)+4} (-1)^i\frac{\pa}{\pa f_i}
\]
where $f_1,\dots,f_{2(t_1+t_2)+4}$ are the edges of $\GG_{\ua'}$
outgoing from $v$.
It is a straightforward computation to check that one obtains
in half the cases a plus and in half the cases a minus.

When $\GG^{en}_{\ua'}$ has a node with $2t_1+2$ edges on one side
(which we will denote by $f_1,\dots,f_{2t_1+2}$)
and $2t_2+3$ edges on the other side, the computation is
similar. The cell occurs as boundary of exactly
$(2t_1+2)(2t_2+3)$ top-dimensional cells and, if
$\GG_{\ua'}$ is obtained by $\GG_{\ua}$ contracting the edge $e$,
then
\[
\i_{de}(B_{\ua})=\pm 2\sum_{i=1}^{2t_1+2} (-1)^i\frac{\pa}{\pa f_i}.
\]
A quick check ensures that the signs cancel.
\end{proof}

Define the
{\it Witten classes}
$\Wbar_{m_*,X}(\up):=[\Mbarcomb_{m_*,X}(\up)]$ and
let $\W_{m_*,X}(\up)$ be its restriction to $\Mcomb_{g,X}(\up)$,
which defines (by Poincar\'e duality) a cohomology class
in $H^{2r}(\M_{g,X})$, independent of $\up$.

%
%%%%%%%%%%%%%%%%%%%%%%%%%%%%%%%%%%%%%%%%%%%%%%%%%%%%%%%%%%%%%%%%%%%%%%%%%
\subsubsection{Generalized Witten cycles.}
%%%%%%%%%%%%%%%%%%%%%%%%%%%%%%%%%%%%%%%%%%%%%%%%%%%%%%%%%%%%%%%%%%%%%%%%%

It is possible to define a slight generalization of
the previous classes, prescribing that some markings
hit vertices with assigned valence.

These {\it generalized Witten classes} are related to
the previous $\W_{m_*,X}$ in an intuitively obvious way,
because forgetting the markings of some vertices will
map them onto one another.
We will omit the details and refer to \cite{mondello:combinatorial}.

%%%%%%%%%%%%%%%%%%%%%%%%%%%%%%%%%%%%%%%%%%%%%%%%%%%%%%%%%%%%%%%%%%%%%%%%%
%%%%%%%%%%%%%%%%%%%%%%%%%%%%%%%%%%%%%%%%%%%%%%%%%%%%%%%%%%%%%%%%%%%%%%%%%
\subsection{Witten cycles and tautological
classes}\label{ss:witten-tautological}
%%%%%%%%%%%%%%%%%%%%%%%%%%%%%%%%%%%%%%%%%%%%%%%%%%%%%%%%%%%%%%%%%%%%%%%%%

In this subsection, we will sketch the proof of the following
result, due to K.~Igusa \cite{igusa:mmm} and \cite{igusa:kontsevich}
(see also \cite{igusa-kleber})
and Mondello \cite{mondello:combinatorial} independently.

\begin{theorem}\label{thm:tautological}
Witten cycles $\W_{m_*,X}$ on $\M_{g,X}$
are Poincar\'e dual to polynomials in the $\kappa$
classes and vice versa.
\end{theorem}

In \cite{mondello:combinatorial}, the following results
are also proven:
\begin{itemize}
\item
Witten generalized cycles on $\M_{g,X}$
are Poincar\'e dual to polynomials in the $\psi$
and the $\kappa$ classes
\item
ordinary and generalized Witten cycles
on $\Mbarcomb_{g,X}(\up)$ are push-forward
of (the Poincar\'e dual of)
tautological classes from $\Mbar_{g,X}$;
an explicit recipe to produce such tautological
classes is given.
\end{itemize}

%%%%%%%%%%%%%%%%%%%%%%%%%%%%%%%%%%%%%%%%%%%%%%%%%%%%%%%%%%%%%%%%%%%%%%%%%%%
\subsubsection{The case with one special vertex.}
%%%%%%%%%%%%%%%%%%%%%%%%%%%%%%%%%%%%%%%%%%%%%%%%%%%%%%%%%%%%%%%%%%%%%%%%%%%

We want to consider a combinatorial cycle
on $\M_{g,X}$
supported
on ribbon graphs, whose vertices are generically all trivalent
except one, which is $(2r+3)$-valent (and $r\geq 1$).
To shorten the notation,
call this Witten cycle $\W_{2r+3}$.

We also define a generalized Witten cycle on the universal
curve $\Cc_{g,X}\subset \Mbar_{g,X\cup\{y\}}$ supported on the locus
of ribbon graphs, which have a $(2r+3)$-valent vertex
marked by $y$ and all the other vertices are trivalent
and unmarked. Call $\W^y_{2r+3}$ this cycle.

We would like to show that $\mathrm{PD}(\W^y_{2r+3})=c(r)\psi_y^{r+1}$,
where $c(r)$ is some constant. As a consequence,
pushing the two hand-sides down through the proper map
$\pi_y:\Cc_{g,X}\rar \M_{g,X}$, we would obtain
$\mathrm{PD}(\W_{2r+3})=c(r)\k_r$.

Lemma~\ref{lemma:eta} gives us the nice piecewise-linear
$2$-form $\eta_y$, that is pulled back to $\psi_y$
through $\xi$. The only problem is that $\eta_y$ is
defined only for $p_y>0$, whereas $\W^y_{2r+3}$ is
exactly contained in the locus $\{p_y=0\}$.

To compare the two, one can look at the
blow-up $\mathrm{Bl}_{p_y=0}\Mbarcomb_{g,X\cup\{y\}}$
of $\Mbarcomb_{g,X\cup\{y\}}$ along
the locus $\{p_y=0\}$. Points in the exceptional
locus $E$ can be
identified with metrized (nonsingular)
ribbon graphs $\GG$, in which $y$ marks a vertex,
plus {\it angles} $\vartheta$ between consecutive
oriented edges outgoing from $y$.
One must think of these angles as of
{\it infinitesimal edges}.

It is clear now that $\eta_y$ extends to
$E$ by
\[
\eta_y|_{|\ua|(\up)}:=\sum_{1\leq s<t\leq k-1}
d\tilde{e}_s \wedge d\tilde{e}_t
\]
where $\dis \tilde{e}_j=\frac{\vartheta_j}{2\pi}$,
$y$ marks a vertex with cyclically ordered
outgoing edges
$(\ora{e}_1,\dots,\ora{e}_k)$
and $\vartheta_j$ is the angle between
$\ora{e}_j$ and $\ora{e}_{j+1}$ (with $j\in\Z/k\Z$).

Thus, pushing forward $\eta_y^{r+1}$
through $E \rar \Mbarcomb_{g,X}(p_y=0)$,
we obtain $c(r)\ol{W}^y_{2r+3}$ plus other
terms contained in the boundary, and the coefficient
$c(r)$ is exactly the integral of $\eta_y^{r+1}$
on a fiber (that is, a simplex), which turns
out to be $\dis c(r)=\frac{(r+1)!}{(2r+2)!}$.
Thus, $\W^y_{2r+3}$ is Poincar\'e dual to
$2^{r+1}(2r+1)!!\psi_y^{r+1}$.

%%%%%%%%%%%%%%%%%%%%%%%%%%%%%%%%%%%%%%%%%%%%%%%%%%%%%%%%%%%%%%%%%%%%%%%%%%%
\subsubsection{The case with many special vertices.}
%%%%%%%%%%%%%%%%%%%%%%%%%%%%%%%%%%%%%%%%%%%%%%%%%%%%%%%%%%%%%%%%%%%%%%%%%%%

To mimic what done for one non-trivalent vertex,
let's consider combinatorial classes
with two non-trivalent vertices.
Thus, we look at the class $\psi_y^{r+1}\psi_z^{s+1}$
(with $r,s\geq 1$)
on $\Cc^2_{g,X}:=\Cc_{g,X}\times_{_{\M_{g,X}}} \Cc_{g,X}$.

Look at the blow-up $\mathrm{Bl}_{p_y=0,p_z=0}\Mbarcomb_{g,X\cup\{y\}}$
of $\Mbarcomb_{g,X\cup\{y,z\}}$ along
the locus $\{p_y=0\}\cup\{p_z=0\}$
and let $E=E_y\cap E_z$, where $E_y$ and $E_z$
are the exceptional loci.

As before, we can identify $E\cap\{y\neq z\}$ with
the set of metrized ribbon graphs $\GG$,
with angles at the vertices $y$ and $z$. Thus, pushing
$\eta_y^{r+1}\eta_z^{s+1}$ forward through
the blow-up map (which forgets the angles
at $y$ and $z$), we obtain
a multiple of the generalized combinatorial
cycles given by $y$ marking a $(2r+3)$-valent
vertex and $z$ marking a $(2s+3)$-valent
(distinct) vertex. The coefficient $c(r,s)$ will just
be $\dis \frac{(r+1)!(s+1)!}{(2r+2)!(2s+2)!}$.

Points in $E\cap\{y=z\}$
can be thought of as metrized ribbon graphs
$\GG$ with two infinitesimal holes (respectively
marked by $y$ and $z$) adjacent to each other.
If we perform the push-forward of
$\eta_y^{r+1}\eta_z^{s+1}$ forgetting first the
angles at $z$ and then the angles at $y$, then we obtain
some contribution only from the loci
in which the infinitesimal $z$-hole 
has $(2s+3)$ edges and the infinitesimal
$y$-hole has $(2r+4)$ edges (included
the common one). Thus, we obtain the
same contribution for each of
the $b(r,s)$ configurations
of two adjacent holes of valences $(2s+3)$
and $(2r+4)$.

Thus, we obtain a cycle
supported on the locus of metrized
ribbon graphs $\GG$ in which $y=z$ marks
a $(2r+2s+3)$-valent vertex, with
coefficient $b(r,s)c(r,s)$.

Hence, $\psi_y^{r+1}\psi_z^{s+1}$ is Poincar\'e
dual to a linear combination of generalized
combinatorial cycles. As before, using the
forgetful map, the same
holds for the Witten cycles obtained by
deleting the $y$ and the $z$ markings.

One can easily see that the transformation
laws from $\psi$ classes to combinatorial classes
are invertible (because they are ``upper
triangular'' in a suitable sense).

Clearly, in order to deal with many $\psi$
classes (that is, with many non-trivalent marked
vertices), one must compute more and more complicated
combinatorial factors like $b(r,s)$.

We refer to \cite{igusa-kleber} and \cite{mondello:combinatorial}
for two (complementary) methods to calculate these factors.

%%%%%%%%%%%%%%%%%%%%%%%%%%%%%%%%%%%%%%%%%%%%%%%%%%%%%%%%%%%%%%%%%%%%%%%%%
%%%%%%%%%%%%%%%%%%%%%%%%%%%%%%%%%%%%%%%%%%%%%%%%%%%%%%%%%%%%%%%%%%%%%%%%%
\subsection{Stability of Witten cycles}\label{ss:stability}
%%%%%%%%%%%%%%%%%%%%%%%%%%%%%%%%%%%%%%%%%%%%%%%%%%%%%%%%%%%%%%%%%%%%%%%%%
\subsubsection{Harer's stability theorem.}
%%%%%%%%%%%%%%%%%%%%%%%%%%%%%%%%%%%%%%%%%%%%%%%%%%%%%%%%%%%%%%%%%%%%%%%%%

The (co)homologies of the mapping class groups have the remarkable
property that they stabilize when the genus of the surface increases.
This was first proven by Harer \cite{harer:stability}, and the
stability bound was then improved by Ivanov \cite{Ivanov:improved}
(and successively again by Harer for homology with rational
coefficients, in an unpublished paper). We now want to
recall some of Harer's results.

Let $S_{g,n,b}$ be a compact oriented
surface of genus $g$ with $n$ marked points
and $b$ boundary components $C_1,\dots,C_b$.
Call $\G(S_{g,n,b})$ the group of isotopy
classes of diffeomorphisms of $S$ that fix the marked
points and
$\pa S$ pointwise.

Call also $P=S_{0,0,3}$ a fixed pair of pants and denote by
$B_1,B_2,B_3$ its boundary components.

Consider the following two operations:
\begin{itemize}
\item[(y)] 
gluing $S_{g,n,b}$ and $P$ by identifying $C_b$ with $B_1$,
thus
producing an oriented surface of genus $g$ with $n$ marked
points and $b+1$ boundary components
\item[(v)]
identify $C_{b-1}$ with $C_b$ of $S_{g,n,b}$,
thus producing an oriented
surface of genus $g+1$ with $n$ marked points
and $b-2$ boundary components.
\end{itemize}

Clearly, they induce homomorphism at the level of mapping class
groups
\[
\mathcal{Y}: \G(S_{g,n,b})\lra \G(S_{g,n,b+1})
\]
when $b\geq 1$
(by extending the diffeomorphism as the identity on $P$)
and
\[
\mathcal{V}:\G(S_{g,n,b})\lra \G(S_{g+1,n,b-2})
\]
when $b\geq 2$.

\begin{theorem}[Harer \cite{harer:stability}]
The induced maps in homology
\begin{align*}
\mathcal{Y}_{*} & : H_k(\G(S_{g,n,b})) \lra H_k(\G(S_{g,n,b+1}))\\
\mathcal{V}_{*} & : H_k(\G(S_{g,n,b})) \lra H_k(\G(S_{g+1,n,b-2}))
\end{align*}
are isomorphisms for $g\geq 3k$.
\end{theorem}
The exact bound is not important for our purposes.
We only want to stress that the theorem implies
that $H_k(\G(S_{g,n,b}))$ stabilizes for large $g$.
In particular, fixed $n\geq 0$,
the rational homology of $\M_{g,n}$
stabilizes for large $g$.

\begin{remark}
We have $B\G(S_{g,n,b})\simeq\M_{g,X,T}$,
where $\M_{g,X,T}$ is the moduli space of Riemann surfaces
of genus $g$ with $X\cup T$ marked points
($X=\{x_1,\dots,x_n\}$ and $T=\{t_1,\dots,t_b\}$)
and a nonzero tangent vector at each point of $T$.
If $b\geq 1$, then $\M_{g,X,T}$ is a smooth variety: in fact,
an automorphism of a Riemann surface that fixes a point
and a tangent direction at that point is the identity
(it follows from uniformization and Schwarz lemma).
\end{remark}

%%%%%%%%%%%%%%%%%%%%%%%%%%%%%%%%%%%%%%%%%%%%%%%%%%%%%%%%%%%%%%%%%%
\subsubsection{Mumford's conjecture.}
%%%%%%%%%%%%%%%%%%%%%%%%%%%%%%%%%%%%%%%%%%%%%%%%%%%%%%%%%%%%%%%%%%

Call $\dis \G_{\infty,n}=\lim_{g\rar\infty} \G(S_{g,n,1})$, where
the map $\G(S_{g,n,1})\rar \G(S_{g+1,n,1})$ corresponds
to gluing a torus with two holes at the boundary
component of $S_{g,n,1}$.

Then, $H^k(\G_{\infty,n})$ coincides with $H^k(\G_{g,n})$
for $g\gg k$.

Mumford conjectured that
$H^*(\G_\infty;\Q)$
is the polynomial algebra on the $\k$ classes.
Miller \cite{miller:homology}
showed that $H^*(\G_\infty;\Q)$ is a Hopf
algebra that contains $\Q[\k_1,\k_2,\dots]$.

Recently, after works of Tillmann
(for instance, \cite{tillmann:infinite}) and
Madsen-Tillmann \cite{madsen-tillmann},
Madsen and Weiss \cite{madsen-weiss:mumford}
proved a much stronger statement of homotopy
theory, which in particular implies Mumford's
conjecture.

Thanks to a result of B\"odigheimer-Tillmann
\cite{boedigheimer-tillmann:stripping}, it follows
that $H^*(\G_{\infty,n};\Q)$ is a polynomial
algebra on $\psi_1,\dots,\psi_n$ and the
$\k$ classes.

Thus, generalized Witten classes, being polynomials in
$\psi$ and $\k$, are also stable.
In what follows, we would like to prove this
stability in a direct way.

%%%%%%%%%%%%%%%%%%%%%%%%%%%%%%%%%%%%%%%%%%%%%%%%%%%%%%%%%%%%%%%%%%
\subsubsection{Ribbon graphs with tails.}
%%%%%%%%%%%%%%%%%%%%%%%%%%%%%%%%%%%%%%%%%%%%%%%%%%%%%%%%%%%%%%%%%%

One way to cellularize the moduli space of curves with
marked points and tangent vectors at the marked points
is to use ribbon graphs with tails (see, for instance,
\cite{godin:unstable}).

Consider $\Si$ a compact Riemann surface of genus $g$
with marked points $X\cup T=\{x_1,\dots,x_n\}\cup\{t_1,\dots,t_b\}$
and nonzero tangent vectors $v_1,\dots,v_b$ at $t_1,\dots,t_b$.

Given $p_1,\dots,p_n\geq 0$ and $q_1,\dots,q_b>0$,
we can construct the ribbon graph $\GG$ associated to
$(\Si,\up,\ul{q})$, say using the
Jenkins-Strebel differential $\varphi$.

For every $j=1,\dots,b$, move from the center $t_j$
along a vertical trajectory $\g_j$ of $\varphi$
determined by the
tangent vector $v_j$, until we hit the critical graph.
Parametrize the opposite path $\g_j^*$ by arc-length,
so that $\g_j^*:[0,\infty]\rar \Si$, $\g_j^*(0)$
lies on the critical graph and $\g_j^*(\infty)=t_j$.
Then, construct a new ribbon
graph out of $\GG$ by ``adding''
a new vertex (which we will call $\tilde{v}_j$)
and a new edge $e_{v_j}$ of length $|v_j|$
(a {\it tail}), whose realization is
$\g_j^*([0,|v_j|])$ (see Figure~\ref{fig:tail}).

\begin{center}
{\large
\begin{figurehere}
\psfrag{x}{$t_j$}
\psfrag{v}{{\color{Plum}$\tilde{v}_j$}}
\psfrag{ev}{{\color{Plum}$e_{v_j}$}}
\includegraphics[width=0.6\textwidth]{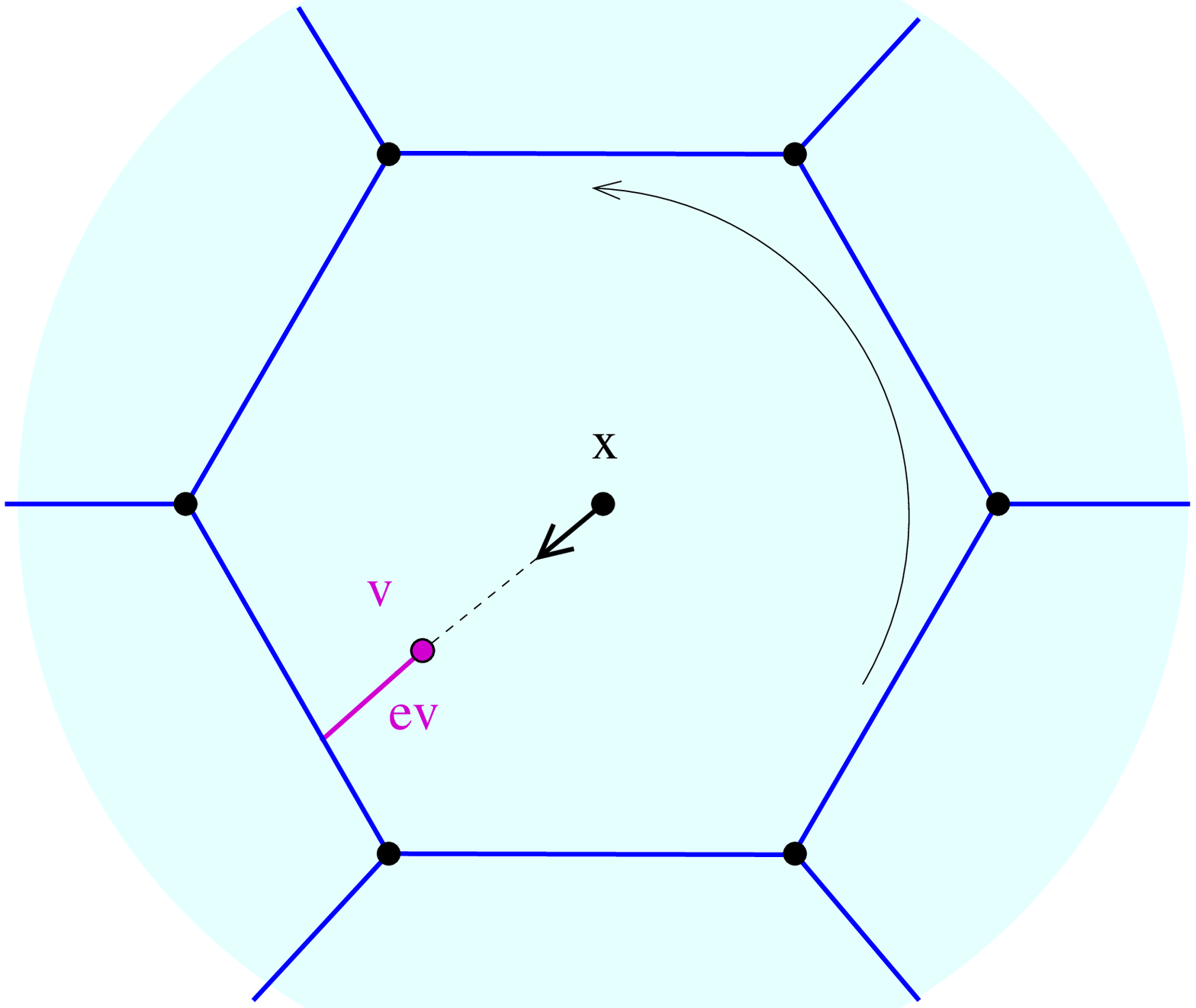}
\caption{Correspondence between a tail and a nonzero tangent vector.}
\label{fig:tail}
\end{figurehere}
}
\end{center}

Thus, we have realized an embedding of
$\M_{g,X,T}\times \R_{\geq 0}^X\times\Delta_T\times\R_+$
inside $\Mcomb_{g,X\cup T\cup V}$, where $V=\{\tilde{v}_1,\dots,
\tilde{v}_b\}$.
If we call $\Mcomb_{g,X,T}$ its image, we have obtained the following.

\begin{lemma}
$\Mcomb_{g,X,T}\simeq B\G(S_{g,n,b})$.
\end{lemma}

Notice that the embedding $\Mcomb_{g,X,T}\hra \Mcomb_{g,X\cup T\cup V}$
allows us to define (generalized) Witten cycles
$W_{m_*,X,T}$ on $\Mcomb_{g,X,T}$ simply by restriction.

%%%%%%%%%%%%%%%%%%%%%%%%%%%%%%%%%%%%%%%%%%%%%%%%%%%%%%%%%%%%%%%%%%
\subsubsection{Gluing ribbon graphs with tails.}
%%%%%%%%%%%%%%%%%%%%%%%%%%%%%%%%%%%%%%%%%%%%%%%%%%%%%%%%%%%%%%%%%%

Let $\GG'$ and $\GG''$ be two ribbon graphs with tails
$\ora{e'}$ and $\ora{e''}$,
i.e. $\ora{e'}\in E(\GG')$ and $\ora{e''}\in E(\GG'')$
with the property that
$\s'_0(\ora{e'})=\ora{e'}$ and $\s''_0(\ora{e''})=\ora{e''}$.

We produce a third ribbon graph $\GG$,
obtained by {\it gluing $\GG'$ and $\GG''$}
in the following way.

We set $E(\GG)=\left(E(\GG')\cup E(\GG'')\right)/\sim$,
where we declare that $\ora{e'}\sim \ola{e''}$ and
$\ola{e'}\sim\ora{e''}$. Thus, we have a natural $\s_1$
induced on $E(\GG)$. Moreover, we define $\s_0$
acting on $E(\GG)$ as
\[
\s_0([\ora{e}])=
\begin{cases}
[\s'_0(\ora{e})] & \text{if $\ora{e} \in E(\GG')$
and $\ora{e}\neq\ora{e'}$}\\ 
[\s''_0(\ora{e})] & \text{if $\ora{e} \in E(\GG'')$
and $\ora{e}\neq\ora{e''}$}
\end{cases}
\]
If $\GG'$ and $\GG''$ are metrized, then we induce
a metric on $\GG$ in a canonical way, declaring
the length of the new edge of $\GG$ to be $\ell(e')+\ell(e'')$.

Suppose that $\GG'$ is marked by $\{x_1,\dots,x_n,t'\}$
and $e'$ is a tail contained in the hole $t'$
and that
$\GG''$ is marked by $\{y_1,\dots,y_m,t''\}$ and
if $e''$ is a tail contained in the hole $t''$,
then $\GG$ is marked by $\{x_1,\dots,x_n,y_1,\dots,y_m,t\}$,
where $t$ is a new hole obtained {\it merging}
the holes centered at $t'$ and $t''$.

Thus, we have constructed a {\it combinatorial gluing map}
\[
\Mcomb_{g',X',T'\cup\{t'\}}\times\Mcomb_{g'',X'',T''\cup\{t''\}}
\lra \Mcomb_{g'+g'',X'\cup X''\cup\{t\},T'\cup T''}
\]

%%%%%%%%%%%%%%%%%%%%%%%%%%%%%%%%%%%%%%%%%%%%%%%%%%%%%%%%%%%%%%%%%%
\subsubsection{The combinatorial stabilization maps.}
%%%%%%%%%%%%%%%%%%%%%%%%%%%%%%%%%%%%%%%%%%%%%%%%%%%%%%%%%%%%%%%%%%

Consider the gluing maps
in two special cases which are slightly different
from what we have seen before.

Call $S_{g,X,T}$ a compact oriented surface
of genus $g$ with boundary components labeled by $T$
and marked points labeled by $X$.

Fix a trivalent ribbon graph $\GG_j$,
with genus $1$, one hole and $j$ tails
for $j=1,2$
(for instance, $j=2$ in Figure~\ref{fig:torus}).

\begin{center}
{\large
\begin{figurehere}
\psfrag{S}{$\GG_2$}
\psfrag{x}{$y$}
\psfrag{v}{{\color{Plum}$v$}}
\psfrag{w}{{\color{Plum}$w$}}
\includegraphics[width=0.6\textwidth]{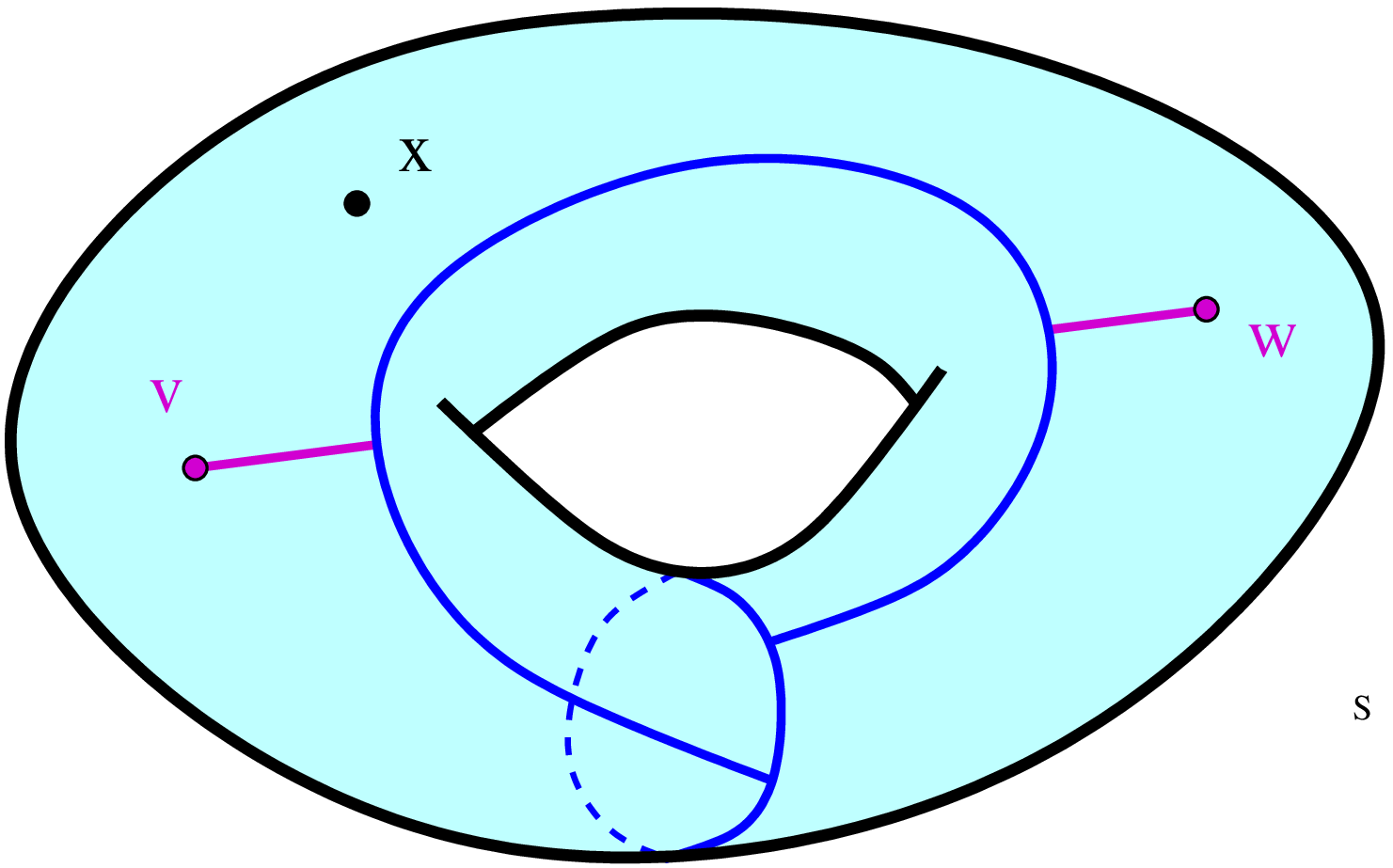}
\caption{Example of a fixed torus.}
\label{fig:torus}
\end{figurehere}
}
\end{center}

Consider the combinatorial gluing maps
\begin{align*}
\mathcal{S}_1^{comb} & :\Mcomb_{g,X,\{t\}} \lra \Mcomb_{g+1,X\cup\{t\}} \\
\mathcal{S}_2^{comb} & :\Mcomb_{g,X,\{t\}} \lra \Mcomb_{g+1,X,\{t\}}
\end{align*}
where $\mathcal{S}_j^{comb}$ is obtained
by simply gluing a graph $\GG$ in $\Mcomb_{g,X,\{t\}}$
with the fixed graph $\GG_j$, identifying the
unique tail of $\Mcomb_{g,X,\{t\}}$ with the $v$-tail
of $\GG_j$ and renaming the new hole by $t$.

It is easy to see that $\mathcal{S}^{comb}_2$
incarnates a stabilization map (obtained
by composing twice $\mathcal{Y}$
and once $\mathcal{V}$).

On the other hand, 
consider the map $\mathcal{S}_1:B\G(S_{g,X,\{t\}})
\rar B\G(S_{g+1,X\cup\{t\}})$, that glues
a torus
$S_{1,\{y\},\{t'\}}$
with one puncture and one boundary component to
the unique boundary component of $S_{g,X,\{t\}}$,
by identifying $t$ and $t'$, and
relabels the $y$-puncture by $t$.

The composition of $\mathcal{S}_1$
followed by the map $\pi_t$ that forgets the $t$-marking
\[
B\G(S_{g,X,\{t\}})\arr{\mathcal{S}_1}{\lra} B\G(S_{g+1,X\cup\{t\}})
\arr{\pi_t}{\lra} B\G(S_{g+1,X})
\]
induces an isomorphism on $H_k$ for $k\gg g$,
because it can be also obtained
composing $\mathcal{Y}$ and $\mathcal{V}$.

Notice that $\pi_t:B\G(S_{g+1,X\cup\{t\}})\rar B\G(S_{g+1,X})$
can be realized as a {\it combinatorial forgetful map}
$\pi_t^{comb}:\Mcomb_{g+1,X\cup\{t\}}(\R_+^X\times\{0\})\rar
\Mcomb_{g+1,X}(\R_+^X)$ in the following way.

Let $\GG$ be a metrized ribbon graph in $\Mcomb_{g+1,X\cup\{t\}}
(\R_+^X\times\{0\})$. If $t$ is marking a vertex of
valence $3$ or more, than just forget the $t$-marking.
If $t$ is marking a vertex of valence $2$, then
forget the $t$ marking and merge the two edges outgoing
from $t$ in one new edge.
Finally, if $t$ is marking a univalent
vertex of $\GG$ lying on an edge $e$, then
replace $\GG$ by $\GG/e$ and forget the $t$-marking.

%%%%%%%%%%%%%%%%%%%%%%%%%%%%%%%%%%%%%%%%%%%%%%%%%%%%%%%%%%%%%%%%%%
\subsubsection{Behavior of Witten cycles.}
%%%%%%%%%%%%%%%%%%%%%%%%%%%%%%%%%%%%%%%%%%%%%%%%%%%%%%%%%%%%%%%%%%

The induced homomorphism on Borel-Moore homology
\[
(\pi_t^{comb})^*:H^{BM}_*(\Mcomb_{g+1,X}(\R_+^X))\lra
H^{BM}_*(\Mcomb_{g+1,X\cup\{t\}})(\R_+^X\times\{0\})
\]
pulls $W_{m_*,X}$ back to the combinatorial
class $W^t_{m_*+\delta_0,X}$, corresponding to
(the closure of the locus of)
ribbon graphs with one univalent vertex
marked by $t$ and
$m_i+\delta_{0,i}$ vertices of
valence $(2i+3)$ for all $i\geq 0$.

We use now the fact that, for $X$ nonempty,
there is a homotopy equivalence
\[
E:\Mcomb_{g+1,X\cup\{t\}}(\R_+^X\times\R_+)
\arr{\sim}{\lra}
\Mcomb_{g+1,X\cup\{t\}}(\R_+^X\times \{0\})
\]
and that $E^*(W^t_{m_*+\delta_0,X})=W_{m_*+2\delta_0,X\cup\{t\}}$.

This last phenomenon can be understood
by simply observing
that $E^{-1}$
corresponds to opening the (generically univalent)
$t$-marked vertex to a small $t$-marked hole,
thus producing an extra trivalent vertex.

Finally, $(\mathcal{S}_1^{comb})^*(W_{m_*+2\delta_0,X\cup\{t\}})=
W_{m_*-\delta_0,X,\{t\}}$, because $\GG_1$ has exactly
$3$ trivalent vertices.

As a consequence, we have obtained that
\[
(\pi_t^{comb}\circ E\circ \mathcal{S}_1^{comb})^*:
H^{BM}_*(\Mcomb_{g+1,X}(\R_+^X)) \lra
H^{BM}_*(\Mcomb_{g,X,\{t\}}(\R_+^X\times\R_+))
\]
is an isomorphism for $g\gg *$ and
pulls $W_{m_*,X}$ back to $W_{m_*-\delta_0,X,\{t\}}$.\\

The other gluing map is much simpler: the induced
\[
(\mathcal{S}^{comb}_{2})^*:
H^{BM}_*(\Mcomb_{g+1,X,\{t\}}(\R_+^X\times\R_+))\lra
H^{BM}_*(\Mcomb_{g,X,\{t\}}(\R_+^X\times\R_+))
\]
carries $W_{m_*,X,\{t\}}$ to $W_{m_*-4\delta_0,X,\{t\}}$,
because $\GG_2$ has $4$ trivalent vertices.\\

We recall that a class in $H^k(\G_{\infty,X})$
(i.e. a stable class)
is a sequence of classes $\{\b_g \in H^k(\M_{g,X})\,|\,g\geq g_0\}$,
which are compatible with the stabilization maps,
and that two sequences are equivalent (i.e. they
represent the same stable class) if they are
equal for large $g$.

\begin{proposition}\label{prop:witten-stable}
Let $m_*=(m_0,m_1,\dots)$ be a sequence of nonnegative
integers such that $m_N=0$ for large $N$
and let $|X|=n>0$.
Define
\[
c(g)=4g-4+2n-\sum_{j\geq 1}(2j+1)m_j
\]
and call $g_0=\mathrm{inf}\{g\in \N\,|\,c(g)\geq 0\}$.
Then, the collection\\
\mbox{$\{ W_{m_*+c(g)\delta_0,X} \in H^{2k}(\M_{g,X})
\,|\, g\geq g_0\}$}
is a stable class, where $k=\sum_{j>0} j\,m_j$.
\end{proposition}

It is clear that an analogous statement can be proven for
generalized Witten cycles.
Notice that Proposition~\ref{prop:witten-stable}
implies Miller's result \cite{miller:homology}
that $\psi$ and $\k$ classes are stable.

%%%%%%%%%%%%%%%%%%%%%%%%%%%%%%%%%%%%%%%%%%%%%%%%%%%%%%%%%%%%%%%%%%%%%%%%%%%
%\bibliographystyle{amsalpha}
%\bibliography{Boer-bib}
%\end{document}
%
%%%%%%%%%%%%%%%%%%%%%%%%%%%%%%%%%%%%%%%%%%%%%%%%%%%%%%%%%%%%%%%%%%%%%%%%%%%

\end{document}